\documentclass[12pt,a4paper]{article}
\pdfoutput=1

\usepackage{amssymb,amsfonts,amsmath,graphicx,mathtools}
\usepackage[alphabetic,y2k,lite]{amsrefs}
\usepackage{fullpage}
\usepackage{MnSymbol}
\usepackage[mathletters]{ucs}
\usepackage[utf8x]{inputenc} %conflicts with hyperref on overleaf
\usepackage{tikz-cd}
\usepackage{amsthm}

\newtheorem*{theorem*}{Theorem}
\newtheorem{theorem}{Theorem}[section]
\newtheorem{lemma}[theorem]{Lemma}
\newtheorem{proposition}[theorem]{Proposition}
\newtheorem{corollary}[theorem]{Corollary}

\theoremstyle{definition}
\newtheorem{definition}[theorem]{Definition}

\newtheorem{example}[theorem]{Example}

\theoremstyle{remark}
\newtheorem{remark}[theorem]{Remark}

\numberwithin{equation}{section}

\newcommand{\comment}[1]{}

\DeclareMathOperator{\htr}{hTr} % horizontal trace
\DeclareMathOperator{\im}{Im}
\DeclareMathOperator{\Obj}{Obj}

\DeclareMathOperator{\Vect}{Vec}
\DeclareMathOperator{\Proj}{Proj}
\newcommand{\RRmod}[1]{{#1}\text{-mod}}
\newcommand{\Rmod}{\RRmod{R}}
\DeclareMathOperator{\smod}{smod}
\DeclareMathOperator{\Col}{Col}
\DeclareMathOperator{\Span}{span}

\DeclareMathOperator{\Ann}{Ann}
\DeclareMathOperator{\Hom}{Hom}
\DeclareMathOperator{\End}{End}
\DeclareMathOperator{\Res}{Res}
\DeclareMathOperator{\Colim}{colim}
\DeclareMathOperator{\Ima}{Im}

\DeclareMathOperator{\coker}{coker}

\newcommand{\HomT}{\wdtld{\Hom}}

\DeclareMathOperator{\ptr}{ptr}

\DeclareMathOperator{\id}{id}
\newcommand{\ov}{\overline}
\newcommand{\wdtld}{\widetilde}
\newcommand{\trv}{{\mathrel{\cap\!\negthickspace\vert}\;}}
\newcommand{\trvs}{{\trv, \cT}}
\DeclareMathOperator{\Int}{Int}
\newcommand{\op}{{\text{op}}}
\newcommand{\hatbox}{{\hat{\times}}}

\newcommand{\inv}{{-1}} % annoying to type {-1} when taking inverse
\newcommand{\one}{\mathbf{1}}

\newcommand{\lact}{\triangleright}
\newcommand{\ract}{\triangleleft}

\newcommand{\eval}[1]{\langle #1 \rangle}

\newcommand{\rcls}[1]{{#1}^{⊕}}
\newcommand{\tcls}[1]{\eval{#1}}

\DeclareMathOperator{\HH}{HH}

\DeclareMathOperator{\ev}{ev}
\DeclareMathOperator{\coev}{coev}

\newcommand{\cA}{{\mathcal{A}}}
\newcommand{\cB}{{\mathcal{B}}}
\newcommand{\cC}{{\mathcal{C}}}

\newcommand{\cI}{{\mathcal{I}}}

\newcommand{\cK}{{\mathcal{K}}}

\newcommand{\cM}{{\mathcal{M}}}
\newcommand{\cN}{{\mathcal{N}}}

\newcommand{\cS}{{\mathcal{S}}}
\newcommand{\cT}{{\mathcal{T}}}
\newcommand{\cU}{{\mathcal{U}}}

\newcommand{\cZ}{{\mathcal{Z}}}

\newcommand{\kk}{{\mathbf{k}}}

\newcommand{\DD}{{\mathbb{D}}}
\newcommand{\VV}{{\mathbf{V}}}
\newcommand{\WW}{{\mathbf{W}}}
\newcommand{\XX}{{\mathbf{X}}}
\newcommand{\YY}{{\mathbf{Y}}}
\newcommand{\bH}{{\mathbf{H}}}

\newcommand{\RR}{{\mathbb{R}}}
\newcommand{\RRO}[1]{{\mathbb{R}_{\geq 0}^{#1}}}

\newcommand{\ZZ}{{\mathbb{Z}}}
\newcommand{\CC}{{\mathbb{C}}}

\newcommand{\ttt}{{\mathfrak{t}}}

\newcommand{\hZ}{{\hat{Z}}}

\newcommand{\hZJ}[1]{{\hZ_{#1}}}
\newcommand{\hZI}{\hZJ{\cI}}
\newcommand{\hZX}{\hZJ{\cS}}
\newcommand{\hZA}{\hZJ{\cA}}
\newcommand{\hZXp}{\hZJ{\cT}}
\newcommand{\hZAJ}[1]{{\hZ_{#1}^{\mathrm{t}}}}
\newcommand{\hZAI}{{\hZAJ{\cI}}}

\DeclareMathOperator{\BV}{BV}
\newcommand{\BVJ}[1]{{\BV_{#1}}}

\newcommand{\BVX}{\BVJ{\cS}}

\DeclareMathOperator{\Graph}{Graph}
\DeclareMathOperator{\VGraph}{VGraph}
\DeclareMathOperator{\VG}{VG}
\DeclareMathOperator{\SN}{Sk}
\DeclareMathOperator{\Null}{Null}
\DeclareMathOperator{\PrimNull}{PrimNull}
\DeclareMathOperator{\StrNet}{SN}
\DeclareMathOperator{\strnet}{sn}

\newcommand{\edgecl}[1]{{[#1]_\chi}}

\newcommand{\GraphJ}[1]{\Graph_{#1}}
\newcommand{\VGraphJ}[1]{\VGraph_{#1}}
\newcommand{\VGJ}[1]{\VG_{#1}}
\newcommand{\NullJ}[1]{\Null_{#1}}
\newcommand{\PrimNullJ}[1]{\PrimNull_{#1}}
\newcommand{\SNJ}[1]{\SN_{#1}}

\newcommand{\NullI}{\NullJ{\cI}}

\newcommand{\SNI}{\SNJ{\cI}}

\newcommand{\GraphX}{\GraphJ{\cS}}
\newcommand{\VGraphX}{\VGraphJ{\cS}}
\newcommand{\VGX}{\VGJ{\cS}}
\newcommand{\NullX}{\NullJ{\cS}}
\newcommand{\PrimNullX}{\PrimNullJ{\cS}}
\newcommand{\SNX}{\SNJ{\cS}}
\newcommand{\SNA}{\SNJ{\cA}}

\newcommand{\GraphXp}{\GraphJ{\cT}}
\newcommand{\VGraphXp}{\VGraphJ{\cT}}

\newcommand{\NullXp}{\NullJ{\cT}}

\newcommand{\SNXp}{\SNJ{\cT}}

\newcommand{\GraphJt}[1]{{\Graph_{#1}^\trv}}
\newcommand{\VGraphJt}[1]{{\VGraph_{#1}^\trv}}

\newcommand{\NullJt}[1]{{\Null_{#1}^\trv}}
\newcommand{\PrimNullJt}[1]{\PrimNull_{#1}^\trv}

\newcommand{\GraphXt}{\GraphJt{\cS}}
\newcommand{\VGraphXt}{\VGraphJt{\cS}}

\newcommand{\NullXt}{\NullJt{\cS}}
\newcommand{\PrimNullXt}{\PrimNullJt{\cS}}

\newcommand{\GraphJs}[1]{{\Graph_{#1}^{\trvs}}}
\newcommand{\VGraphJs}[1]{{\VGraph_{#1}^{\trvs}}}
\newcommand{\VGJs}[1]{{\VG_\cS^{\trvs}}}
\newcommand{\NullJs}[1]{{\Null_{#1}^{\trvs}}}
\newcommand{\PrimNullJs}[1]{{\PrimNull_{#1}^{\trvs}}}

\newcommand{\GraphXs}{\GraphJs{\cS}}
\newcommand{\VGraphXs}{\VGraphJs{\cS}}
\newcommand{\VGXs}{\VGJs{\cS}}
\newcommand{\NullXs}{\NullJs{\cS}}
\newcommand{\PrimNullXs}{\PrimNullJs{\cS}}

\newcommand{\GraphAJ}[1]{\Graph_{\mathrm{t}#1}}
\newcommand{\VGraphAJ}[1]{\VGraph_{\mathrm{t}#1}}
\newcommand{\PrimNullAJ}[1]{{\PrimNull_{\mathrm{t}#1}}}
\newcommand{\NullAJ}[1]{\Null_{\mathrm{t}#1}}
\newcommand{\SNAJ}[1]{\SN_{\mathrm{t}#1}}

\newcommand{\VGraphAI}{\VGraphAJ{\cI}}

\newcommand{\NullAI}{\NullAJ{\cI}}
\newcommand{\SNAI}{\SNAJ{\cI}}

\newcommand{\GraphAX}{\GraphAJ{\cS}}
\newcommand{\VGraphAX}{\VGraphAJ{\cS}}
\newcommand{\PrimNullAX}{\PrimNullAJ{\cS}}
\newcommand{\NullAX}{\NullAJ{\cS}}
\newcommand{\SNAX}{\SNAJ{\cS}}

\usepackage[pdfencoding=auto]{hyperref}
    \makeatletter
    \def\HyPsd@expand@utfvii{}
    \makeatother
\hypersetup{
    colorlinks,
    linkcolor={red!50!black},
    citecolor={blue!50!black},
    urlcolor={blue!80!black},
    breaklinks=true
}

\begin{document}

\thispagestyle{empty}
\phantom{.}
\vspace{.7cm}
\begin{center}

{\Large \textbf{Excision for Spaces of Admissible Skeins}}

\vspace{3em}

{\large Ingo Runkel, Christoph Schweigert, Ying Hong Tham}

\vspace{2em}

Fachbereich Mathematik,\\
Universität Hamburg,\\
Bundesstraße 55,\\
20146 Hamburg, Germany

\vspace{2em}

{\small \sf ingo.runkel},
{\small \sf christoph.schweigert},
{\small \sf ying.hong.tham@uni-hamburg.de}
\end{center}

\section*{Abstract}
The skein module for a $d$-dimensional manifold is a vector space spanned by embedded framed graphs decorated by a category $\cA$ with suitable extra structure depending on the dimension $d$, modulo local relations which hold inside $d$-balls. For a full subcategory $\cS$ of $\cA$, an $\cS$-admissible skein module is defined analogously, except that local relations for a given ball may only be applied if outside the ball at least one edge is coloured in $\cS$.

In this paper we prove that admissible skein modules in any dimension
satisfy excision, namely that 
the skein module of a glued manifold is expressed as a coend over boundary values on the boundary components glued together. We furthermore relate skein modules for different choices of $\cS$, apply our result to cylinder categories, and recover the relation to modified traces.

\newpage

\setcounter{tocdepth}{2}
\tableofcontents

%++%
\section{Introduction}

The focus of this paper is a generalisation of skein modules, introduced in dimensions 2 and 3 as \emph{non-unital skein modules} \cite{Reutter:2020}, and as \emph{admissible skein modules} \cite{Costantino:2023}.
One motivation comes from the study of modified traces on pivotal (rigid) tensor categories as defined in \cite{GeerKujawa:2013,GeerPatureauMirand:2013},
which overcome some deficiencies
of the categorical trace based on the pivotal structure in the non-semisimple setting:
the evaluation of a skein in a ball can be computed from the categorical trace,
while the admissible skein module of the ball
is dual to the space of modified traces on an appropriate tensor ideal \cite{Reutter:2020,Costantino:2023}.
Another motivation to study admissible skein modules is to extend the relation between skein modules and state spaces of topological field theories to so-called renormalised quantum invariants which arise from non-semisimple categories as introduced in
\cite{Costantino:2014,Blanchet:2014ova}.
For example,
in \cite{Muller:2023} it was shown that a variant of admissible string nets
for a pivotal finite tensor category reproduce Lyubashenko's modular functor
for the corresponding Drinfeld centre.

\medskip

Let us give a brief description of admissible skein modules for $d$-dimensional oriented manifolds $M$.
Fix a commutative ring $R$.
Let $\cA$ be an essentially small $R$-linear category with suitable extra structure
depending on $d$,
i.e.\ pivotal monoidal for $d = 2$, ribbon for $d = 3$,
and symmetric ribbon for $d ≥ 4$.
At this point, $\cA$ is not required to be abelian or even additive.
For a full subcategory $\cS$ of $\cA$,
an $\cS$-admissible graph is an $\cA$-coloured 
framed graph in $M$
with at least one edge labelled in $\Obj \cS$,
and an $\cS$-admissible skein is a linear combination of $\cS$-admissible graphs
modulo $\cS$-null relations, which are the usual null relations inside balls in $M$
with the extra condition that at least part of an $\cS$-labelled edge lies outside the corresponding ball.
Skeins are allowed to meet the boundary of $M$ transversely, leaving a boundary value
$\VV$ consisting of a finite set of points labelled with elements of $\Obj \cA$.
We denote by $\SNX(M;\VV)$ the $\cS$-admissible skein module of $M$
with boundary value $\VV$;
note that if $\VV$ has at least one label in $\cS$
and $M$ is connected, then the $\cS$-admissibility conditions
become inconsequential and $\SNX(M;\VV) = \SNJ{}(M;\VV)$,
the usual skein module.
We stress that even though it is not included in the notation, $\SNX(M;\VV)$ also depends on the ambient category $\cA$.

The main theorem of this paper is that admissible skein modules
satisfy an excision property.
This is known to hold in the setting of usual $\cA$-coloured skein modules,
see~\cite{Cooke:2019,AlexanderKirillov:2020hkv},
motivated by the unfinished notes \cite{Walker:2006}, 
and see also \cite[Thm.\,5.1]{CostantinoLe:2022}
for (uncoloured) Kauffman skein modules.

One of the theme of this paper is to relate
$\cS$-admissibility conditions for different choices of $\cS$.
In order to to so, we introduce another full subcategory
$\cS ⊆ \cT ⊆ \cA$.
Even if $d \ge 2$ such that $\cA$ is in particular monoidal, we do not require $\cS$ and $\cT$ to be monoidal, too. In particular, they do not need to satisfy
the additional ``ideal'' condition
-- closure under retracts and tensor products with objects in $\cA$ --
as in \cite{Costantino:2023}. Instead we find that the
ideal condition is a natural consequence of topology, that is,
for $\cS ⊆ \cI ⊆ \cA$ where $\cI$ is the ideal closure of $\cS$,
we have $\SNX \simeq \SNI$ (see Proposition~\ref{p:ideal-closure}). 

Let $M$ be a $d$-dimensional oriented manifold with corners,
and let $L$ be a codimension-1 submanifold of $M$.
Cutting $M$ along $L$, we obtain a $d$-manifold $M'$ with new boundary parts $N,N' ⊆ ∂M'$,
and a natural gluing map $π: M' \to M$ identifying $N,N'$ back to $L$.
A skein in $M'$ with the same boundary value $\VV$ at $N$ and $N'$
can be glued up to become a skein in $M$.
This gives maps $π_* : \SNX(M';\XX,\VV,\VV) \to \SNX(M;\XX)$
that are dinatural in $\VV$. We will assume $\VV$ to be $\cT$-admissible and consider it as an object in the cylinder category $\hZXp(N)$, i.e.\ the category whose morphism spaces are skeins in $N \times [0,1]$, see Section~\ref{s:cylinder-cat} for details.
The dinatural maps $π_*$ factor through the coend, and our main theorem states that this induces an isomorphism:

\medskip

\noindent
\textbf{Theorem~\ref{t:excision}.} \emph{Let $\cS ⊆ \cT ⊆ \cA$ be full subcategories.
The maps $π_* : \SNX(M';\XX,\VV,\VV) \to \SNX(M;\XX)$
induced by the gluing $π: M' \to M$ descend to an isomorphism
\[
\hat{\pi}_* : \int^{\hZXp(N)} \SNX(M';\XX,-,-) ~\xrightarrow{~\sim~}~ \SNX(M;\XX) \ .
\]}

\medskip

In the above theorem, the skein modules providing the dinatural transformation satisfy the $\cS$-admissibility condition, while the cylinder category is built from $\cT$-admissible skeins. 
In fact, both $\SNX$ and $\hZXp$ only depend on the completion of $\cS$ and $\cT$ with respect to direct sums and retracts, as well as tensor products and duals for $d\ge 2$.

The freedom to vary $\cT$ in $\cS ⊆ \cT ⊆ \cA$ independently of the $\cS$-admissibility condition turns out to be useful in studying the properties of admissible skein modules. In more detail, the general properties we establish are:
\begin{itemize}
    \item Admissible skein modules over different subcategories can be combined,
for example, for connected
$M = M₁ ∪_N M₂$, and $\cS ⊆ \cT ⊆ \cA$,
we have (see Proposition~\ref{p:different-skein})
\[
\SNX(M) \simeq \int^{\VV ∈ \hZXp(N)} \SNX(M₁;\VV) ⊗_R \SNXp(M₂;\VV) \ .
\]

\item $\cS$-admissible cylinder categories also satisfy excision,
generalising the known result for usual cylinder categories as shown in
\cite{Cooke:2019,AlexanderKirillov:2020hkv}
(see Proposition~\ref{p:excision-cylinder-cat}).

\item Admissible skein modules do not change if one passes from $\cS$ to the smallest full subcategory $\cI$ with $\cS ⊆ \cI ⊆ \cA$, which is closed with respect to retracts, direct sums, tensor products, and duals (the last two conditions apply only for $d \ge 2$), see Proposition~\ref{p:ideal-closure}. Namely, for $M$ a $d$-manifold and $\VV$ any boundary value we have
\[
    \SNX(M;\VV) \,\cong\, \SNI(M;\VV) \ .
\]

\item
We consider a variant of admissible skein modules,
which we call \emph{totally-$\cS$-admissible skein modules},
where every edge is labelled in $\cS$.
When $\cS = \cI$ is an ideal, these are isomorphic to the original
$\cS$-admissible skein modules (see Proposition~\ref{p:totally-I}).

\end{itemize}

As applications of the above results, we study admissible skein modules in more specific settings:
\begin{itemize}
    \item Let $\cA$ be an $R$-linear pivotal category 
    which is moreover abelian and has enough projectives, and let $\cI ⊆ \cA$ be the ideal of projective objects. Consider a surface $\Sigma$ with boundary value $(V_1,\dots,V_n) \in \cA^{\times n}$. Treating the objects as variables, the corresponding skein module can be turned into  a functor $\cA^{\times n} \to \Rmod$. We show that this functor is right-exact in each argument
    (see Proposition~\ref{p:right-exact}). This allows us to make contact to the string net spaces in \cite{Muller:2023}.

\item
We recover the relation between the admissible skein modules 
$\SNI(\DD²)$ and $\SNI(S^2)$ for an ideal $\cI$ and modified traces on $\cI$
as in \cite{Reutter:2020,Costantino:2023}. While the proofs are similar in spirit,
our proof works more directly with the skein modules,
which we think more clearly demonstrates the conceptual connection between
the topology and the partial trace property of modified traces
(see Proposition~\ref{p:trace-skein-D2-S2}).

\item
Working over a field $R = \kk$,
we present some results on the dimensions of the skein modules.
In particular, when $\cA$ is ribbon and
an ideal $\cI$ has infinitely many indecomposable objects,
then for surfaces $Σ$ with non-trivial first homology,
$\SNI(Σ)$ is infinite dimensional (see Section~\ref{s:dimension}).
We also give an example of an ideal $\cI$ in a finite pivotal tensor category such that $\SNI(\DD²)$ and $\SNI(S^2)$ are infinite dimensional.
\end{itemize}

The paper is structured as follows.
In Section~\ref{s:defn}, we give the basic definitions and constructions
concerning admissible skein modules and cylinder categories.
In Section~\ref{s:excision}, we prove the main theorem.
In Section~\ref{s:properties}, we prove several general properties of
admissible skein modules, applying to any manifold and dimension $d$
and any $\cA$ and $\cS$.
In Section~\ref{s:applications}, we give several applications of the main theorem as discussed above.

\medskip

\noindent
\textbf{Note added:} While we were finalising the present paper, the preprint \cite{Brown:2024} appeared which has 
some overlap with our results. In particular, their Theorem 2.21 includes a proof of excision as in our Theorem~\ref{t:excision} in the case $d=3$ with $\cS = \cT = \cI$ and $\cI$ in addition satisfying an ideal condition. In \cite{Brown:2024} the authors work with what we call totally-$\cI$-admissible skein modules, but due to Proposition~\ref{p:totally-I} below, this does not make a difference.

\subsubsection*{Acknowledgements:}
We thank
Jennifer Brown,
Benjamin Ha\"ioun,
Aaron Hofer,
Alexander Mang,
Vincentas Mulevičius,
David Reutter,
    and
Paul Wedrich
for discussions.
YHT is supported by the Deutsche Forschungsgemeinschaft (DFG, German Research Foundation) under Germany's Excellence Strategy - EXC 2121 ``Quantum Universe" - 390833306.
IR and CS are supported in part by the Collaborative Research Centre CRC 1624 ``Higher structures, moduli spaces and integrability'' - 506632645, and the Excellence Cluster EXC 2121 ``Quantum Universe" - 390833306.

\subsubsection*{Conventions:}

Manifolds are oriented, smooth (i.e.\ $C^{∞}$) with corners,
and have finitely many connected components, but they need not be compact;
see Appendix~\ref{s:mfld-corner}.
All categories will be $R$-linear, where $R$ is a fixed choice of commutative ring.
Categories are typically not required to be additive;
they will come with extra structure depending on the dimension $d$
as listed in Table~\ref{t:A-structure-vs-d} below.

There is quite a lot of notation introduced in this paper and for reference we provide a glossary in  Appendix~\ref{s:notation}.
In Appendix~\ref{s:mfld-corner} we briefly review some basic definitions for manifolds with corners.

%++%
\section{Definitions}
\label{s:defn}

In this section we give the definition of admissible skein modules 
for a $d$-dimensional manifold following \cite{Costantino:2023}, though our setting will differ in some aspects. 
This construction was also announced in \cite{Reutter:2020} under the name of non-unital skein theory.
We then state some basic properties of admissible skein modules, review cylinder categories, and explain how skein modules define functors out of cylinder categories.

We already summarised the setting in the introduction, and the reader familiar with skein modules may find it convenient to have a quick look at Definition~\ref{d:S-skein} (admissible skein modules), as well as at \eqref{e:VGraphX}, \eqref{eq:PrimNullS} and \eqref{eq:NullS} for the notation used there, and then jump ahead to Section~\ref{s:cylinder-cat}, referring back to the definitions and results in Sections~\ref{s:adm-skein-defn} and \ref{s:basic-prop} as needed.

%++%
\subsection{Admissible Skeins}
\label{s:adm-skein-defn}

\subsubsection{Topological setting}\label{s:topol-setup}

Let us first discuss the topological aspects of the construction
of admissible skein modules.
We use manifolds with corners, largely based on \cite{Hajek:2014} (which in turn is based on \cite{Lee-book,Joyce:2012}).
\label{pg:manifold}
The corners do not make any difference on the skein modules themselves
(see Lemma \ref{l:ignore-corners}),
but they help with the technicalities of gluing manifolds along boundaries.
We give an exposition on manifolds with corners in
Appendix~\ref{s:mfld-corner}.
In short, manifolds with corners have coordinate charts
that are locally modelled on open subsets of $\RRO{d}$, with $\RRO{} = [0,∞)$.
Our manifolds with corners will be assumed to have
finitely many connected components, but they need not be compact.

\label{pg:boundary-piece}
The \emph{open boundary of $M$} consist of points that are not in
the interior $\Int(M)$ of $M$ nor its corners,
and an \emph{open boundary piece of $M$} is a connected component of it.
A \emph{boundary piece of $M$} is an appropriate closure
of an open boundary piece of $M$,
and the \emph{boundary $∂M$ of $M$} is the disjoint union of
all boundary pieces of $M$.
Figure~\ref{f:teardrop} gives an example of a manifold with corners
of dimension 2 and shows the natural surjection $∂M \to M \backslash \Int(M) \subseteq M$, which also illustrates that it need not be injective.
For a boundary piece $N ⊆ ∂M$,
when this map restricted to $N$ is a diffeomorphism,
we say $N$ is an \emph{embedded boundary piece}.

Each boundary piece $N ⊆ ∂M$ is oriented;
we say $N$ is \emph{outgoing} if, at a point $b ∈ \Int(N)$,
the orientation given by concatenating an
outward-pointing vector at $b$ with the orientation on $N$
is equal to the orientation on $M$,
and we say $N$ is incoming otherwise.

\begin{figure}
\centering
\includegraphics[width=4cm]{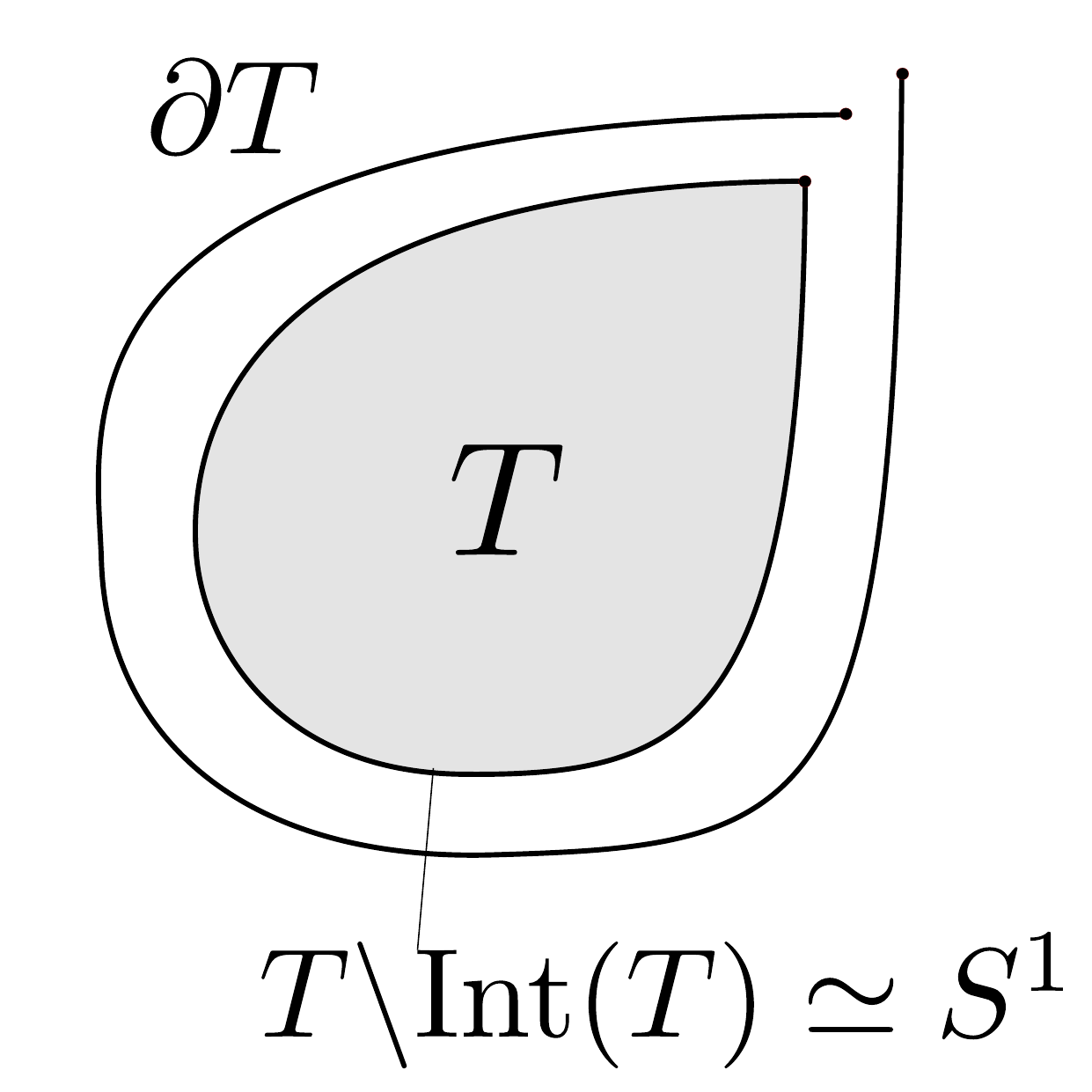}
\caption{Teardrop $T$.
The boundary $∂T$ is a closed interval, depicted as an unclosed loop
wrapped around $T$, naturally mapping down onto $T \backslash \Int(T)$;
the two endpoints of $∂T$ are sent to the corner of $T$.
}
\label{f:teardrop}
\end{figure}

A \emph{collar neighbourhood} of an embedded boundary piece $N$
is a orientation-preserving embedding
$[0,1] \times N ↪ M$ or $[-1,0] \times N ↪ M$;
then $N$ is an \emph{incoming} or \emph{outgoing boundary piece},
respectively.
For a codimension-1 embedded submanifold $L ⊆ M$,
a \emph{tubular neighbourhood} is an orientation-preserving embedding
$\wdtld{ι}: [-1,1] \times L ↪ M$.
There are extra technical conditions on collar and tubular neighbourhoods which we state in Appendix~\ref{s:mfld-corner}.
Cutting $M$ along $L$, we obtain a new manifold $M'$
with $L_-, L_+ ⊆ ∂M$ and a projection $π: M' \to M$,
such that $π|_{M' \backslash L_-,L_+}$ is a diffeomorphism onto $M \backslash L$,
and $π|_{L_-}$ and $π|_{L_+}$ are diffeomorphisms onto $L$;
the tubular neighbourhood $\wdtld{ι}$ restricts and lifts to
collar neighbourhoods $ξ_-: [-1,0] \times L_- \to M'$,
$ξ_+: [0,1] \times L_+ \to M'$ of $L_-,L_+$,
see Figure~\ref{f:Mp-to-M}.

\begin{figure}
\centering
\includegraphics[width=12cm]{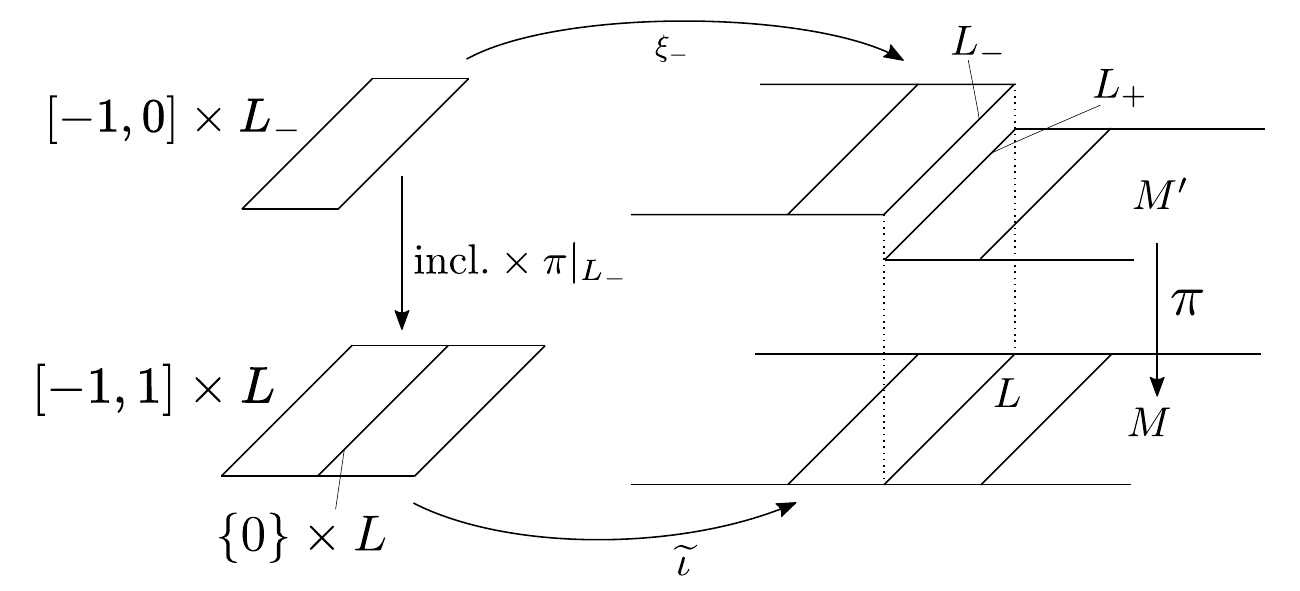}
\caption{Cutting $M$ along a codimension-1 submanifold $L$.
There is a corresponding map $\xi_+ : [0,1] \times L_+ \to M'$ not shown in the diagram.
}
\label{f:Mp-to-M}
\end{figure}

\label{pg:Cinfty-isotopy}
We will consider ambient isotopies of $d$-manifolds $M$,
that is, a smooth map 
$φ^{\bullet} : [0,1] \times M \to M$, $(t,x) \mapsto φ^t(x)$,
such that $φ^t$ is a diffeomorphism for each time $t$;
unless otherwise specified, we will assume that an ambient isotopy begin
with the identity, i.e.\ that $φ⁰ = \id_M$
(ambient isotopies are also called diffeotopies in \cite{Hirsch:2012}).
We say $φ^{\bullet}$ is \emph{$C^{∞}$-constant on $V ⊆ M$} 
if for all $t$, it leaves $V$ pointwise fixed,
and the $k$-jet of $φ^t$ at points in $V$
is the same as that of $\id_M$ for all $t$
(essentially $φ^t|_V$ is the identity map up to all orders of derivatives;
see \cite{Golubitsky:2012} for notes on $k$-jets).
In particular, we say $φ^{\bullet}$ is \emph{$C^{∞}$-constant at the boundary}
if it is $C^{∞}$-constant for $V = ∂M$.

\label{pg:uncoloured-graph}
An \emph{(uncoloured-)graph $\ov{Γ}$ in $M$} refers to
a finite smoothly embedded oriented graph in $M$,
possibly meeting the boundary $∂M$ (where the intersection is required to be
transversal and away from the corners),
together with a continuous framing of the normal bundle to each edge
such that the tangent vector plus the framing gives the orientation of $M$;
in addition for $d ≥ 2$, each vertex $v$ comes with a
coframed oriented 2-plane $Ξ_v$ in its tangent space $T_v M$,
such that the edges meeting $v$ are tangent to the 2-plane $Ξ_v$,
and framings are compatible, and the edges around a vertex
are partitioned into \emph{incoming} and \emph{outgoing} edges;
see Figure~\ref{f:vertex} for details.
In the case $d = 1$, we require that the graph covers the entire manifold
(see also Remark \ref{r:ncat-strat} below).

\begin{figure}
\centering
\includegraphics[width=8cm]{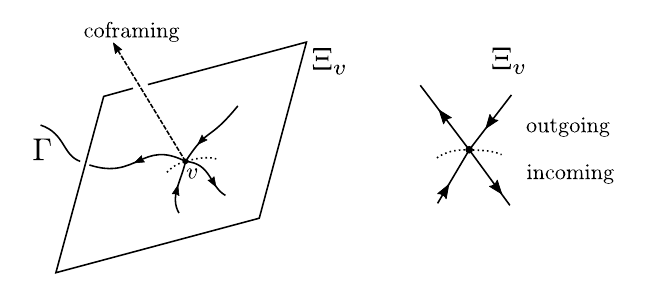}
\caption{Data around vertex for $d ≥ 2$.
The orientation on the 2-plane $Ξ_v$ plus the coframing (framing of $T_vM / Ξ_v$)
is required to give the ambient orientation of $M$.
For each edge $e$ of $Γ$, with framing $(u₂,...,u_d)$ at $v$,
the first vector $u₂$ lies in $Ξ_v$ (more precisely, there should be
a vector in $Ξ_v$ that projects onto $u₂$ under the quotient map
to the normal bundle of $e$),
and the quotient map from the normal bundle of $e$ to $T_vM / Ξ_v$
must send the rest of the framing, $(u₃,...,u_d)$,
to the coframing of $Ξ_v$.
Furthermore, the tangent vectors of the edges at $v$ must be distinct.
They are partitioned into two contiguous ordered subsets of
\emph{incoming} and \emph{outgoing} edges,
here depicted by the dashed lines;
note that the edge orientations are independent from the designation of
being an incoming/outgoing edge.
}
\label{f:vertex}
\end{figure}

\label{pg:config-marked-points}
For a $(d-1)$-manifold $N$,
we define a \emph{configuration of marked points on $N$}
as a finite set of points $B ⊆ N \backslash ∂N$,
a $(d-1)$-frame of $T_b N$ for each $b ∈ B$,
with \emph{infinitesimal transverse data}
in the form of the data of all $k$-jets, $k ≥ 0$, of some oriented smooth curve $γ_b$
(same curve for all $k$) through $(0,b) ∈ \RR \times N$
that is transverse to the 0-section $\{0\} \times N$.
The infinitesimal transverse data is considered the same under to reparametrisation
of the curves $γ_b$, so that only the image of $γ_b$ matters.
For $d = 1$, we impose the condition $B = N$.
If the smooth curve $γ_b$ is oriented positively in the $\RR$ direction,
then we say the point $b$ is \emph{positively oriented};
otherwise, we say it is \emph{negatively oriented} otherwise.

For the boundary $∂M = \bigcup N_i$ of a $d$-manifold,
a configuration of marked points is a collection of configurations on each $N_i$;
since the marked points cannot be on the boundary of the $N_i$'s,
this is equivalent to a configuration of marked points on $∂M \backslash \text{corners}$.
In order to be able to relate the infinitesimal transverse data of a marked point $b$ to the $k$-jets of a curve $μ : [0,ε) \to M$ with $μ(0) = b$,
we always implicitly fix an embedding
$(-ε,0] \times \Int(N_i)$ or $[0,ε) \times \Int(N_i) ↪ M$ for some small $ε > 0$
for every boundary piece $N_i$,
depending on whether $N_i$ is outgoing or incoming, respectively. Again, actually only the $k$-jets of these embeddings matter.
When we endow an embedded boundary piece with a collar neighbourhood,
we also implicitly assume that this embedding
is given by restricting the collar neighbourhood.

The intersection of a graph $\ov{Γ}$ with each $N_i ⊆ ∂M$
defines a \emph{boundary configuration (of marked points)} on $N_i$,
where the set of points is $∂\ov{Γ}$, and for each point $b ∈ ∂\ov{Γ}$,
the framing of $T_b N_i$ is the unique framing that
descends to the framing at $b$ of the normal bundle along the edge $e$ ending at $b$,
and infinitesimal transverse data is given by $e$;
the boundary configuration of $\ov{Γ}$ refers to the collection of these
boundary configurations at all such pieces.
An ambient isotopy $φ^{\bullet}$ of $M$ that is $C^{∞}$-constant at the boundary
will preserve the boundary values of graphs.

We will consider isotopies $Γ^{\bullet}$ of graphs, that is,
a smooth
family $Γ^{\bullet} : [0,1] \times Γ \to M$ of graphs, in the sense that the restriction of $\Gamma^\bullet$ to each edge of $Γ$ is smooth. We will also consider
a more restrictive type of graph isotopies consisting of those that are
carried by an ambient isotopy, i.e.\ there is an isotopy $φ^{\bullet}$ of $M$
such that $Γ^t = φ^t(Γ)$.
Ignoring the framing, ambient isotopies can approximate graph isotopies
everywhere except in a tiny neighbourhood of vertices
(see Lemma \ref{l:ambient-graph}).
We say an isotopy of graphs $Γ^{\bullet}$ (of the former type) is
\emph{$C^{∞}$-constant at the boundary}
if the infinitesimal transverse data of $∂Γ$ is constant through the isotopy
and the framing is constant at $∂Γ$.
This condition allows a graph isotopy to be extended outside a region:
given a subset $U ⊆ M$, a graph isotopy of $Γ ∩ U$
extends to a graph isotopy of all of $Γ$ by the constant isotopy outside $U$
if the graph isotopy of $Γ ∩ U$ is $C^{∞}$-constant on $∂(Γ ∩ U)$.

\subsubsection{Algebraic input and relations}\label{s:alg-input}

\begin{table}[bt]
\begin{center}
\begin{tabular}{|c|c|}
\hline
$d$ & $\cA$
\\
\hline
1 & no requirement
\\
\hline
2 & pivotal
\\
\hline
3 & ribbon
\\
\hline
≥ 4 & symmetric ribbon
\\
\hline
\end{tabular}
\end{center}
\caption{Additional structure on the category $\cA$ depending on the dimension $d$ of the manifold in which the $\cA$-coloured graphs are embedded. Independent of $d$, $\cA$ is assumed to be essentially small and $R$-linear.}
\label{t:A-structure-vs-d}
\end{table}

Now we consider algebraic aspects. Fix a commutative ring $R$.\label{p:ring-R}
The construction of admissible skein modules
depends on a choice of an essentially small
$R$-linear category $\cA$ as input,\label{p:cat-A}
equipped with extra structure depending on $d$  as given in Table~\ref{t:A-structure-vs-d}.
That $\cA$ is essentially small is needed in the excision statement later on to ensure the existence of certain coends.
We do not require that $\cA$ is abelian or even additive.
Fix a full subcategory $\cS ⊆ \cA$.
Note that $\cS$, as yet, does not need to be closed under
retracts or tensor products with objects in $\cA$.
This differs from the setting in \cite{Costantino:2023},
but we will show in Proposition~\ref{p:ideal-closure} below
that our construction of admissible skein modules
is invariant under closure with respect to these operations.

\label{pg:edge-colouring}
An \emph{$\cA$-edge-colouring} of a graph $\ov{Γ}$ is an assignment $χ$
of an object $χ_e ∈ \cA$ to each edge $e$ of the graph.
We refer to $(\ov{Γ},χ)$ as an $\cA$-edge-coloured graph.
\label{pg:edge-colouring-admissible}

\begin{definition}
For a full subcategory $\cS ⊆ \cA$,
an \emph{$\cS$-admissible} edge-colouring $χ$ is a colouring such that
in each connected component of $M$ there is at least one edge $e$ of $\ov{Γ}$ with colour $χ_e \in \cS$.
\end{definition}

The subcategory $\cS$ will enter the construction of $\cS$-admissible skein modules only through the colouring of edges via objects in $\cS$ and not via its morphisms -- this is the reason we only consider full subcategories $\cS ⊆ \cA$.

Given an $\cA$-edge coloured graph $(\ov{Γ},χ)$,
each vertex $v$ has a space of colourings $\Col(\ov{Γ},χ, v) \in \Rmod$,
which is defined to be
\[
\cA(χ_{e₁}^{[*]} ⊗ \cdots ⊗ χ_{e_k}^{[*]},
χ_{e₁'}^{[*]} ⊗ \cdots ⊗ χ_{e_l'}^{[*]}) \ ,
\]
where $e₁,...,e_k;e_l',...,e₁'$ are the incoming and outgoing edges
(see Figure~\ref{f:vertex}), and the $[*]$ in the superscript indicates that we take the dual if
the edge's orientation disagrees with the ``expected'' direction
from being an incoming/outgoing edge.
If there are no incoming/outgoing edges, then the corresponding tensor product
is defined to be $\one$;
note that in the $d = 1$ case, this is ruled out by our stipulation that
the graph cover the entire manifold.
\label{pg:vertex-colouring}
A \emph{vertex colouring} of $(\ov{Γ}, χ)$
is an assignment $φ$ of an element $φ_v ∈ \Col(\ov{Γ},χ,v)$
for each vertex $v$.
\label{pg:vertex-colouring-all}
Then the space of vertex colourings on $(\ov{Γ},χ)$ is
\[
\Col(\ov{Γ},χ) := \prod_v \Col(\ov{Γ},χ,v) \ .
\]
We stress that this is the product of sets, not the tensor product over $R$ of $R$-modules -- the reason for this is explained in Remark~\ref{r:colouring} below.
\label{pg:coloured-graph}
We define an \emph{$\cA$-coloured graph} to be a triple $Γ = (\ov{Γ}, χ, φ)$;
its underlying edge-coloured graph is $\edgecl{Γ} = (\ov{Γ}, χ)$.
\label{pg:admissible-graph}
An $\cA$-coloured graph $Γ = (\ov{Γ}, χ, φ)$
is $\cS$-admissible if the edge colouring $χ$ is $\cS$-admissible.

\label{pg:boundary-value}
A \emph{boundary value} $\VV = (B, \{χ_b\})$
on a $(d-1)$-manifold $N$
is a configuration of marked points $B$ on $N$
together with an assignment of objects $χ_b ∈ \cA$ to each $b$.
\label{pg:bv-admissible-bulk}
\label{pg:bv-admissible}

\begin{definition}
Let $\cS ⊆ \cA$ be a full subcategory, $N$ be a $(d-1)$-manifold, and $\VV$ a boundary value on $N$.
\begin{enumerate}
    \item $\VV$ is \emph{$\cS$-admissible} if in each component of $N$ at least one point of $B$ is labelled by an object in $\cS$.
    \item If $N = ∂M$, $\VV$ is \emph{bulk-$\cS$-admissible (relative to $M$)} if in each component $M_j$ of $M$, the restriction $\VV|_{N_i}$ to at least one boundary piece $N_i ⊆ ∂M_j$ is $\cS$-admissible.
\end{enumerate}
\end{definition}

The boundary value of an $\cA$-coloured graph $Γ$
is the boundary configuration of its underlying uncoloured graph $\ov{Γ}$
together with the assignment $χ_b = χ_e$ where $b = e ∩ \ov{Γ}$;
we denote this boundary value by $∂Γ$, or, since it depends only on
the underlying edge-coloured graph, we may write $∂\edgecl{Γ}$.
Note that a graph whose boundary value is bulk-$\cS$-admissible
is automatically $\cS$-admissible, but the converse is not true;
in particular, there exist $\cS$-admissible graphs with no boundary.

We denote (here $\coprod$ denotes disjoint union of sets):
\begin{align}
\label{e:Graph}
\Graph(M;\VV) & :=
\{ (\ov{Γ}, χ, φ) \;|\; ∂(\ov{Γ}, χ) = \VV \} ~= \hspace{-.3em}
\coprod_{∂(\ov{Γ}, χ) = \VV} \hspace{-.4em} \Col(\ov{Γ}, χ)
\\
\label{e:VGraph}
\VGraph(M;\VV) & :=
\text{ free } R \text{-module generated by }
\Graph(M;\VV)
\end{align}
and
\begin{align}
\label{e:GraphS}
\GraphX(M;\VV) & :=
\{ (\ov{Γ}, χ, φ) ∈ \Graph(M;\VV) \;|\;
χ \text{ is } \cS\text{-admissible} \}
\\
\label{e:VGraphX}
\VGraphX(M;\VV) & :=
\text{ free } R \text{-module generated by }
\GraphX(M;\VV)
\end{align}

\begin{remark}
\label{r:colouring}
For $a ∈ R$ let us denote by $a ⋅_v Γ$
the same $\cA$-coloured graph as $Γ$ except at the vertex $v$,
where the colouring is scaled by $a$.
Then we actually have
\[
a ⋅ Γ ≠ a ⋅_v Γ ∈ \VGraphX(M;\VV) \ ,
\]
and a similar inequality holds for sums of vertex colourings.
Intuitively, these should be identified,
and indeed, they will be identified in the skein module.
However, we want the identification to be a consequence of
the null relations that we describe below.
This makes the separation between topology and algebra a bit clearer.
For example, as it is defined now,
the restriction of an $\cA$-coloured graph $Γ$ to a subset of $M$
is well-defined, but if the above inequality were an equality,
we would have $a ⋅_v Γ = a ⋅ Γ = a ⋅_{v'} Γ$,
so if we consider an open set $U$ containing $v$ but not $v'$,
then $Γ ∩ U$ is not well-defined.
\end{remark}

%++%

Let $M = \DD^d = \{\mathbf{x} ∈ \RR^d \;|\; |\mathbf{x}| ≤ 1\}$
the standard closed ball,
and for $d ≥ 2$, consider some standard fixed boundary configuration
on the unit circle in the $x_1$-$x_2$-plane.
For $\cA$-coloured graph $Γ ⊆ \DD^d$ with this boundary configuration,
there is an evaluation
\begin{equation}
\label{e:RT-evaluation}
\eval{Γ} ∈
\begin{cases}
\cA(\one,\bigotimes χ_e^{[*]})
& \text{ for } d ≥ 2
\\
\cA(χ_e, χ_{e'})
& \text{ for } d = 1
\end{cases}
\end{equation}
where in the $d ≥ 2$ case, $\bigotimes χ_e^{[*]}$ is the tensor product of
the labels of edges of $Γ$ intersecting $∂\DD^d$,
taking the dual if the edge is oriented inwards, and in the $d = 1$ case,
$e, e'$ are the edges going into and coming out of $\DD^d$.
For $d = 1$, $\eval{Γ}$ is evaluated by simply composing the morphisms
labelling the internal vertices of $Γ$.
For $d = 2, 3$, $\eval{Γ}$ is the usual Reshetikhin-Turaev evaluation
\cite{ReshetikhinTuraev:1990}.
For $d ≥ 4$, it is possible to ambiently isotope $Γ$ so that
its projection to $\DD³ = \DD^d ∩ \{x₄ = ... = x_d = 0\}$ is generic,
i.e.\ the image of $Γ$ under projection has no self-intersections,
and such that for vertices, the 2-plane $\Xi_v$ is tangent to $D^3$.
Furthermore, we homotope all framings to be the constant standard frame
in the last $d-3$ tangent directions.\footnote{
This can always be achieved by first rotating the framing in a neighbourhood of the vertices so that at each vertex the last $d-3$ vectors of the frame are $e_4,e_5,\dots$. For each edge pick a new framing arbitrarily such that the last $d-3$ vectors are of this form. The new framing either is in the same homotopy class as the original framing (relative to the fixed endpoints at the vertices), or it is not. If it is not add an extra twist to the first three components -- since there are just two homotopy classes of framings (by $π₁(\textrm{GL}_{d-1}^+(\RR)) \simeq \ZZ/2)$ for $d \ge 4$),
the resulting framing of the edge is then homotopic to the initial one.}
We can readily evaluate the projection $\wdtld{Γ}$ of $Γ$,
and we set $\eval{Γ} = \eval{\wdtld{Γ}}$.
We note that this means that even though the category is symmetric,
the framing on the graph can still matter (modulo inserting an even number of twists).

For an $\cA$-coloured graph $Γ$ in $M$ and a codimension-1 submanifold $N ⊆ M$,
we write $Γ \trv N$ if $\Gamma$ intersects $N$ transversally
and away from its boundary and corners.
Consider $D ⊆ M$ a smoothly embedded closed $d$-ball,
possibly meeting the boundary
(think ``squashed against the boundary in one or more places'',
see e.g. Figure \ref{f:squashed-ball}).
For an $\cA$-coloured graph $Γ$ such that $Γ \trv ∂D$,
we denote its \emph{evaluation in $D$} by
$\eval{Γ}_D = \eval{φ(Γ)}$,
where $φ$ is a diffeomorphism carrying $D$ to $\DD^d$
and the boundary configuration of $Γ$ to the standard one on $\DD^d$.
Certainly, $\eval{Γ}_D$ depends on $φ$, but ultimately for our purposes,
in particular for defining null relations,
the choice does not matter, since they are all related by
(non-canonical) isomorphisms,
and we only care about comparing with 0.

The most important properties of the evaluation are (see
\cite{ReshetikhinTuraev:1990}, as well as \cite[Thm.\,2.6]{TuraevVirelizier:2017} (2d case) and \cite[Thm.\,I.2.5]{Turaev-book} (3d case)):
\begin{itemize}
\item $\eval{⋅}_D$ is invariant under isotopy supported in $D$,
i.e.\ if $Γ^{\bullet}$ is a graph isotopy that is $C^{∞}$-constant outside
the interior of $D$, then $\eval{Γ⁰}_D = \eval{Γ¹}_D$.

\item $\eval{⋅}_D$ is invariant under continuous change of framing,
i.e.\ if $Γ⁰,Γ¹$ are $\cA$-coloured graphs, equal as unframed oriented graphs,
that are related by a continuous change of framing of its edges
fixed at their endpoints, then $\eval{Γ⁰}_D = \eval{Γ¹}_D$.

\item $\eval{⋅}_D$ is closed under inclusion of balls,
i.e.\ if $D' ⊆ D$ is a smaller ball, and $\sum a_i Γ_i,\sum b_j Γ_j'$ are
linear combinations of $\cA$-coloured graphs that agree outside $D'$ and
$\sum a_i \eval{Γ_i}_{D'} = \sum b_j \eval{Γ_j'}_{D'}$,
then $\sum a_i \eval{Γ_i}_D = \sum b_j \eval{Γ_j'}_D$.

\item For $D = M = \DD^d$, and $Γ$ a star graph,
i.e.\ one vertex at the origin and edges point radially out,
if the vertex is labelled by a morphism $f$,
then $\eval{Γ}_D = f$.
\end{itemize}

%++%

Let $D ⊆ M$ be a smoothly embedded closed $d$-ball,
possibly meeting the boundary.
\label{pg:primitive-null}
We define a \emph{primitive null graph with respect to $D$}
to be a sum $\sum a_j Γ_j$ such that $Γ_j \trv ∂D$ and
the $Γ_j$'s agree outside $D$,
and the linear combination of their evaluations in $D$ is 0,
$\sum a_j \eval{Γ_j}_D = 0$.
We denote the set of primitive null graphs with a given boundary value $\VV$ by
\begin{align*}
\PrimNull(M,D;\VV) \,:=~
	&\big\{\, \sum a_j Γ_j ∈ \VGraph(M;\VV) \,\big|
 \\
		& \quad Γ_j \backslash D \text{ are the same, }
		Γ_j \trv ∂D \,,~
		\sum a_j \eval{Γ_j}_D = 0
	\,\big\} \ .
\end{align*}
For \label{pg:null} a subset $U ⊆ M$,
we define a \emph{null relation in $U$} to be a sum of
primitive null graphs that are null with respect to balls
with interior contained in $U$:\footnote{
We take $\Int D ⊆ U$ instead of $D ⊆ U$,
as we will use $U = M \backslash L$, $L$ a codimension-1 submanifold,
and we want to allow $D$ to be ``squashed against $L$'';
see e.g. Figure \ref{f:squashed-ball}.}
\[
\Null(M,U;\VV) = \Span \big\{\, \sum a_j Γ_j ∈ \PrimNull(M,D;\VV) \,|\,
	\Int D ⊆ U \,\big\}
\]
Omitting $U$ means we take $U = M$.
\label{pg:skein}
The usual \emph{skein module of $M$ with boundary value $\VV$}
is defined as
\[
\SNJ{}(M;\VV) := \VGraph(M;\VV) / \Null(M;\VV) \ .
\]

Next we define the $\cS$-admissible versions. Set
\label{pg:S-primitive-null}
\begin{equation}
\label{eq:PrimNullS}
\PrimNullX(M,D;\VV) :=
	\big\{ \, \sum a_j Γ_j ∈  \PrimNull(M,D;\VV) \,\big|\,
		Γ_j \backslash D
		\text{ is } \cS \text{-admissible}
	\,\big\} \ .
\end{equation}
We say $\sum a_j Γ_j ∈ \PrimNull(M,D;\VV)$ are
\emph{primitive $\cS$-null graphs with respect to $D$}.
In other words, the primitive $\cS$-null graphs are those primitive null graphs
for which the part of $\Gamma_j$ that lies outside of $D$ is $\cS$-admissible
(recall that by definition all $\Gamma_j$ in $\sum a_j Γ_j$ agree outside of $D$).

\label{pg:S-null}

For a subset $U ⊆ M$,
we define an \emph{$\cS$-null relation in $U$} to be a sum of
primitive $\cS$-null graphs that are null
with respect to a ball $D$ with $\Int D ⊆ U$:
\begin{equation}
\label{eq:NullS}
\NullX(M,U;\VV) = \Span \big\{\, \sum a_j Γ_j ∈ \PrimNullX(M,D;\VV) \,\big|\,
	\Int D ⊆ U \,\big\} \ .
\end{equation}
As above, omitting $U$ means we take $U = M$.

\begin{definition}
\label{d:S-skein}
Let $\cS ⊆ \cA$ be a full subcategory.
The \emph{$\cS$-admissible skein module of $M$ with boundary value $\VV$}
is defined as
\[
\SNX(M;\VV) := \VGraphX(M;\VV) / \NullX(M;\VV) \ .
\]
\end{definition}

The two notions of skein modules, $\SN$ and $\SNX$, are naturally related,
\label{pg:iota-S}
namely, there are obvious inclusions
$\VGraphX(M;\VV) ⊆ \VGraph(M;\VV)$ and $\NullX(M;\VV) ⊆ \Null(M;\VV)$,
resulting in a map $\ov{ι}_{\cS} : \SNX(M;\VV) \to \SN(M;\VV)$:
\begin{equation}
\label{e:iota-S-A}
\begin{tikzcd}
\NullX(M;\VV) \ar[r] \ar[d]
& \VGraphX(M;\VV) \ar[r] \ar[d]
& \SNX(M;\VV)
\ar[d, "\ov{ι}_{\cS}"]
\\
\Null(M;\VV) \ar[r]
& \VGraph(M;\VV) \ar[r]
& \SN(M;\VV)
\end{tikzcd}
\end{equation}
Similarly, an inclusion of full subcategories $\cS ⊆ \cT$
induces a map $\ov{ι}_{\cS ⊆ \cT} : \SNX(M;\VV) \to \SNXp(M;\VV)$: \label{pg:iota-S-T}
\begin{equation}
\label{e:iota-S-Sp}
\begin{tikzcd}
\NullX(M;\VV) \ar[r] \ar[d]
& \VGraphX(M;\VV) \ar[r] \ar[d]
& \SNX(M;\VV)
\ar[d, "\ov{ι}_{\cS ⊆ \cT}"]
\\
\NullXp(M;\VV) \ar[r]
& \VGraphXp(M;\VV) \ar[r]
& \SNXp(M;\VV)
\end{tikzcd}
\end{equation}
These maps $\ov{ι}_{\cS}, \ov{ι}_{\cS ⊆ \cT}$ are compatible with each other in the following sense:
\[
\ov{ι}_{\cS} = \ov{ι}_{\cT} \circ \ov{ι}_{\cS ⊆ \cT} \ .
\]
We may even generalise to (pivotal/ribbon) functors $\mathcal{F} : \cA \to \cA'$,
sending a full subcategory $\cS ⊆ \cA$ into a full subcategory $\cS' ⊆ \cA'$,
resulting in a map $\ov{ι}_{\mathcal{F}} : \SNX \to \SNJ{\cS'}$
(see also \cite[Prop.\,2.3]{Costantino:2023}),
but we will not be using such maps in this paper.

\begin{remark}
\label{r:ncat-strat}
\begin{enumerate}
\item
In general, the map $\ov{ι}_{\cS ⊆ \cT}$ is neither injective nor surjective.
For example, consider $M = \DD²$ with empty boundary value.
Take $\cA$ to be unimodular finite tensor and non-semisimple,
and take $\cS = \Proj(\cA)$ to be the projective ideal.
Then both $\SNX(\DD²)$ and $\SNJ{}(\DD²)$ are 1-dimensional,
but the map $\ov{ι}_{\cS ⊆ \cA}$ is 0.
Skein modules for this choice of $\cA$ and $\cS$ are discussed in more detail
in Section~\ref{s:relation-MSWY}.

\item
It seems plausible to expect 
that there is a suitable notion of $d$-category,
which allows one to define
``$d$-dimensional skeins'' in $d$-dimensional manifolds.
Such skein modules are spanned by stratifications of the manifold,
with labelling of codimension-$k$ strata 
with $k$-morphisms,
and which are again considered up to null relations in $d$-balls.
See for example \cite{Fuchs:2021} for $d = 2$,
and \cite{BarrettMeusburger:2024} for $d = 3$.
In the context of the present paper, we can think of the graphs defining our skeins
as the $d$- and $(d-1)$-codimension strata,
and the $d$-category is highly degenerate in that there is a unique
$k$-morphism for $k ≤ d-2$. Hence there is no need to label
codimension-$k$ strata other than $k = d-1$ and $d$.
When $d = 1$, this means that the graphs representing skeins
should in fact be a stratification of the 1-manifold,
hence must be surjective on the whole 1-manifold.
\end{enumerate}
\end{remark}

%++%
\subsection{Some Basic Properties}
\label{s:basic-prop}

Recall from Section~\ref{s:topol-setup} the notion of an isotopy of graphs $Γ^{\bullet}: Γ \times [0,1] \to M$.
In \cite{Costantino:2023}, $\cA$-coloured graphs were considered up to isotopy,
even before taking quotient by null relations.
In our setting, isotopy invariance has to be proved:

\begin{lemma}
\label{l:isotopy-null}
Let $\{U_i\}$ be a finite open cover of $M$,
and let $Γ,Γ'$ be $\cA$-coloured graphs in $M$ with boundary value $\XX$,
and suppose they are related by a graph isotopy $Γ^{\bullet}$
with $Γ = Γ⁰, Γ' = Γ¹$, that is $C^{∞}$-constant on the boundary $∂Γ$.
Then there exists a sequence of graph isotopies
$Γ_j^{\bullet}$, $j = 1,...,k$,
with $Γ = Γ₁⁰, Γ_j¹ = Γ_{j+1}⁰, Γ_k¹ = Γ'$,
that are $C^{∞}$-constant on the boundary,
each $Γ_j^{\bullet}$ is supported in a ball $D_j$
contained in some $U_{i_j}$, and $Γ_j^t$ is transverse to $∂D_j$ for all $t \in [0,1]$.

Moreover, suppose that $Γ$ is $\cS$-admissible.
Then each $Γ_j^{\bullet}$ and $D_j$ can be chosen so that
$Γ_j^t \backslash D_j$ is in addition
$\cS$-admissible for all $t \in [0,1]$.
\end{lemma}

We give the proof in Appendix~\ref{s:appendix-isotopy-null}.
Since the evaluation in balls is isotopy invariant
by the properties of $\eval{⋅}_D$ as listed in Section~\ref{s:alg-input},
this shows that ($\cS$-admissible) skeins are isotopy invariant.

\medskip

An inclusion $ι: N ↪ N'$ of $(d-1)$-manifolds
sends a boundary value $\VV$ on $N$ to a boundary value on $N'$,
which we denote by $ι_*(\VV)$.
Let $φ: M \to M'$ be an inclusion of $d$-manifolds, such that every component of $M'$ is in the image of $\varphi$.
Let $\VV$ be a boundary value on $∂M$.
Suppose for each boundary piece $N_i ⊆ ∂M$ such that $\VV|_{N_i}$ is non-empty,
its image $φ(N_i)$ is contained in $∂M'$;
then it makes sense to define $φ_*(\VV) = \bigcup (φ|_{N_i})_*(\VV|_{N_i})$,
the union taken over such boundary pieces $N_i$,
and the image of any $Γ ∈ \GraphX(M;\VV)$ under $φ$ is a graph
$φ(Γ) ∈ \GraphX(M';φ_*(\VV))$.

Extending $R$-linearly, we get a map
$φ_* : \VGraphX(M;\VV) \to \VGraphX(M';φ_*(\VV))$.
It clearly preserves primitive $\cS$-null graphs,
so $φ_*(\NullX(M;\VV)) ⊆ \NullX(M';φ_*(\VV))$,
hence $φ_*$ descends to a map
\[
φ_* : \SNX(M;\VV) \to \SNX(M';φ_*(\VV))
\]

The following lemma says that
we can essentially ignore corner structure and most of the boundary:
\begin{lemma}
\label{l:ignore-corners}
Let $M$ be a $d$-manifold, and $\VV$ a boundary value on $∂M$.
Let $φ: M' ↪ M$ be a submanifold containing the interior of $M$
and some neighbourhood of $\VV$,
so we may regard $\VV$ as a boundary value on $∂M'$.
Then maps $φ_*$ induced by the inclusion are isomorphisms:
\begin{align*}
φ_* : \GraphX(M';\VV) &\simeq \GraphX(M;\VV)
\\
φ_* : \VGraphX(M';\VV) &\simeq \VGraphX(M;\VV)
\\
φ_*(\NullX(M';\VV)) &= \NullX(M;\VV)
\\
φ_* : \SNX(M';\VV) &\simeq \SNX(M;\VV)
\end{align*}
\end{lemma}
\begin{proof}
Since graphs are not allowed to touch the boundary again away from
its boundary value, it is clear that $φ$ induces a bijection on $\GraphX$,
and hence an isomorphism on $\VGraphX$.
Balls in $M$ that do not touch the boundary certainly can be considered
as balls in $M'$, so give the same $\cS$-null relations;
balls $D ⊆ M$ that do touch the boundary,
or more precisely, $D$ not contained in $M'$,
do not give any new null relations: again, since graphs avoid the boundary
except at their boundary value, if $D$ gives an $\cS$-null relation
involving graphs $Γ₁,...Γ_k$, then we may shrink $D$ slightly away from
the boundary, and now we have a ball $D'$ that gives
the same $\cS$-null relation.
\end{proof}

In particular, we can take $M'$ to be $M$ without its corners, and hence
we will not be too careful with corners and boundaries of $d$-manifolds;
for example, when we refer to a $d$-ball, we really mean a ball
with any possible boundary/corner structure.
This will be useful for example in the context of handle attachments
as used in the proof of Proposition~\ref{p:totally-I} below.

We say a $d$-manifold $M$ is \emph{finitary}
if it can be built by attaching a finite number of handles to a finite number of balls;
loosely speaking, it can be obtained by performing
a finite number of gluing operations on a finite number of balls.
Skeins are built from finite graphs, and
since each graph and each ball which is part of a null relation is
contained in some finitary submanifold $M' ⊆ M$,
the skein module for $M$ is determined by its finitary submanifolds.
This establishes the following lemma:

\begin{lemma}
\label{l:finitary-colimit}
Let $M$ be a $d$-manifold, with boundary value $\VV$ on $∂M$.
Let $N ⊆ ∂M$ be the union of boundary pieces $N_i ⊆ ∂M$
on which $\VV$ is non-empty.
Consider the collection of finitary submanifolds
\[
\cM^{fin} := \{M' ∈ \cM \;|\; N ⊆ ∂M'  \text{ and } M' \text{ is finitary} \} \ ,
\]
partially-ordered by inclusion.
We may naturally treat $\VV$ as a boundary value on $∂M'$ for $M' ∈ \cM$, so that
the inclusions induce maps $(φ_{M' ⊆ M''})_*: \SNX(M';\VV) \to \SNX(M'';\VV)$.
Then the induced maps $(φ_{M' ⊆ M})_* : \SNX(M';\VV) \to \SNX(M;\VV)$
descend to an isomorphism out of the colimit of $\SNX(M';\VV)$
over $\cM^{fin}$:
\[
\displaystyle \underset{M' ∈ \cM^{fin}}{\Colim} \SNX(M';\VV) \simeq \SNX(M;\VV) \ .
\]
\end{lemma}

$\cS$-admissible skein modules also behave well under disjoint union.
Consider $d$-manifolds $M₁, M₂$.
Taking the disjoint of graphs in $M₁$ and $M₂$ defines a bijection
\begin{align}
\label{e:disjoint-union-graph}
\begin{split}
\GraphX(M₁;\VV₁) \times \GraphX(M₂;\VV₂) &\simeq \GraphX(M₁ \sqcup M₂;\VV)
\\
(Γ₁,Γ₂) &\mapsto Γ₁ \sqcup Γ₂
\end{split}
\end{align}
which induces an isomorphism of $R$-modules
\begin{align}
\label{e:disjoint-union-vgraph}
\begin{split}
\VGraphX(M₁;\VV₁) ⊗_R \VGraphX(M₂;\VV₂)
&\simeq
\VGraphX(M₁ \sqcup M₂;\VV)
\\
Γ₁ ⊗_R Γ₂ &\mapsto Γ₁ \sqcup Γ₂
\end{split}
\end{align}
This in turn induces an isomorphism on the level of $\SNX$:

\begin{lemma}
\label{l:M-disjoint-union}
For $d$-manifolds $M₁, M₂$ with boundary values $\VV₁,\VV₂$ respectively,
\[
\SNX(M₁; \VV₁) ⊗_R \SNX(M₂;\VV₂) 
~\simeq~
\SNX(M₁ \sqcup M₂; \VV₁ \sqcup \VV₂) 
\ .
\]
\end{lemma}
\begin{proof}
The tensor product of skein modules is the quotient of the left hand side in
\eqref{e:disjoint-union-vgraph} by the subspace
\[
\NullX(M₁;\VV₁) ⊗_R \VGraphX(M₂;\VV₂) + \VGraphX(M₁;\VV₁) ⊗_R \NullX(M₂;\VV₂) \ .
\]
Write $M := M₁ \sqcup M₂, \VV = \VV₁ \sqcup \VV₂$.
We need to show that the above subspace is the image of the null relations on $M$ under the isomorphism \eqref{e:disjoint-union-vgraph}.
Since any ball $D ⊆ M$ is in either $M₁$ or $M₂$,
\[
\PrimNullX(M;\VV) =
\PrimNullX(M,M₁;\VV) ∪ \PrimNullX(M,M₂;\VV)
\]
where $M₁,M₂$ are treated as subsets of $M$, and hence
\[
\NullX(M;\VV) =
\NullX(M,M₁;\VV) + \NullX(M,M₂;\VV)
\ .
\]
It is then easy to see that under the isomorphism \eqref{e:disjoint-union-vgraph},
\begin{align*}
\NullX(M,M₁;\VV) &= \NullX(M₁;\VV₁) ⊗_R \VGraphX(M₂;\VV₂) \ ,
\\
\NullX(M,M₂;\VV) &= \VGraphX(M₁;\VV₁) ⊗_R \NullX(M₂;\VV₂)
\ .
\end{align*}
\end{proof}

Next we consider some properties related to choices of $\cS ⊆ \cA$,
in particular, we discuss some situations when $\SNX$ agrees with
the usual skein module $\SNJ{}$:

\begin{lemma}
\label{l:S=A}
For $\cS = \cA$ and for any boundary value $\VV$,
$\ov{ι}_{\cA}: \SNJ{\cA}(M;\VV) \simeq \SN(M;\VV)$.
\end{lemma}
\begin{proof}
There is nothing to prove for $d = 1$; for $d ≥ 2$, 
one just has to introduce $\one$-labelled trivial loops.
We omit the details.
\end{proof}

When we impose $\cS$-admissibility on the boundary value,
we have the stronger result:

\begin{lemma}
\label{l:bulk-cS}
Let $\VV$ be a bulk-$\cS$-admissible boundary value on $∂M$.
Then for any $\cT ⊆ \cA$ containing $\cS$ and any $M$,
\[
 \ov{ι}_{\cS ⊆ \cT} : \SNX(M;\VV) \to \SNXp(M;\VV)
\]
is an isomorphism.
\end{lemma}
\begin{proof}
By Lemma \ref{l:M-disjoint-union}, it suffices to show this for $M$ connected.
The bulk-$\cS$-admissibility of $\VV$ implies that any
$Γ ∈ \Graph(M;\VV)$ is in fact also $\cS$-admissible,
so the natural inclusion $\GraphX(M;\VV) ⊆ \GraphXp(M;\VV)$ is a bijection.
In regard to the null relations,
let $Σ a_j Γ_j$ be a primitive $\cT$-null graph with respect to $D$,
and consider a point $b ∈ \VV$ with label in $\cS$.
If $b \nin D$, then clearly $Σ a_j Γ_j$ is also $\cS$-null.
If $b ∈ D$, we isotope each $Γ_j$ to some new $Γ_j'$ that is unchanged outside $D$,
so that all $Γ_j'$ agree in a neighbourhood of $b$.
By graph isotopy invariance (Lemma~\ref{l:isotopy-null}),
$Γ_j - Γ_j' ∈ \NullX(M;\VV)$.
Then we may deform $D$ away from $b$,
so that $Σ a_j Γ_j'$ is primitive $\cS$-null with respect to the deformed $D$,
hence $Σ a_j Γ_j = Σ a_j Γ_j' + Σ a_j (Γ_j - Γ_j')$ is $\cS$-null.
\end{proof}

In particular, setting $\cT = \cA$, we get 
$\SNX(\DD¹;V,W) \simeq \SN(\DD¹;V,W) \simeq \cA(V,W)$
if either $V$ or $W$ is in $\cS$,
and for $d ≥ 2$, $\SNX(\DD^d;\VV) \simeq \SN(\DD^d;\VV)
\simeq \cA(\one,V₁⊗\cdots⊗ V_k)$,
where $\VV$ is an $\cS$-admissible
boundary value on $∂\DD^d$ with labels $V₁,...,V_k$ and
all marked points going outwards
(we take the dual for marked points going inwards), as in the last of the basic properties of the evaluation $\eval{⋅}_D$ in Section~\ref{s:adm-skein-defn}.

\medskip

Finally, let us give a simple relation between skein modules for manifolds of different dimensions. First note that a choice of $\cA$ which is suitable in dimension $d$ is also suitable for any dimension $1 \le k < d$, and so for a full subcategory $\cS ⊆ \cA$ it makes sense to talk about $\cS$-admissible skein modules of $k$-manifolds $M^k$. Then $\cS$-admissible skein modules for $M^k$ and for $M^k \times \DD^{d-k}$ agree in the following sense:

\begin{lemma}
\label{l:thickened-mfld}
Let $\cA$ be a category suitable for $d$-dimensional skein modules
(Table~\ref{t:A-structure-vs-d}),
and let $\cS ⊆ \cA$ be a full subcategory. For $1 \le k < d$ consider a k-dimensional manifold $M^k$ and set $M := M^k \times \DD^{d-k}$. The inclusion $ι : M^k \simeq M^k \times \{0\} ↪ M$
induces inclusions of graphs and boundary values
into $M$ and $∂M$ respectively, where the framing is extended by the constant standard framing on $\DD^{d-k}$.
Then this inclusion induces an isomorphism
\[
ι_* : \SNX(M^k;\XX) \to \SNX(M;ι_*(\XX)) \ .
\]
\end{lemma}

\begin{proof}
The proof is standard and amounts to defining an inverse map by using the properties of $\eval{⋅}_D$ to replace a graph in $M$ by a graph in $M^k \times \{0\}$ which represents the same element in $\SNX(M;ι_*(\XX))$. We omit the details.
\end{proof}

%++%
\subsection{Cylinder Categories}
\label{s:cylinder-cat}

We define cylinder categories based on $\cS$-admissible skeins,
analogous to the skein categories of
\cite{Walker:2006,JohnsonFreyd:2015}.
Let $N$ be a $(d-1)$-manifold as in the previous section, and
let $\BVX(N)$ be the set of $\cS$-admissible boundary values on $N$.

Consider the $d$-manifold $M = [0,1] \times N$,
and let the identifications of $N$ with the bottom (incoming)
and top (outgoing) boundaries be denoted by
$ι₀: N \simeq \{0\} \times N ↪ [0,1] \times N$
and $ι₁: N \simeq \{1\} \times N ↪ [0,1] \times N$.
Then we define the assignment
$\HomT_\cS: \BVX(N) \times \BVX(N) \to \Rmod$, by
\[
\HomT_\cS(\VV, \WW) = \VGraphX(N \times [0,1]; \VV, \WW)
\]
where $\VV$ is placed at $N \times 0$, and $\WW$ at $N \times 1$.
Stacking and shrinking cylinders, we get a notion of composition,
\[
∘: \HomT_\cS(\WW, \XX) \times \HomT_\cS(\VV, \WW)
	\to \HomT_\cS(\VV, \XX)
\]
which is $R$-bilinear, but this operation is not associative,
only associative up to isotopy ($A_{∞}$-associative).
By Lemma \ref{l:isotopy-null}, $\cS$-null relations include isotopies,
hence quotienting by $\cS$-null relations makes the composition associative
and independent of the precise stacking and shrinking operation.
\begin{definition}
The \emph{$\cS$-admissible cylinder category over $N$} is
\[
\hZX(N) :=
\begin{cases}
\Obj = \BVX(N)
\\
\Hom(\VV,\WW) = \SNX(N \times [0,1]; \VV, \WW)
\end{cases}
\]
\end{definition}

For $N = \DD^{d-1}$, fixing some marked point $p ∈ \DD^{d-1}$,
we get a functor $\cS \to \hZX(\DD^{d-1})$
sending an object $V ∈ \cS$ to the boundary value $(\{p\},\{V\})$.
Since we are looking at the special case of a cylinder category over a disc, by
Lemma \ref{l:bulk-cS} and the remark after it, this functor is fully faithful.
When $\cS$ is closed under certain operations,
this functor is an equivalence (see Section \ref{s:closure});
in particular, for $\cS = \cA$, we have $\hZJ{\cA}(\DD^{d-1}) \simeq \cA$.

Forgetting the $\cS$-admissibility conditions on boundary values
and skein modules, we can also consider the category $\hZ(N)$.
Compared with $\hZJ{\cA}(N)$, the only new objects are those boundary values
that have some component of $N$ with no marked points;
such objects, in case $d ≥ 2$, are equivalent to
a boundary value with $\one$-labelled marked points,
and in case $d = 1$, are ruled out by definition.
With Lemma~\ref{l:S=A}, this means that the inclusion of boundary values
induces an equivalence $\hZJ{\cA}(N) \simeq \hZ(N)$.

More generally, by Lemma \ref{l:bulk-cS},
for full subcategories $\cS ⊆ \cT$,
the inclusion of $\cS$-admissible boundary values
into $\cT$-admissible boundary values induces a fully faithful functor
\begin{equation}
\label{e:iSSp}
\ov{ι}_{\cS ⊆ \cT}: \hZX(N) \to \hZXp(N) \ .
\end{equation}

The cylinder category of the disjoint union
is an $R$-linear cartesian product,
as defined below:
\begin{definition}
\label{d:cartesian-tnsr}
For $\Rmod$-enriched categories $\cA,\cB$,
define the category $\cA \hatbox \cB$
to have objects given by pairs of objects from $\cA$ and $\cB$,
$\Obj \cA \hatbox \cB = \Obj \cA \times \Obj \cB$,
and morphisms given by the tensor product:
\[
\Hom_{\cA \hatbox \cB}((A,B),(A',B'))
= \Hom_\cA(A,A') ⊗_R \Hom_\cB(B,B')
\]
where composition is just the tensor product
of the appropriate composition maps.
\end{definition}

This construction is not very well-behaved;
for example, it does not preserve additivity as $(A₁,B₁) ⊕ (A₂,B₂)$
is not necessarily found in $\cA \hatbox \cB$
(though if we Karoubi complete $\cA \hatbox \cB$,
it would be a subobject of $(A₁⊕ A₂,B₁⊕ B₂)$).

\begin{lemma}
\label{l:union-hatbox}
For $N = N₁ \sqcup N₂$, we have isomorphisms
\[
\hZX(N) \cong \hZX(N₁) \hatbox \hZX(N₂)
\]
\end{lemma}
\begin{proof}
An $\cS$-admissible boundary value on $N$ can naturally be considered as
a pair of a boundary value on $N₁$ and a boundary value on $N₂$,
and vice versa, so we have bijection of objects.
Fully faithfulness follows immediately from Lemma \ref{l:M-disjoint-union}.
\end{proof}

%++%
\subsection{Skein Modules as Functors from Cylinder Categories}
\label{s:skein-module-as-functor}

The gluing operation that defines the composition in the cylinder categories
can be generalised to manifolds with a collar neighbourhood.
Let $M$ be a $d$-manifold, and let $ι: N ↪ ∂M$
be an outgoing embedded boundary piece with some collar neighbourhood
$\wdtld{ι}: [-1,0] \times N ↪ M$.
Consider the operation of gluing $[0,1] \times N$ to $M$ along
$\{0\} \times N \sim N$, followed by a diffeomorphism
$[0,1] \times N ∪_{\{0\} \times N \sim N} M \simeq M$
that is fixed point-wise away from the collar neighbourhood of $N$.

Let $\cT ⊆ \cA$ be a full subcategory containing $\cS$.
Let $\VV,\WW$ be $\cT$-admissible boundary values,
and let $\XX$ be any boundary value on $∂M$ that is empty on the
collar neighbourhood of $N$.
For an $\cS$-admissible graph $Γ$ in $M$ with boundary value $\XX ∪ \VV$,
and $\cT$-admissible graph $Γ'$ in $[0,1] \times N$
with boundary value $\VV ∪ \WW$,
their union $Γ' ∪_{\VV} Γ$ in $M$, obtained from the gluing operation above,
is also an $\cS$-admissible graph.
Moreover, for an $\cT$-null graph $Σ a_j Γ_j'$ in $[0,1] \times N$,
the sum $Σ a_j Γ_j' ∪ Γ$ is also an $\cS$-null graph;
similarly for $\cS$-null graph in $M$.
Thus, we have an $R$-bilinear map
\[
\lact : \hZXp(N)(\VV,\WW) \times \SNX(M;\XX,\VV) \to
\SNX([0,1] \times N ∪_{N \sim N \times \{0\}} M)
\simeq \SNX(M;\XX,\WW) \ ,
\]
which, by isotopy invariance (Lemma \ref{l:isotopy-null}), is associative from the left,
i.e.\ $(g ∘ f) \lact Γ = g \lact (f \lact Γ)$,
and does not depend on the choice of collar neighbourhood $ι$.
In other words, we have a functor
\[
\SNX(M; \XX, -) : \hZXp(N) \to \Rmod
\]
for any boundary value $\XX$ on $∂M \backslash N$
(one can take the collar neighbourhoods to be arbitrarily small
so as to avoid $\XX$).
It is easy to see that $\lact$ is compatible with $\ov{ι}_{\cT ⊆ \cU}$
($\cU ⊆ \cA$ is some full subcategory containing $\cT$):
\begin{equation}
\label{e:SNX-diff-Z}
\SNX(M;\XX,-) ∘ \ov{ι}_{\cT ⊆ \cU} = \SNX(M;\XX,-) : \hZXp(N) \to \Rmod
\end{equation}
hence there is no ambiguity in the notation $\SNX(M;\XX,-)$.

Similarly, if $N$ is an incoming boundary,
we get $R$-bilinear map
\[
\ract : \SNX(M;\XX,\WW) \times \hZXp(N)(\VV,\WW) \to
\SNX(M ∪_{N \sim \{1\} \times N} [0,1] \times N)
\simeq \SNX(M;\XX,\VV)
\]
which is associative from the right,
and we have a contravariant functor $\hZXp(N) \to \Rmod$:
\[
\SNX(M; \XX, -) : (\hZXp(N))^\op \to \Rmod
\]

We can consider multiple boundary pieces at once,
outgoing or incoming, and not necessarily disjoint
(though in our main use case, in Theorem \ref{t:excision},
the relevant boundary pieces must be pairwise disjoint).
It is not hard to see that the actions coming from different pieces
pairwise commute, so that, if we have
$k$ outgoing boundary pieces $N₁,...,N_k$
and $l$ incoming boundary pieces $N₁',...,N_l'$,
then we have an $R$-multilinear functor
\[
\SNX(M;\XX,-,\cdots,-):
	\hZXp(N₁')^\op \times \cdots \times \hZXp(N_l')^\op \times
	\hZXp(N₁) \times \cdots \times \hZXp(N_k) \to \Rmod
\]

It is straightforward to check that this functor is natural
with respect to $\cS$, in the sense that
given another full subcategory $\cS' ⊆ \cS$,
$\ov{ι}_{\cS' ⊆ \cS} : \SNJ{\cS'}(M;\XX,-,\cdots,-) \to \SNX(M;\XX,-\cdots,-)$
is a natural transformation;
in particular, Lemma~\ref{l:bulk-cS} can upgraded as follows:

\begin{lemma}
\label{l:bulk-cS-natural}
Let $\cS ⊆ \cT ⊆ \cA$ be full subcategories,
and let $M$ be a $d$-manifold.
Let $\hZX^b(∂M) ⊆ \hZXp(∂M)$ be the full subcategory consisting of
bulk-$\cS$-admissible boundary values on $∂M$.
Then $\ov{ι}_{\cS ⊆ \cT} : \SNX(M;-)|_{\hZX^b(∂M)}
\simeq \SNXp(M;-)|_{\hZX^b(∂M)}$ is a natural isomorphism.
\end{lemma}

%++%
\section{Excision}
\label{s:excision}

\subsection{Main Theorem}
\label{s:theorem-setup}

Let $M$ be a $d$-manifold, and let $ι: L ↪ M$ be a
codimension-1 oriented submanifold of $M$
with finitely many components,
which admits an (orientation-preserving)
tubular neighbourhood
$ι: [-1,1] \times L ↪ M$;
in particular, the (possibly empty)
intersection of $L$ with $∂M$
is transverse, and $L ∩ ∂M = ∂L$.
Let $M'$ be the $d$-manifold obtained by cutting $M$ along $L$,
as described in Section~\ref{s:topol-setup}
(and Appendix~\ref{s:mfld-corner})
with natural quotient map $π: M' \to M$.
We rename the preimages of $L$ as
$N = L_-$ (outgoing) and $N' = L_+$ (incoming),
which consist of pairwise disjoint embedded boundary pieces,
and the tubular neighbourhood of $L$ lifts to collar neighbourhoods of $N,N'$.

Let $\XX$ be some boundary value on $∂M' \backslash (N ∪ N')$;
we will also denote $π_*(\XX)$ on $∂M \backslash L$ by $\XX$.
Using a smaller tubular neighbourhood of $L$ if necessary,
we may assume that $\XX$ is disjoint from the collar neighbourhoods of $N,N'$.
We will implicitly identify boundary values on $N, L, N'$
with their images under the diffeomorphisms
$π|_N : N \simeq L \simeq N' : π|_{N'}$.

Let $\cT ⊆ \cA$ be a full subcategory containing $\cS$.
Recall from Section \ref{s:skein-module-as-functor} the functors
$\lact$ and $\ract$ from $\hZXp(N)$ and $\hZXp(N')^\op$ to $\Rmod$, respectively.
Using that $N \simeq L \simeq N'$, these combine into a bifunctor
\begin{equation}
\SNX(M';\XX,-,-) : \hZXp(N)^\op \times \hZXp(N) \to \Rmod
\ .
\end{equation}

The coend $\int^{\hZXp(N)} \SNX(M';\XX,-,-)$ exists and can be
be defined as the cokernel of
\begin{align}
\label{eq:skein-coend-def}
\begin{split}
\bigoplus_{\VV,\WW ∈ \hZXp(N)} \bigoplus_{f : \VV \to \WW} \SNX(M';\XX,\WW,\VV)
&~\to~
\bigoplus_{\VV} \SNX(M';\XX,\VV,\VV) \ ,
\\
Ψ \quad &~\mapsto ~~f \lact Ψ - Ψ \ract f \ .
\end{split}
\end{align}

The map $π : M' \to M$ glues up an $\cA$-coloured graph
in $\Graph(M';\XX,\VV,\VV)$ along $\VV$,
producing a graph in $M$, which gives us a map
\begin{equation}
\label{e:pi-graph}
π_*: \VGraphX(M';\XX,\VV,\VV) \to \VGraphX(M;\XX) \ .
\end{equation}
This map descends to skeins, and we denote the resulting map also by $π_*$,
\begin{equation}
\label{e:pi-skein}
π_*: \SNX(M';\XX,\VV,\VV) \to \SNX(M;\XX) \ .
\end{equation}
Using the isotopy invariance, it is easy to see from Lemma \ref{l:isotopy-null}
that $f \lact Ψ$ and $Ψ \ract f$ are mapped to the same skein under $π_*$,
hence $π_*$ defines a dinatural transformation
\begin{equation}
\label{eq:Sk-dinat}
 π_*: \SNX(M';\XX,-,-) \xrightarrow{∙} \SNX(M;\XX)
\end{equation}
of functors $\hZXp(N)^\op \times \hZXp(N) \to \Rmod$.
Equivalently, $π_*$ factors through the coend via a universal map
\begin{equation}
\label{eq:pi*-coend-to-skein}
\hat{π}_* : \int^{\hZXp(N)} \SNX(M';\XX,-,-) \to \SNX(M;\XX) \ .
\end{equation}
With these preparations, we can state the main result of this paper:

\begin{theorem}
\label{t:excision}
Let $\cA$ be as in Table~\ref{t:A-structure-vs-d}, and
let $\cS ⊆ \cT ⊆ \cA$ be full subcategories.
Let $M$ be a $d$-manifold, and let $L ⊆ M$ be a codimension-1 submanifold
with finitely many components that admits a tubular neighboorhood.
Let $M'$ be the manifold obtained by cutting $M$ along $L$,
with natural quotient map $π: M' \to M$
identifying parts of the boundary $N, N' ⊆ ∂M'$ with $L$.
Then the corresponding universal map from the coend is an isomorphism:
\[
\hat{\pi}_* : \int^{\hZXp(N)} \SNX(M';\XX,-,-) ~\simeq~ \SNX(M;\XX)
\]
\end{theorem}

The isomorphism is compatible across different $\cT$:
\begin{corollary}
\label{c:excision}
For $\cS ⊆ \cT ⊆ \cU ⊆ \cA$ full subcategories,
the functor $\ov{ι}_{\cT ⊆ \cU} : \hZXp(N) \to \hZJ{\cU}(N)$
in \eqref{e:iSSp}
is fully faithful, hence induces a map of coends
\[
\int^{\hZXp(N)} \SNX(M';\XX,-,-) \to \int^{\hZJ{\cU}(N)} \SNX(M';\XX,-,-) \ .
\]
This map is naturally compatible with $\hat{π}_*$,
\[
\begin{tikzcd}
\int^{\hZXp(N)} \SNX(M';\XX,-,-) \ar[r] \ar[rr, bend right=10, "\hat{π}_*"']
& \int^{\hZJ{\cU}(N)} \SNX(M';\XX,-,-) \ar[r, "\hat{π}_*"]
& \SNX(M;\XX)~,
\end{tikzcd}
\]
and hence is also an isomorphism.
\end{corollary}

The rest of Section~\ref{s:excision} will be dedicated to the proof of
Theorem \ref{t:excision}.

%++%
\subsection{Proof of Theorem \ref{t:excision}}
\label{s:excision-proof}

Our approach will be similar to \cite{AlexanderKirillov:2020hkv},
but we are more methodical about keeping arguments
on the level of graphs rather than on the level of skeins,
and we keep the observations about $\SNX(M';\XX,-,-)$ and $\SNX(M;\XX)$
separate until the very end where we combine them.
Note that in \cite{AlexanderKirillov:2020hkv},
the categories are assumed to be semisimple, but upon close inspection,
the proof of \cite[Thm.\,7.4]{AlexanderKirillov:2020hkv}
(excision for skein categories) does not rely on semisimplicity.

First we need to get a better understanding of the coend over skein modules.
We show that an element of the coend is a graph in $M'$
modulo the $\cS$-null relations in $M'$
and relations that ``carry graphs across $N$'',
which is encapsulated in (\ref{e:kernel-NS}) below.

To make the diagrams below less bulky,
we shorten $\VGraph$ to $\VG$, and further
\begin{align*}
\VGX'(\VV,\WW) &:= \VGraphX(M';\XX, N' = \VV, N = \WW) \ ,
\end{align*}
and similarly for $\SNX'(\VV,\WW), \NullX'(\VV,\WW)$.
Consider the commutative diagram below,
which we explain in the numbered list following it:

\begin{equation}
\label{e:big-cd-N-I}
\begin{tikzcd}
0 \ar[rd]
	&
	& 0
	\ar[d]
\\
	& \cN_\cS
	\ar[rd,"ι","(4)"']
	& \displaystyle \bigoplus_{\VV \in \hZXp(N)} \hspace{-.7em} \NullX'(\VV,\VV)
	\ar[d,"i","(3)"']
\\
0 \ar[r]
	& V\cK_\cS
	\ar[r, "(2)"'] \ar[d, "(3)"']
	& \displaystyle \bigoplus_{\VV \in \hZXp(N)} \hspace{-.7em} \VGX'(\VV,\VV)
	\ar[r, "q", "(2)"'] \ar[d,"p'","(3)"'] \ar[rd,"\ov{q} \circ p'","(4)"']
	& \displaystyle \mathrm{I}\VGX'
	\ar[r] \ar[d, "\ov{p'}","(3)"']
	& 0
\\
0 \ar[r]
	& \ker(\ov{q})
	\ar[r, "(1)"'] \ar[d]
	& \displaystyle \bigoplus_{\VV \in \hZXp(N)} \hspace{-.7em} \SNX'(\VV,\VV)
	\ar[r, "\ov{q}", "(1)"'] \ar[d]
	& \displaystyle \int^{\hZXp(N)} \SNX'(-,-)
	\ar[r] \ar[d] \ar[rd]
	& 0
\\
	& 0
	& 0
	& 0
	& 0
\end{tikzcd}
\end{equation}

\begin{enumerate}
\item The bottom row is the realisation of the
coend as a quotient of a direct sum.
	$\ker(\ov{q}) = \Span\{ f \lact Ψ - Ψ \ract f \}$
	where $f, Ψ$ range over all $f: \WW → \VV, Ψ ∈ \SNX'(\VV,\WW)$,
	and $\VV,\WW ∈ \hZXp(N)$ as in \eqref{eq:skein-coend-def}.

\item In the middle row,
define $V\cK_\cS := \Span\{ f \lact Ψ - Ψ \ract f \}$,
with $Ψ ∈ \VGX(\WW,\VV), f: \VV \to \WW$;
we define $\mathrm{I}\VGX'$
to be the quotient of $\bigoplus \VGX'(\VV,\VV)$ by $V\cK_\cS$.
We note that $\mathrm{I}\VGX'$ is not a coend since there is no dinatural transformation
analogous to \eqref{eq:Sk-dinat} for $\VGraphX$ instead of $\SNX$.

\item The maps $i,p'$ in the middle column are simply the defining maps
for $\cS$-admissible skeins.
The leftmost vertical arrow is the restriction 
of $p'$ to $V\cK_\cS$,
which clearly maps to $\ker(\ov{q})$ surjectively.
As a consequence, $p'$ also descends to $\mathrm{I}\VGX'$,
giving the rightmost vertical arrow $\ov{p'} : \mathrm{I}\VGX' \to \int \SNX'$.

\item The right square commutes; we define
	$\cN_\cS := \ker(\ov{q} ∘ p') = \ker(\ov{p'} ∘ q)$.
\end{enumerate}
By diagram chasing, one finds
\begin{equation}
\label{e:kernel-NS}
 \cN_\cS ~=~ V\cK_\cS \,+ \!\!\! \bigoplus_{\VV \in \hZXp(N)} \!\!\!\! \NullX'(\VV,\VV)
 ~ \subseteq ~ \bigoplus_{\VV} \VGX'(\VV,\VV) \ .
\end{equation}
Namely, for $x ∈ \cN_\cS$, $\ov{q}(p'(x)) = 0$ so $p'(x)$ lies in $\ker(\ov{q})$.
By surjectivity of the leftmost vertical arrow $p'|_{V\cK_\cS}$,
$p'(x)$ has a lift $y ∈ V\cK_\cS$.
Then $p'(x - y) = 0$, so $x - y$ is in the image of $i$ and we can write
$x = y + i(z)$ for some
$z ∈ \bigoplus_{\VV \in \hZXp(N)} \NullX'(\VV,\VV)$.

Next, we make a few more definitions and state Lemma \ref{l:cT-s}
which is the counterpart to (\ref{e:kernel-NS}) for skeins in $M$.
Consider the following map
which collects all the gluing operations $π_*$
induced by the gluing map $π: M' \to M$ into a single map
\[
π_*: \bigsqcup_{\VV ∈ \hZXp(N)} \GraphX(M';\XX,\VV,\VV) \to \GraphX(M;\XX) \ .
\]
If we allowed any boundary value $\VV$ (in the subscript under $\bigsqcup$),
then this map is almost surjective -
it misses precisely those graphs in $M$ which are
not transverse to $L$, the image of $N$ (and $N'$) under $π$.
The restriction to boundary values $\VV ∈ \hZXp(N)$
means that the graphs in the image of $π_*$
must have an intersection with $L$ that is a $\cT$-admissible boundary value.
This motivates the following definitions:
\begin{align*}
\GraphX^{\trvs}(M;\XX)
	&:= \big\{ Γ ∈ \GraphX(M;\XX) \,\big|\,
	Γ \trv L \text{ and } Γ ∩ L \text{ is } \cT \text{-admissible} \big\}
\\
\VGraphX^{\trvs}(M;\XX) &:= \Span\big( \GraphXs(M;\XX)\big) ~⊂~ \VGraphX(M;\XX)
\\
\PrimNullX^{\trvs}(M,U;\XX) &:= \PrimNullX(M,U;\XX) ∩ \VGraphXs(M;\XX)
\\
\NullX^{\trvs}(M,U;\XX) &:= \Span\big( \PrimNullX^{\trvs}(M,U;\XX)\big)
	~⊆~ \NullX ∩ \VGraphXs
\end{align*}

From the discussion above, the gluing map $π_*$ defines an isomorphism
\begin{equation}
\label{e:pi-bijective}
π_* : \bigoplus_{\VV ∈ \hZXp(N)} \VGX'(\VV,\VV) \simeq \VGraphX^{\trvs}(M;\XX)
\end{equation}

We make precise the idea of ``translation across $L$'';
this is the counterpart to taking the coend over boundary values on $N$,
which might be intuitively thought of as allowing the operation of
taking a chunk of a graph near $N$ and gluing it to $N'$.
Recall the tubular neighbourhood $ι: [-1,1] \times L \to M$ of $L$.
Let $U₀ := M \backslash L$ and $U₁ = ι( (-1,1) \times L)$.
Let $X = ι_*( (f ⊗ \id_L) ⋅ ∂_t)$,
where $∂_t$ is the constant vector field on $[-1,1] \times L$
in the $[-1,1]$ direction,
and $f$ is a smooth function on $\RR$ with support on $(-1,1)$,
Let $U_{1m} = ι( (-1,1) \times L_m)$ be a neighbourhood of
the $m$-th component $L_m$ of $L$,
and let $X_m$ be the restriction of $X$ to $U_{1m}$.

We define $\{θ_m^{α} = \exp(α X_m) \;|\; α ∈ \RR\}$,
a 1-parameter subgroup of diffeomorphisms of $M$ generated by $X_m$.
Note that the $θ_m^{α}$ do not fix the boundary of $M$,
which might have the unintended effect of not preserving
a boundary value $\XX$ of $∂M$.
Given an $\XX$ that has no points on $∂L$
(which is always the case in our context,
as $\XX$ actually arises as $π_*(\XX)$ from a boundary value on $M'$),
we can rescale the $[-1,1]$ direction of the tubular neighbourhood of $L$,
so as to have $\XX$ not intersect the tubular neighbourhood $[-1,1] \times L$;
we will assume this to be the case in the following.

The key property of this subgroup of diffeomorphisms
is the ability to displace compact sets:
any compact set $V ⊆ U_{1m}$ can be displaced from a neighbourhood of $L_m$,
that is, for any $0 < ε < 1$, there exist $C$ such that for all $|α| > C$,
$θ_m^{α}(V) ∩ ι([-ε,ε] \times L_m) = ∅ $.

We want to show that null graphs of $M$ that have a
$\cT$-admissible intersection with $L$
are made of null graphs in $U₀ = M \backslash L$
plus something that ``moves things across $L$''.
Define $Θ_m^\trvs$ to be the subspace of $\VGraphX^{\trvs}(M;\XX)$
spanned by elements of the form $Γ- θ_m^{α}(Γ)$,
where both $Γ,θ_m^{α}(Γ)$ are transverse to $L$
and their intersection boundary values on $L$ are $\cT$-admissible,
and let $Θ^\trvs = \sum_{L_m ⊆ L} Θ_m^\trvs$.
Then the counterpart to (\ref{e:kernel-NS}) for skeins in $M$ is:

\begin{lemma}
\label{l:cT-s}
$\NullXs(M;\XX) = \NullXs(M,U₀;\XX) + Θ^\trvs$.
\end{lemma}

We give the proof in Appendix~\ref{s:appendix-cT-s}.

\medskip

%++%
%===========================================================================

Finally, we can put the observations about skeins in $M'$ and $M$ together.
We consider the following commutative diagram:
\begin{equation}
\label{e:big-cd-Phi-I}
\begin{tikzcd}
0 \ar[r]
	& \cN_\cS
	\ar[r,"ι"] \ar[d, "π_*|_{\cN_\cS}"]
	& \displaystyle \bigoplus_{\VV \in \hZXp(N)} \VGX'(\VV,\VV)
	\ar[r,"\ov{q}\circ p'"] \ar[d,"\simeq"',"π_*"]
	& \int^{\hZXp(N)} \SNX'(-,-)
	\ar[r] \ar[d,"\hat{π}_*"]
	& 0
\\
0 \ar[r]
	& \NullXs
	\ar[r] \ar[d, hookrightarrow]
	& \VGXs
	\ar[r,"p|_\trv"] \ar[d, hookrightarrow]
	& \SNX
	\ar[r]
	& 0
\\
0 \ar[r]
	& \NullX
	\ar[r]
	& \VGX
	\ar[ru, "p"']
\end{tikzcd}
\end{equation}
The middle vertical arrow $π_*$ is the bijection from (\ref{e:pi-bijective}),
and the rightmost vertical arrow $\hat{π}_*$ is the map
from \eqref{eq:pi*-coend-to-skein} that we want to prove is an isomorphism;
the square they form is commutative by definition.
This implies that $π_*(\cN_\cS) ⊆ \NullXs$,
justifying the target of the arrow labelled $π_*|_{\cN_\cS}$.
Therefore, since $π_*$ is bijective,
$π_*|_{\cN_\cS}$ is injective and $\hat{π}_*$ is surjective.

\medskip

It remains to show that $\hat{π}_*$ is injective,
or equivalently, that $π_*|_{\cN_\cS}$ is surjective.
Using \eqref{e:kernel-NS} and Lemma \ref{l:cT-s} to rewrite source and target,
we need to show that
    \[
π_*|_{\cN_\cS} ~:~ \cN_\cS \,=\!\!\!
	\bigoplus_{\VV \in \hZXp(N)} \!\!\! \NullX'(\VV,\VV) + V\cK_\cS
~\to~ \NullXs(U₀) + Θ^\trvs = \NullXs
\]
is surjective (we omit $M,\XX$ from the notation for brevity).

First observe that 
$\bigoplus_{\VV \in \hZXp(N)} \NullX'(\VV,\VV)$
is sent onto $\NullXs(U₀)$;
indeed, 
\[ π_
    *(\NullX'(\VV,\VV)) = \NullXs(U₀;L = \VV)
\]
for any $\VV ∈ \hZXp(N)$, so
\[
π_* \big( \bigoplus_{\VV ∈ \hZXp(N)} \NullX'(\VV,\VV) \big)
~= \!\!\bigoplus_{\VV ∈ \hZXp(N)}\!\!\! \NullXs(U₀;L = \VV)
\,=\, \NullXs(U₀) \ .
\]

For $Γ - θ^{α}(Γ) ∈ Θ^\trvs$, say for $α > 0$,
suppose $Γ ∩ L = \VV$, and $θ^{α}(Γ) ∩ L = \WW$.
Let $f$ be the part of $Γ$ that is moved across $L$,
i.e.\ $f = Γ ∩ ι([α',0] \times L)$ where $α'<0$ is such that $θ^{α}(\{α'\} \times L) = \{0\} \times L$;
as a skein, $f$ defines a morphism $\ov{f} ∈ \hZX(N)(\WW,\VV)$.
Let $Γ'$ be the rest of $Γ$,
i.e.\ $Γ' = Γ ∩ (M \backslash ι((α',0) \times L))$.
We can identify $M \backslash ι((α',0) \times L)$ with $M'$,
hence we think of $Γ'$ as a graph in $M'$,
with boundary values $\XX$, $N = \WW$, and $N' = \VV$,
i.e.\ $Γ' ∈ \VGX'(\VV,\WW)$.

It is easy to see that $π_*(f \lact Γ')$ is isotopic to $Γ$,
and $π_*(Γ' \ract f)$ is isotopic to $θ^{α}(Γ)$,
and these isotopies can be taken to be fixed on $L$, so
\[
π_*(f \lact Γ' - Γ' \ract f) - (Γ - θ^{α}(Γ))
\,∈\, \NullXs(U₀) \ .
\]
Since $f \lact Γ' - Γ' \ract f \in V\cK_\cS$ by definition, and
since we have already shown that $π_*|_{\cN_\cS}$ is surjective onto $\NullXs(U₀)$,
this implies that $Γ - θ^{α}(Γ)$ is in the image of $π_*|_{\cN_\cS}$.
Altogether we see that $π_*|_{\cN_\cS}$ is surjective.

This completes the proof of Theorem~\ref{t:excision}.

%++%
\section{Properties}
\label{s:properties}

In this section we establish a number of general properties of admissible skein modules. We show how skein modules with different admissibility conditions can be glued together; we give a version of excision for cylinder categories and relate it to the so-called horizontal trace of a bimodule category; we show invariance of admissible skein modules under two closure operations, the sum-retract closure and the tensor-dual closure; and finally we explain the relation to a version of $\cS$-admissible skeins where all non-boundary edges are labelled in $\cS$.

%++%
\subsection{Gluing Different Types of Skein Modules}
\label{s:different-skein}

\begin{proposition}
\label{p:different-skein}
Consider $M = M₁ ∪_N M₂$, such that every connected component of $M₂$
meets $M₁$ somewhere along $N$.
Let $\cS ⊆ \cT ⊆ \cA$ be full subcategories,
and $\XX$ some boundary value on $∂M \backslash N$,
with restrictions $\XX₁,\XX₂$ to $M₁,M₂$ respectively.
Then
\[
\SNX(M;\XX) \simeq
\hspace{-10pt} \int\displaylimits^{\VV ∈ \hZXp(N)} \hspace{-10pt}
\SNX(M₁;\XX₁,\VV) ⊗_R \SNXp(M₂;\XX₂,\VV)
\]
\end{proposition}
\begin{proof}
The proof simply amounts to moving the $\cS$-admissibility condition
between the boundary value and the skein module.
Let $C$ be a collar neighbourhood of $N$ in $M₁$,
let $M₁' = \overline{M₁ \backslash C}$,
and let $N' = M₁' ∩ C$ be the $(d-1)$-manifold separating them.
Then (see below for explanations):
{\allowdisplaybreaks
\begin{align*}
\int\displaylimits^{\VV ∈ \hZXp(N)} &\SNX(M₁;\XX₁,\VV) ⊗_R \SNXp(M₂;\XX₂,\VV)
\\
&\overset{(1)}{\simeq}
\int\displaylimits^{\VV ∈ \hZXp(N)}
\Big( \int\displaylimits^{\WW ∈ \hZX(N')}
\SNX(M₁';\WW) ⊗_R \SNX(C;\XX₁,\WW,\VV) \Big) ⊗_R \SNXp(M₂;\XX₂,\VV)
\\
&\overset{(2)}{\simeq}
\int\displaylimits^{\VV ∈ \hZXp(N)} 
\int\displaylimits^{\WW ∈ \hZX(N')}
\SNX(M₁';\WW) ⊗_R \SNX(C;\XX₁,\WW,\VV) ⊗_R \SNXp(M₂;\XX₂,\VV)
\\
&\overset{(3)}{\simeq}
\int\displaylimits^{\WW ∈ \hZX(N')} \int\displaylimits^{\VV ∈ \hZXp(N)}
\SNX(M₁';\WW) ⊗_R \SNX(C;\XX₁,\WW,\VV) ⊗_R \SNXp(M₂;\XX₂,\VV)
\\
&\overset{(4)}{\simeq}
\int\displaylimits^{\WW ∈ \hZX(N')} \int\displaylimits^{\VV ∈ \hZXp(N)}
\SNX(M₁';\WW) ⊗_R \SNXp(C;\XX₁,\WW,\VV) ⊗_R \SNXp(M₂;\XX₂,\VV)
\\
&\overset{(5)}{\simeq}
\int\displaylimits^{\WW ∈ \hZX(N')}
\SNX(M₁';\WW) ⊗_R \SNXp(C ∪_{N} M₂;\XX₁,\XX₂,\WW)
\\
&\overset{(6)}{\simeq}
\int\displaylimits^{\WW ∈ \hZX(N')}
\SNX(M₁';\WW) ⊗_R \SNX(C ∪_{N} M₂;\XX₁,\XX₂,\WW)
\\
&\overset{(7)}{\simeq}
\SNX(M;\XX)
\end{align*}} The individual isomorphisms are obtained as follows:
\begin{enumerate}
\item Theorem \ref{t:excision}, cutting $M₁$ along $N'$,
\item right exactness of $⊗_R$ to pull the coend $\int^{\hZX(N')}$ out,
\item Fubini theorem\footnote{One can easily show that
both coends agree with the coend over the $R$-linear cartesian product,
$\int^{(\VV,\WW) ∈ \hZXp(N) \hatbox \hZX(N')}
\SNX(D;\WW) ⊗_R \SNX(M₁';\XX₁,\WW,\VV) ⊗_R \SNXp(M₂;\XX₂,\VV)
$.}
to swap the order of coends,
\item
change $\SNX$ to $\SNXp$ in the middle tensor factor using
$\ov{ι}_{\cS ⊆ \cT}$ is an isomorphism
(by Lemma \ref{l:bulk-cS-natural}
    -- the boundary value $\WW \cup \VV$ on $C$ is always bulk-$\cS$-admissible as $C$ is a cylinder and $\WW$ is $\cS$-admissible), 
\item right exactness of $⊗_R$ and Theorem \ref{t:excision}
	to glue up along $\VV$,
\item 
use the isomorphism $\ov{ι}_{\cS ⊆ \cT}$ to replace $\SNXp$ by $\SNX$ in the last tensor factor
    (again by Lemma \ref{l:bulk-cS-natural} --
	the connectivity condition on $M₂$ is needed to ensure that
	$\WW$ is bulk-$\cS$-admissible in $C ∪_{N} M₂$),
\item Theorem \ref{t:excision} to glue along $N'$.
\end{enumerate}
\end{proof}

%++%
\subsection{Excision for Cylinder Categories}
\label{s:excision-cylinder}

The cylinder categories, as defined in Section \ref{s:cylinder-cat},
also satisfy an excision property, where the cylinder category of
the glued up manifold is given by the ``horizontal trace''
(\cite{AlexanderKirillov:2020hkv}, see also \cite[Sec.\,2.4]{Beliakova:2017}),
which we recount below:

\begin{definition}[{\cite[Sec.\,3.2]{AlexanderKirillov:2020hkv}}]
\label{d:htr}
Let $\cC$ be a monoidal category, and $\cB$ a $\cC$-bimodule category.
Define the \emph{horizontal trace} $\htr_\cC(\cB)$
as the category with the same objects as $\cB$,
morphisms given by
\[
\Hom_{\htr_\cC(\cB)}(B₁, B₂) :=
\int^{X ∈ \cC} \cB(X \lact B₁, B₂ \ract X)
\]
and composition as follows:
for $ψ₁ ∈ \cB(X \lact B₁, B₂ \ract X)$
representing $[ψ₁] ∈  \htr_\cC(\cB)(B₁, B₂)$,
and $ψ₂ ∈ \cB(Y \lact B₂, B₃ \ract Y)$
representing $[ψ₂] ∈  \htr_\cC(\cB)(B₂, B₃)$,
their composition is
\begin{align*}
[ψ₂] ∘ [ψ₁] &:= [ψ₂ ∘_{B₂} ψ₁] \ ,
\\
ψ₂ ∘_{B₂} ψ₁ &:= (ψ₂ ⊗ \id_X) ∘ (\id_Y ⊗ ψ₁)
∈ \cB(YX \lact B₁, B₃ \ract YX) \ ,
\end{align*}
where we omit the coherence isomorphisms of the $\cC$ actions.
\end{definition}

There is a functor 
\[
i: \cB \to \htr_{\cC}(\cB)
\]
that is the identity on objects and
sends a morphism $ψ ∈ \cB(B₁,B₂) \simeq \cB(\one \lact B₁, B₂ \ract \one)$
to $[ψ] ∈ \htr_{\cC}(\cB)(B₁,B₂)$,
where the isomorphism uses the unit constraints of the $\cC$-actions on $\cB$.

\medskip

We will now give a skein-theoretic counterpart of the algebraic construction above.
Let $d ≥ 2$. We consider a similar setup of manifolds to the main theorem
as laid out in Section \ref{s:theorem-setup}, but one dimension lower.
Let $N$ be a $(d-1)$-manifold, and let
$ι: Q ↪ N$ be a codimension-1 oriented submanifold of $N$
that admits a tubular neighbourhood $\wdtld{ι} : [-1,1] \times Q ↪ N$.
Let $N'$ be the $(d-1)$-manifold obtained by cutting $N$ along $Q$.
Let $π: N' \to N$ be the natural quotient map,
and let $P,P'$ be the preimages of $Q$ corresponding to the two sides of $Q$,
with $P$ outgoing and $P'$ incoming.
The tubular neighbourhood $\wdtld{ι}$ lifts to collar neighbourhoods
$\wdtld{P} := [-1,0] \times P ↪ N'$, $\wdtld{P'}:= [0,1] \times P' ↪ N'$
that join up to $\wdtld{ι}$ under $π$
(see Figure~\ref{f:Np-to-N}).

\begin{figure}
\centering
\includegraphics[width=38em]{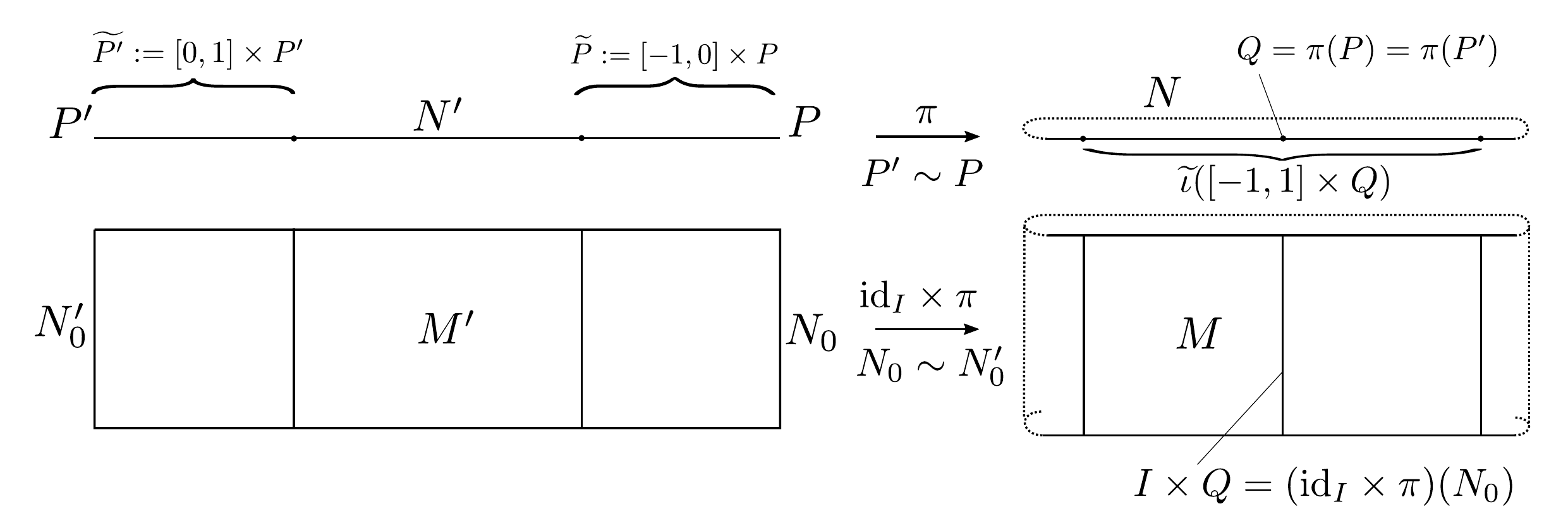}
\caption{Top: Gluing map $π: N' \to N$ identifying $P$ with $P'$.
Shown is the case $d=2$, where $N'$ is an interval and $P,P'$ its endpoints.
Bottom: Thickening by the interval $I$.}
\label{f:Np-to-N}
\end{figure}

Let $\cS ⊆ \cT ⊆ \cA$ be full subcategories.
Let $I = [3,4]$ (to avoid confusion with the other intervals).
We consider a thickened version of the setup above,
crossing everything with $I$, that is,
we consider $d$-manifolds $M = I \times N$, $M' = I \times N'$,
and $(d-1)$-manifolds $N₀ = I \times P ⊆ ∂M'$, $N₀' = I \times P' ⊆ ∂M'$,
so that gluing $N₀$ to $N₀'$ turns $M'$ into $M$ (see again Figure~\ref{f:Np-to-N}).

The ``stacking'' monoidal structure on $\hZXp(\wdtld{P})$ is defined as follows.
The tensor product $\VV ⊗ \WW$ on objects $\VV,\WW ∈ \hZXp(\wdtld{P})$
is obtained from applying the stacking-then-shrinking operation on $\wdtld{P}$
as seen in Section \ref{s:cylinder-cat},
to $\VV \sqcup \WW ∈ \hZXp(\wdtld{P} \sqcup \wdtld{P})$
with a fixed choice of embedding $\wdtld{P} \sqcup \wdtld{P} ↪ \wdtld{P}$.
This operation is associative up to ambient isotopy,
and the isotopy is used to define the associativity constraint.
The tensor product $f ⊗ g$ of morphisms
$f ∈ \hZXp(\wdtld{P})(\VV,\VV')$, $g ∈ \hZXp(\wdtld{P})(\WW,\WW')$
is obtained from applying the $I$-thickened version of the above operation.
The same structure can be defined for $\hZXp(\wdtld{P'})$,
and identifications $P \simeq P'$ and $[-1,0] \simeq [0,1]$
give a monoidal equivalence $\hZXp(\wdtld{P}) \simeq \hZXp(\wdtld{P'})$.

Likewise, we obtain a $\hZXp(\wdtld{P'})$-$\hZXp(\wdtld{P})$-bimodule structure
on $\hZX(N')$ by the operations on the collar neighbourhood
as seen in Section \ref{s:skein-module-as-functor}
(see Figure \ref{f:Z-bimodule-action});
with the monoidal equivalence $\hZXp(\wdtld{P}) \simeq \hZXp(\wdtld{P'})$,
we now have the $\hZXp(\wdtld{P})$-bimodule category $\hZX(N')$,
and we may consider its horizontal trace
$\htr_{\hZXp(\wdtld{P})}(\hZX(N'))$.

\begin{figure}
\centering
\includegraphics[width=10cm]{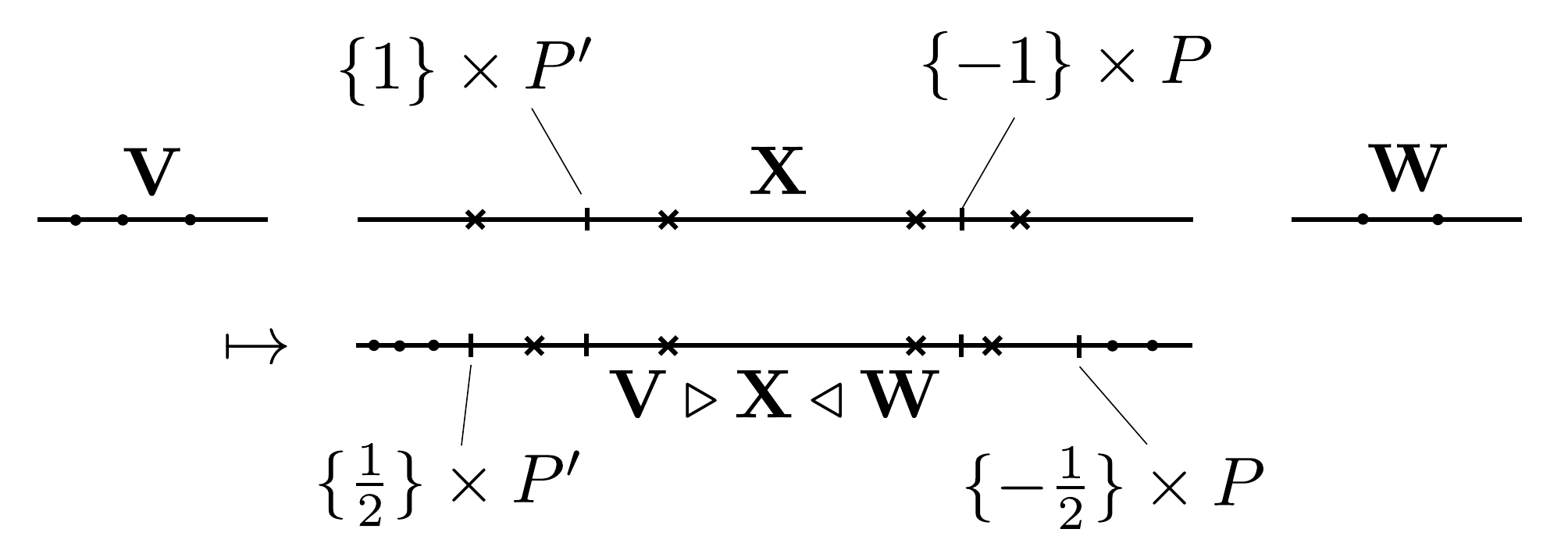}
\includegraphics[width=10cm]{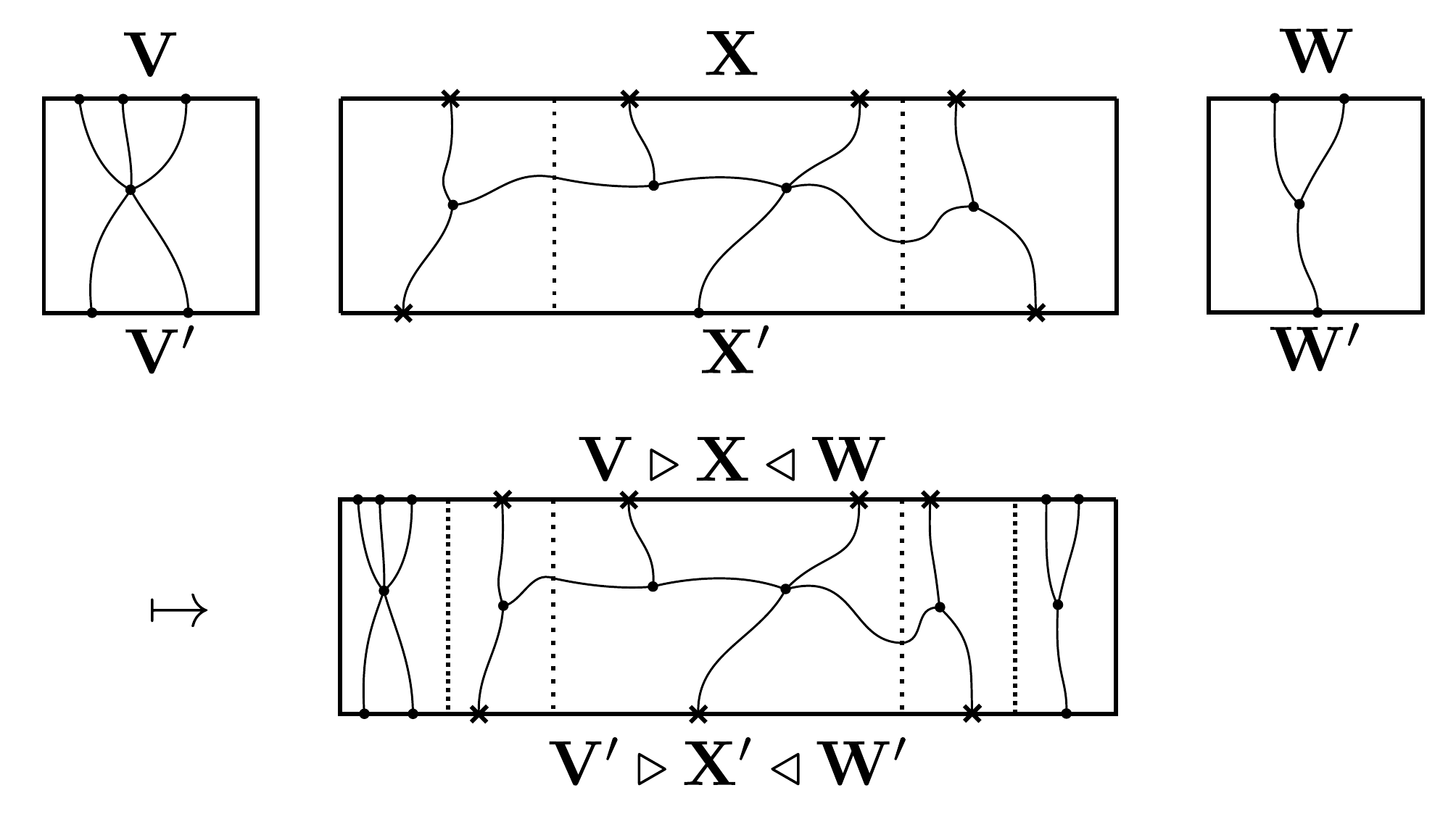}
\caption{$\hZXp(P)$-bimodule structure on $\hZX(N)$ from gluing.}
\label{f:Z-bimodule-action}
\end{figure}

We will now define an extension 
$\wdtld{π}_*: \htr_{\hZXp(\wdtld{P})}(\hZX(N')) \to \hZX(N)$
of the functor 
\[
    \pi_* : \hZX(N') \to \hZX(N)
\]    
induced by $π: N' \to N$ along the functor
$i: \hZX(N') \to \htr_{\hZXp(\wdtld{P})}(\hZX(N'))$.
Figure \ref{f:excision-glue} describes a homeomorphism $φ$ from $M'$ to itself,
which is a diffeomorphism away from corners as detailed in the figure. It induces an isomorphism
\[
φ_*: \SNX(M';\VV \lact \XX, \YY \ract \VV, ∅ , ∅ )
\simeq
\SNX(M';φ_*(\XX), φ_*(\YY),\VV,\VV) \ .
\]
On the left, $\SNX(M';\VV \lact \XX, \YY \ract \VV, ∅ , ∅ )$
has boundary value $\VV \lact \XX$ at the bottom $\{3\} \times N'$,
$\YY \ract \VV$ at the top $\{4\} \times N'$,
and empty boundary values on the vertical walls $I \times P', I \times P$.
On the right, $\SNX(M';φ_*(\XX), φ_*(\YY),\VV,\VV)$
has boundary value $φ_*(\XX)$ at the bottom, $φ_*(\YY)$ at the top,
and $\VV$ on the left and right vertical walls.
There are natural isomorphisms $φ_*(\XX) \simeq \XX$, $φ_*(\YY) \simeq \YY$,
which induce an isomorphism
\[
σ : \SNX(M';φ_*(\XX),φ_*(\YY),\VV,\VV)
\simeq
\SNX(M';\XX,\YY,\VV,\VV)
\]
and the gluing map $\id_I \times π: M' \to M$ induce a map
\[
(\id_I \times π)_* : \SNX(M';\XX,\YY,\VV,\VV)
\to
\SNX(M;π_*(\XX),π_*(\YY))
\]
with boundary values $π_*(\XX)$ at the bottom and $π_*(\YY)$ at the top.
Composing $φ_*$, $\sigma$, and $(\id_I \times π)_*$
we get a map
\[
\kappa_\VV : \hZX(N')(\VV \lact \XX, \YY \ract \VV) \to \hZX(N)(\XX,\YY)
\ .
\]
Adding these up over $\VV ∈ \hZXp(\wdtld{P})$,
we get a map $\sum \kappa_\VV :
\bigoplus \hZX(N')(\VV \lact \XX, \YY \ract \VV) \to \hZX(N)(\XX,\YY)$,
and it is easy to see that this factors through the coend
(dinaturality follows as in \eqref{eq:Sk-dinat} in the excision statement),
\[
\wdtld{\kappa} :
\int^{\VV ∈ \hZXp(\wdtld{P})} \hZX(N')(\VV \lact \XX, \YY \ract \VV)
\to \hZX(N)(\XX,\YY)
\ .
\]

\begin{figure}
\centering
\includegraphics[width=15cm]{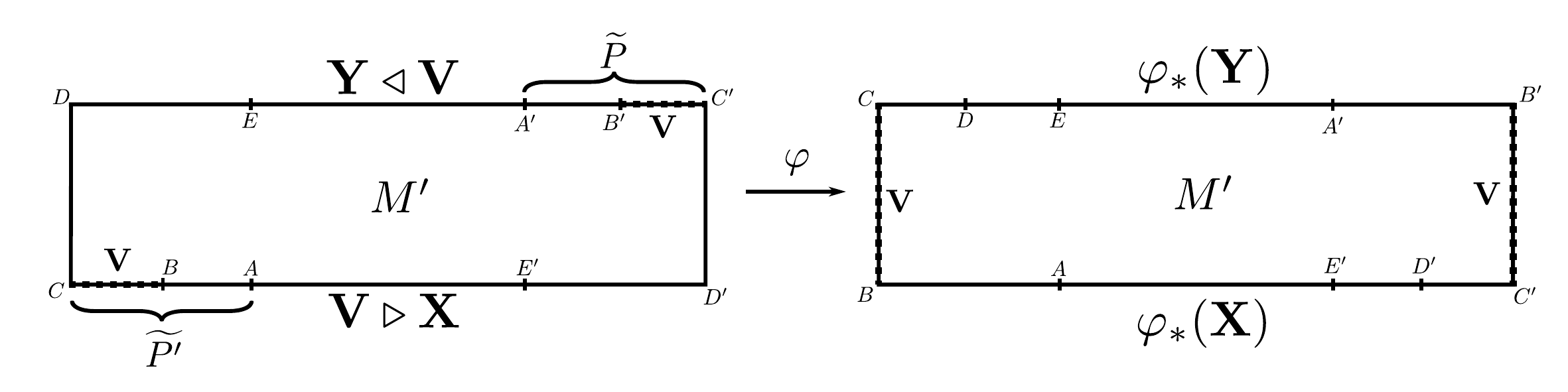}
\caption{
Left: 
Gluing operation from $M' = I \times N'$ to $M = I \times N$.
We glue the thickened segments
$[0,1/2] \times P'$, i.e.\ the segment $B$-$C$,
to $[-1/2,0] \times P$, i.e.\ to the segment $B'$-$C'$.
Right: $M'$
after applying a deformation 
$\varphi : M' \to M'$ (a diffeomorphism away from 
$B$, $C$, $D$ and $B'$, $C'$, $D'$).
We depict the deformation by keeping track of the points $A$ to $E$
and $A'$ to $E'$ (each of these points represents a copy of $P'$); the segments $A$-$E'$ and $A'$-$E$ are fixed.
The segment $C'$-$D'$-$E'$ in the left figure is compressed into the bottom copy of $N'$;
the segment $A$-$B$-$C$ in the left figure is expanded so that only $A$-$B$ remains in $N'$. A corresponding deformation is carried out for the top copy of $N'$.
}
\label{f:excision-glue}
\end{figure}

We can now state excision for cylinder categories in terms of the horizontal trace:

\begin{proposition}
\label{p:excision-cylinder-cat}
Consider the $\hZXp(\wdtld{P})$-bimodule category $\hZX(N')$, where $\cS ⊆ \cT$.
The functor 
\[
\wdtld{π}_*: \htr_{\hZXp(\wdtld{P})}(\hZX(N')) \to \hZX(N) \ ,
\]
defined by $π_*$ on objects and $\wdtld{\kappa}$ on morphisms
extends $π_*$ along $i$, i.e.\ $π_* = \wdtld{π}_* ∘ i$.
Furthermore, $\wdtld{π}_*$ is an equivalence.
\end{proposition}
\begin{proof}
Essential surjectivity is clear:
$\wdtld{π}_*$ only misses boundary values that have points on $Q ⊆ N$,
but such boundary values are isomorphic to ones obtained
by displacing marked points off of $Q$.
For fully-faithfulness,
\begin{align*}
\htr_{\hZXp(\wdtld{P})}(\hZX(N'))(\XX,\YY)
&= \int^{\VV ∈ \hZXp(\wdtld{P})} \hZX(N')(\VV \lact \XX,\YY \ract \VV)
\\
&= \int^{\VV ∈ \hZXp(\wdtld{P})} \SNX(M';\VV \lact \XX, \YY \ract \VV, ∅ , ∅ )
\\
&\overset{σ ∘ φ_*}{\simeq}
\int^{\VV ∈ \hZXp(\wdtld{P})} \SNX(M';\XX,\YY,\VV,\VV)
\\
&\overset{(*)}{\simeq}
\SNX(M;\XX,\YY)
\overset{\text{def.}}= \hZX(\XX,\YY)
\end{align*}
where in $(*)$ we apply Theorem \ref{t:excision} to the map $\id_I \times π$.
It straightforward to check that $\wdtld{π}_*$ respects composition,
so we omit the details.
\end{proof}

\begin{example}
\label{x:S1}
Let $\cC$ be a finite pivotal tensor category,
and let $\cI = \Proj(\cC)$ be the projective ideal.
By the excision result in Proposition \ref{p:excision-cylinder-cat} above,
we have
\begin{align*}
\hZJ{\cC}(S¹) \simeq \htr_{\hZJ{\cC}(I)}(\hZJ{\cC}(I)) \simeq \htr_\cC(\cC) \ ,
\\
\hZI(S¹) \simeq \htr_{\hZJ{\cC}(I)}(\hZI(I)) \simeq \htr_\cC(\cI) \ ,
\\
\hZI(S¹) \simeq \htr_{\hZI(I)}(\hZI(I)) \simeq \htr_\cI(\cI) \ ,
\end{align*}
in particular $\htr_\cC(\cI) \cong \htr_\cI(\cI)$.

Concretely, the equivalence $\htr_\cC(\cC) \simeq \hZJ{\cC}(S¹)$
sends $V ∈ \cC$ to the boundary value $\VV$
with one marked point $p$ (fixed beforehand) labelled by $V$,
and sends a morphism $[φ] ∈ \htr_\cC(\cC)(V,W)$,
represented by $φ ∈ \cC(AV,WA)$,
to the graph in the cylinder with one vertex labelled by $φ$,
with the $A$-labelled edge looping around the cylinder.
Similarly for the other equivalences.

The cylinder category over $S¹$ is closely related to $\cZ(\cC)$,
the Drinfeld centre of $\cC$.
For $\cC$ semisimple, \cite{Kirillov:2011} showed that
$\text{Kar}(\hZJ{\cC}(S¹)) \simeq \cZ(\cC)$,
where Kar indicates the Karoubi envelope.
This was generalised in \cite{Muller:2023} to non-semisimple $\cC$,
where they give a fully faithful functor $\hZI(S¹) \to \cZ(\cC)$
whose extension to the finite cocompletion of $\hZI(S¹)$
(in lieu of Karoubi envelope) is an equivalence;
here we implicitly claim that $\hZI(S¹)$ is equivalent to the
boundary categories $\textbf{sn}_\cC(S¹)$ of \cite{Muller:2023}
(see also Section~\ref{s:relation-MSWY}).
The second equivalence above reproduces \cite[Prop.\,4.2]{Muller:2023}:
\[
\hZI(S¹)(P,Q) \simeq \htr_\cC(\cI)(P,Q)
\simeq \int^{X ∈ \cC} \hspace{-.5em} \cC(XP,QX)
\simeq \int^{X ∈ \cC} \hspace{-.5em} \cC(P,X^*QX)
\simeq \cC(P, z(Q))
\]
where $P,Q \in \cI$, and
$z(Q) = \int^{X ∈ \cC} X^*QX$ is the central monad on $\cC$.
Since $P$ is projective by assumption, $\cC(P,-)$ is exact and we can pull the coend into the Hom-functor.
\end{example}

%++%
\subsection{Ideal Closure}
\label{s:closure}

\begin{definition}
\label{d:sum-retract-closure}
Let $\cS ⊆ \cA$ be a full subcategory.
The \emph{sum-retract closure $\rcls{\cS}$ of $\cS$}
is the smallest full subcategory of $\cA$ containing $\cS$
that is closed under taking retracts and finite products in $\cA$.
\end{definition}

Finite products in an $R$-linear category are always biproducts,
hence the use of the symbol $⊕$,
and we will use ``direct sum" interchangeably with product.
Since we do not require $\cA$ to be additive or idempotent complete,
the direct sums and retracts are taken when they exist.
On the other hand, if $\cA$ is additive and idempotent complete and if $\cS$ is additive,
then $\rcls{\cS}$ is equivalent to the Karoubi envelope of $\cS$, i.e.\ its idempotent completion.

\begin{remark}\label{rem:retract-sum-closure}
Let us separate the two parts of the sum-retract closure into closure under taking retracts, say $\cS^{rt}$, and closure under finite products, say $\cS^{Σ}$.
If $\cA$ is additive, then 
    $\rcls{\cS} = (\cS^{Σ})^{rt}$,
i.e.\ $\rcls{\cS}$ consists of objects in $\cA$
that are retracts of finite products $\bigoplus X_j$ of objects in $\cS$. But this is not true in general: for example, let $\Vect_\kk[\ZZ/3]$ be the semisimple category with 3 simple objects $X₀,X₁,X₂$. 
Take $\cA$ to be
the full subcategory with objects
$\cA = \{X₀^{⊕ 2}, X₁^{⊕ 2}, X₂^{⊕ 2},
X₀^{⊕ 2} ⊕ X₁^{⊕ 2}, X₀ ⊕ X₁, X₀ ⊕ X₁ ⊕ X₂^{⊕ 2}, X₀ ⊕ X₁ ⊕ X₂\}$,
and take $\cS = \{X₀^{⊕ 2}, X₁^{⊕ 2}, X₂^{⊕ 2}\}$.
One has $\cS_{(1)} := (\cS^{Σ})^{rt} 
= \cS \cup \{ X₀^{⊕ 2} ⊕ X₁^{⊕ 2} , X₀ ⊕ X₁ \}
≠ \rcls{\cS}$,
but $((\cS_{(1)})^{Σ})^{rt}= \rcls{\cS} = \cA$.
In general, one has to apply $((-)^{Σ})^{rt} $
infinitely many times to reach the sum-retract closure,
i.e.\ $\rcls{\cS} = \bigcup \cS_{(k)}$,
where $\cS_{(k+1)} = ((\cS_{(k)})^{Σ})^{rt}$.
\end{remark}

We will need the following technical lemma, whose proof is given in Appendix~\ref{s:appendix-closure}.

\begin{lemma}
\label{l:coend-closure}
Let $\cC$ be an essentially small $R$-linear category,
and let $\cB ⊆ \cC$ be a full subcategory.
Let $F: \cC^\op \times \cC \to \Rmod$ be a bifunctor,
and let $G : \cB^\op \times \cB \to \Rmod$ be its restriction to $\cB$.
Let $τ_X : F(X,X) \to \int^{\cC} F$ and $σ_Y: G(Y,Y) \to \int^{\cB} G$
be the canonical dinatural transformations to the respective coends.
Since for $Y\in \cB$ we have $G(Y,Y) = F(Y,Y)$, restricting $τ_Y$ to $Y \in \cB$ yields a dinatural transformation from $G$ to $\int^{\cC} F$,
which canonically descends to a map
\[
\xi ~:~ \int^{X ∈ \cB} G(X,X) \to \int^{X ∈ \cC} F(X,X) \ .
\]
If $\cC = \rcls{\cB}$, then $\xi$ is an isomorphism.
\end{lemma}

With these preparations, we can state the behaviour of cylinder categories under sum-retract closure.

\begin{lemma}
\label{l:r-closure-cylinder}
Let $N$ be a $(d-1)$-manifold,
and let $\cT ⊆ \cA$ be a full subcategory, with sum-retract closure $\rcls{\cT}$.
Recall that the functor $\ov{ι}_{\cT ⊆ \rcls{\cT}} : \hZXp(N) \to
\hZJ{\rcls{\cT}}(N)$ is fully faithful,
and is injective on the set of objects,
i.e.\ it presents $\hZXp(N)$ as a full subcategory of $\hZJ{\rcls{\cT}}(N)$.
Then we have 
\[
\rcls{\big(\hZXp(N)\big)} = \hZJ{\rcls{\cT}}(N) \ .
\]
\end{lemma}
\begin{proof}
Let $\YY = (B,\{Y_b\}) ∈ \hZJ{\rcls{\cT}}(N)$ be a boundary value on $N$
with labels $Y_b ∈ \rcls{\cT}$.
By definition, $\rcls{\big(\hZXp(N)\big)} \subseteq \hZJ{\rcls{\cT}}(N)$,
so equality follows if we can show that $\YY \in \rcls{\big(\hZXp(N)\big)}$.
    
Each $Y_b$ is a retract of some direct sum $\bigoplus_{j=1}^{N_b} X_j^{(b)}$
of objects in $\cT$.
Let $\XX_{\{j_b\}} = (B,\{X_{j_b}^{(b)}\})$ be the boundary value
with the same marked points as $\YY$,
but with labels given by choosing some $X_{j_b}^{(b)}$ for each $b ∈ B$;
clearly $\XX_{\{j_b\}} ∈ \hZXp(N)$.
Let $\XX_{\bullet} = (B,\{\bigoplus X_j^{(b)}\})$.

For boundary values with the same marked points,
say $\VV = (B,\{V_b\})$ and $\WW = (B,\{W_b\})$,
and morphisms $f_b : V_b \to W_b$,
we can promote the morphisms to a morphism
$\{f_b\} : \VV \to \WW$ in the cylinder category
given by a graph in $[0,1] \times N$ which is $[0,1] \times B$
with each $[0,1] \times \{b\}$ interrupted by a vertex labelled by $f_b$.
Applying this operation to the inclusion/projection morphisms
of $X_j^{(b)} ⇌ \bigoplus X_j^{(b)}$,
it is clear that $\XX_{\bullet}$ is the direct sum of all the $\XX_{\{j_b\}}$'s,
and similarly, $\YY$ is a retract of $\XX_{\bullet}$,
and hence $\YY$ is in $\rcls{\big(\hZXp(N)\big)}$.
\end{proof}

Next we turn to closure under tensor products and taking duals.

\begin{definition}
\label{d:tensor-dual-closure}
Let $\cC$ be a pivotal monoidal category, and let $\cS ⊆ \cC$ be a full subcategory.
Define the \emph{tensor-dual closure of $\cS$}, denoted $\tcls{\cS}$,
to be the smallest subcategory of $\cC$ containing $\cS$
that is closed under tensoring by objects of $\cC$ from the left and right
as well as taking duals,
i.e. for $X ∈ \tcls{\cS}, A ∈ \cC$, objects that are isomorphic to
$A ⊗ X$, $X ⊗ A$, and $X^*$ are also in $\tcls{\cS}$.
We say that $\cS$ is \emph{tensor-dual closed} if $\cS = \tcls{\cS}$.
\end{definition}

Equivalently, $\tcls{\cS}$ consists of objects isomorphic to $A ⊗ d(X) ⊗ A'$
where $A,A' ∈ \cA$, $X ∈ \cS$ and $d(X)$ stands for $X$ or $X^*$.
It turns out that cylinder categories do not change when one passes to the tensor-dual closure:

\begin{lemma}
\label{l:tensor-dual-closure}
Let $d ≥ 2$, let $N$ be an $(d-1)$-manifold,
and let $\cS ⊆ \cA$ be a full subcategory.
Then we have an equivalence of cylinder categories
$\ov{ι}_{\cS ⊆ \tcls{\cS}} : \hZX(N) \simeq \hZJ{\tcls{\cS}}(N)$.
\end{lemma}
\begin{proof}
Fully faithfulness holds by \eqref{e:iSSp}. For essential surjectivity consider an $\tcls{\cS}$-boundary value $\VV = (B, \{V_b\})$ on $N$. Each $V_b ∈ \tcls{\cS}$ is equivalent to some tensor product, say $V_b \simeq A₁ ⊗  ⋅⋅⋅ ⊗ A_k$, where for
at least one $i$ we have $A_i ∈ \cS$ or $A_i^* ∈ \cS$. We now split $b$ into $k$ points labelled by the $A_j$'s (if $A_i^* ∈ \cS$, we label the point with $A_i$
and give it a negative orientation).
This new boundary value is clearly $\cS$-admissible and equivalent to $\VV$.
\end{proof}

\begin{definition}
\label{d:tensor-ideal}
Let $\cC$ be an $R$-linear pivotal monoidal category.
An \emph{ideal} $\cI \subseteq \cC$ is a full subcategory that is
closed under duality, tensor products from $\cC$, direct sums, and retracts,
i.e. $\tcls{\cI} = \cI = \rcls{\cI}$.
The \emph{ideal closure} of a full subcategory $\cS \subseteq \cC$
the smallest ideal $\cI$ in $\cC$ containing $\cS$.
\end{definition}

Note that this definition differs slightly from that of ideals in
\cite{GeerPatureauMirand:2013} and \cite{Costantino:2023},
where closure under direct sums is not required.
On the other hand, while \cite{GeerPatureauMirand:2013,Costantino:2023} do not explicitly state closure under duality, this is a consequence of being simultaneously closed under retracts and tensor products \cite[Lem.\,1]{GeerPatureauMirand:2013}.

\begin{lemma}
\label{l:ideal-closure}
Let $\cC$ be an $R$-linear pivotal monoidal category.
Let $\cS ⊆ \cC$ be a full subcategory,
and let $\cI$ be the ideal closure of $\cS$.
Then $\cI$ can be obtained from $\cS$ in two steps:
first take the tensor-dual closure and then take sum-retract closure, i.e.
\[
\cI = \rcls{\tcls{\cS}} \ .
\]
\end{lemma}
\begin{proof}
We consider the increasing sequence of full subcategories $\{\cS^{(k)}\}$
defined inductively by $\cS^{(0)} = \cS$ and $\cS^{(k+1)} = \rcls{\tcls{\cS^{(k)}}}$;
each term is contained in $\cI$, and it follows from the minimality of $\cI$
that the sequence exhausts $\cI$, i.e.\ $\bigcup_k \cS^{(k)} = \cI$.
We will show that $\rcls{\tcls{\cT}} = \rcls{\tcls{\rcls{\cT}}}$
for any full subcategory $\cT$ of $\cC$;
applied to $\cT = \tcls{\cS}$, this gives $\cS^{(1)} = \cS^{(2)}$,
and hence the entire sequence collapses, so that $\cS^{(1)} = \cI$.

Since clearly $\tcls{\cT} ⊆ \tcls{\rcls{\cT}}$,
it suffices to show that $\tcls{\rcls{\cT}} ⊆ \rcls{\tcls{\cT}}$.
Let $A ⊗ d(X) ⊗ A' ∈ \tcls{\rcls{\cT}}$,
where $X ∈ \rcls{\cT}$ and
$d(-)$ stands for $(-)^*$ or the identity functor.
Consider the functor $F = A ⊗ d(-) ⊗ A'$ from $\cC$ to $\cC$ (or $\cC^\op$);
it preserves direct sums and retracts.
As discussed in Remark~\ref{rem:retract-sum-closure},
the sum-retract closure can be constructed by
repeatedly taking closures by direct sums and retracts in an alternating manner,
just as with $\cI$ above, so applying $F$ at each stage,
we find that $F(X) ∈ \rcls{\tcls{\cT}}$. In more detail,
$\im F|_{\cT^{rt}} ⊆ (\im F|_{\cT})^{rt}$,
and $\im F|_{\cT^{Σ}} ⊆ (\im F|_{\cT})^{Σ}$
(here $\im$ refers to the essential image), and
since $X ∈ ((⋅⋅⋅(\cT^{rt})^{Σ} ⋅⋅⋅)^{rt})^{Σ} ⊆ \rcls{\cT}$,
we then have 
\[
A ⊗ d(X) ⊗ A' = 
F(X) ∈ ((\cdots ((\im F|_{\cT})^{rt})^{Σ} \cdots)^{rt})^{Σ}
⊆ \rcls{(\im F|_{\cT})} ⊆ \rcls{\tcls{\cT}} \ .
\]
\end{proof}

The main result of this section is that $\cS$-admissible skein modules do not change if one passes to the ideal closure of $\cS$. 

\begin{proposition}
\label{p:ideal-closure}
Let $\cS ⊆ \cA$ be a full subcategory.
For $d = 1$, let $\cI = \cS^{⊕ }$ be the sum-retract closure,
and for $d ≥ 2$, let $\cI = \eval{\cS}^{⊕ }$ be the ideal closure.
Then 
\[ 
    \ov{ι}_{\cS ⊆ \cI} : \SNX(M;\XX) \to \SNI(M;\XX) 
\] 
is an isomorphism for any $M$ and $\XX$.
\end{proposition}

\begin{proof}
We induct on the number of connected components $M^{(a)}$ of $M$
such that $\XX|_{∂M^{(a)}}$ is not bulk-$\cS$-admissible
(we always assume that a $d$-manifold $M$ has finitely many connected components,
cf.\ the conventions stated just before Section~\ref{s:defn}).
In the base case, i.e.\ when $\XX$ itself is bulk-$\cS$-admissible, $\ov{ι}_{\cS ⊆ \cI} : \SNX(M; \XX) \simeq \SNI(M; \XX)$ follows from Lemma \ref{l:bulk-cS}.

For the inductive step,
let $M^{(a)}$ be a connected component of $M$
such that $\XX|_{∂M^{(a)}}$ is not bulk-$\cS$-admissible,
and let $N$ be a boundary piece of $M^{(a)}$.
Let $\XX₀ := \XX|_N$, and let $\XX' := \XX|_{∂M \backslash N}$.
Writing $M \simeq M ∪ _{N \sim \{0\} \times N} [0,1] \times N$,
we have (with explanation to follow):
\begin{align*}
\SNX(M; \XX)
&\overset{(1)}{\simeq}
\int^{\VV ∈ \hZX(N)} \SNX([0,1] \times N \sqcup M; \XX_0,\VV, \VV, \XX')
\\
&\overset{(2)}{\simeq}
\int^{\VV ∈ \hZX(N)} \SNX([0,1] \times N; \XX₀,\VV ) ⊗_R \SNX(M; \VV,\XX')
\\
&\overset{(3)}{\simeq}
\int^{\VV ∈ \hZX(N)} \SNI([0,1] \times N; \XX₀,\VV) ⊗_R \SNI(M; \VV,\XX')
\\
&\overset{(4)}{\simeq}
\int^{\VV ∈ \hZJ{\eval{\cS}}(N)} \SNI([0,1] \times N; \XX₀,\VV) ⊗_R \SNI(M; \VV,\XX')
\\
&\overset{(5)}{\simeq}
\int^{\VV ∈ \hZI(N)} \SNI([0,1] \times N; \XX₀,\VV) ⊗_R \SNI(M; \VV,\XX')
\\
&\overset{(6)}{\simeq}
\SNI(M;\XX) \ ,
\end{align*}
where
\begin{enumerate}
\item by Theorem \ref{t:excision},

\item by Lemma \ref{l:M-disjoint-union},

\item the inductive hypothesis applies to $\VV ∪ \XX'$,
	as it is bulk-$\cS$-admissible on $M^{(a)}$
	in addition to components of $M$ on which $\XX$ is bulk-$\cS$-admissible;
for the admissible skeins on $[0,1] \times N$ one can directly apply Lemma~\ref{l:bulk-cS},

\item by Lemma \ref{l:tensor-dual-closure},
	$\ov{ι}_{\cS ⊆ \tcls{\cS}}$ is an equivalence
	(only applies/needed in the case $d ≥ 2$),

\item by Lemmas \ref{l:coend-closure} and \ref{l:r-closure-cylinder}
	(applied to $\cT = \cS$ for $d = 1$ and to
	$\cT = \tcls{\cS}$ for $d ≥ 2$),

\item by Theorem \ref{t:excision}.
\end{enumerate}
\end{proof}

\begin{remark}\label{r:ideal-closure}
We have seen in 
Remark~\ref{r:ncat-strat}\,(1) that the map $\ov{ι}_{\cS ⊆ \cT} : \SNX(M;\VV) \to \SNXp(M;\VV)$ is in general neither injective nor surjective. One way to improve this is to impose conditions on the boundary value $\VV$ as in Lemma~\ref{l:bulk-cS}, resulting in $\ov{ι}_{\cS ⊆ \cT}$ being an isomorphism. Proposition~\ref{p:ideal-closure} provides another criterion for $\ov{ι}_{\cS ⊆ \cT}$ to be an isomorphism which leaves $\VV$ arbitrary but restricts the choice of $\cT$. Namely, for any $\cT ⊆ \cI$ containing $\cS$
(in particular, in the $d ≥ 2$ case, for $\cT = \eval{\cS}$),
the sum-retract/ideal closure of $\cT$ is also $\cI$,
so it immediately follows that $\ov{ι}_{\cS ⊆ \cT}
= \ov{ι}_{\cT ⊆ \cI}^{-1} ∘ \ov{ι}_{\cS ⊆ \cI}$ is an isomorphism.
\end{remark}

%++%
\subsection{Totally-$\cS$ Skeins}
\label{s:totally-S}

Let $\cS ⊆ \cA$ be a full subcategory.
We say an $\cA$-coloured graph is \emph{totally-$\cS$}
if it is non-empty and every internal edge is in $\cS$.
We do allow boundary values, and hence boundary edges, to have labels not from $\cS$.
    This greater generality is sometimes convenient, we will see an example of this in the proof of Proposition~\ref{p:finite-dim} later on.
We denote the subset of totally-$\cS$ graphs by $\GraphAX$, and we say a boundary value is \emph{totally-$\cS$}
if it is non-empty and every label is in $\cS$.

Given a $d$-manifold with boundary value $\XX$ on $∂M$,
we have the following definitions:
\begin{align*}
\GraphAX(M;\XX) &:= \{ \text{totally-} \cS \text{ graphs} \}
	⊆ \GraphX(M;\XX)
\\
\VGraphAX(M;\XX) &:= \Span ( \GraphAX(M;\XX) ) ⊂ \VGraphX(M;\XX)
\\
\PrimNullAX(M,U;\XX) &:= \PrimNullX(M,U;\XX) ∩ \VGraphAX(M;\XX)
\\
\NullAX(M,U;\XX) &:= \Span(\PrimNullAX(M,U;\XX))
\\
\SNAX(M;\XX) &:= \VGraphAX(M;\XX) / \NullAX(M;\XX)
\end{align*}

The natural inclusions
$ι_{\mathrm{t}\cS \to \cS} : \VGraphAX(M;\XX) ↪ \VGraphX(M;\XX)$ descends to skeins, $\ov{ι}_{\mathrm{t}\cS \to \cS} : \SNAX(M;\XX) \to \SNX(M;\XX)$,
since $\NullAX ⊂ \NullX ∩ \VGraphAX$.

\begin{remark}
The map $\ov{ι}_{\mathrm{t}\cS \to \cS}$ may be neither injective nor surjective,
in particular it can happen that
    $\NullAX ≠ \NullX ∩ \VGraphAX$.
For example, we may take $\cA$ 
to be the non-spherical
pivotal category of
$\ZZ/3$-graded vector spaces over $\CC$, with simple objects $X_0 = \one$, $X_1$, $X_2$, and where $X_1$ has dimensions $\dim_r X₁ = ζ := e^{2πi/3}$
and $\dim_l X₁ = ζ^\inv$, and $\cS = \{\one\}$.
Then $\ov{ι}_{\mathrm{t}\cS \to \cS}$ is non-injective for $S²$ as
$\SNAX(S²) \simeq \CC$ and 
$\SNX(S²) = 0$ (see \cite[Prop.\,4.4]{Runkel:2019vze}),
while $\ov{ι}_{\mathrm{t}\cS \to \cS}$ is non-surjective for $\Ann$ as
$\SNAX(\Ann) \simeq \CC$ and $\SNX(\Ann) \simeq \CC³$.
\end{remark}

The next proposition allows us to pass between totally-$\cI$ skein modules and $\cI$-admissible skein modules under the condition that $\cI$ is tensor-dual closed.

\begin{proposition}
\label{p:totally-I}
Let $\cI ⊆ \cA$ be a full subcategory,
and in the case $d ≥ 2$
suppose $\cI$ is tensor-dual closed, i.e.\ $\cI = \eval{\cI}$.
Then $\ov{ι}_{\mathrm{t}\cI \to \cI}$ is an isomorphism for any $M,\XX$:
\[
\ov{ι}_{\mathrm{t}\cI \to \cI} : \SNAI(M;\XX) \simeq \SNI(M;\XX)
\]
\end{proposition}

\begin{proof}
We first show surjectivity,
i.e.\ that any $\cI$-admissible graph is equal to some
totally-$\cI$ graph as $\cI$-admissible skeins.
For $d ≥ 2$, the map is surjective by \cite[Prop.\,2.3]{Costantino:2023}
(they do not deal with skeins with boundary values,
but the core argument works all the same:
combining the $\cI$-labelled edge with another edge
results in a new edge with label in $\cI$ by closure under tensor products).
For $d = 1$, given an $\cI$-admissible graph,
we can always reduce the number of internal edges that are not $\cI$
by using an $\cI$-null relation, which squeezes the edge
and combines the labels at its endpoints by composition of morphisms.
It remains to show injectivity,
i.e.\ $\NullAI = \NullI ∩ \VGraphAI$.

\medskip

We first prove injectivity in the special case that $\XX$ is bulk-$\cI$-admissible and $M$ finitary by induction on the number of gluings it takes
to produce $M$ from a disjoint union of finitely many balls. 

\smallskip

\noindent
\emph{Base case}:
$M = \DD^d \sqcup \cdots \sqcup \DD^d$
(with any corner structure on the boundary),
and $\XX$ is bulk-$\cI$-admissible.

\smallskip

\noindent
\emph{Proof of base case}:
We just prove it for $M = \DD^d$,
as it is easy to see that $\SNAI$ is also behaves well under
disjoint union, as in Lemma~\ref{l:M-disjoint-union}.
Let $Σ b_k Γ_k$ be a (not necessarily primitive) $\cI$-null graph
whose summands are totally-$\cI$, i.e.\ $Σ b_k Γ_k ∈ \NullI \cap \VGraphAI$.
Similarly to the proof of Lemma \ref{l:bulk-cS},
we may isotope each $Γ_k$ to $Γ_k'$ with an ambient isotopy
supported in a small neighbourhood of $∂M$
such that $Γ_k'$ all agree in a neighbourhood of $∂M$.
Then $Γ_k - Γ_k'$ is in $\NullAI$,
and $Σ b_k Γ_k'$ is primitive totally-$\cI$-null
with respect to the ball $M \backslash \{\text{neighbourhood of } ∂M\}$.
Thus $Σ b_k Γ_k' \in \NullAI$. $\triangle$

\smallskip

\noindent
\emph{Inductive step:}
Suppose $\SNAI(M';\XX') \simeq \SNI(M';\XX')$
for any bulk-$\cI$-admissible boundary value $\XX'$ on $∂M'$.
Let $M = M'/(N \sim N')$ be obtained by gluing two embedded boundary pieces $N,N'$ of $M'$.
Then $\SNAI(M;\XX) \simeq \SNI(M;\XX)$
for bulk-$\cI$-admissible boundary values $\XX$ on $∂M$.

\smallskip

\noindent
\emph{Proof of inductive step}:
Let $\hZAI(N) ⊆ \hZI(N)$ 
be the full subcategory consisting of
totally-$\cI$ boundary values.
Observe that, for $d = 1$, $\hZAI(N) = \hZI(N)$ on the nose,
and for $d ≥ 2$, $\hZAI(N)$ is equivalent to $\hZI(N)$,
since any boundary value $\VV ∈ \hZI(N)$ is isomorphic
to a boundary value $\VV'$ that has one marked point $b^{(a)}$
on each connected component $N^{(a)}$ of $N$,
where the label on $b^{(a)}$ is the tensor product
(in arbitrary order if $\dim N = d-1 \ge 2$)
of the labels of $\VV$ on $N^{(a)}$,
so by closure of $\cI$ under tensor product, $\VV'$ is totally-$\cI$.

We have the following commutative diagram:
\[
\begin{tikzcd}
\int^{\hZAI(N)} \SNAI(M';\XX,-,-)
	\ar[r,"\simeq","(1)"']
	\ar[d,"f"']
& \int^{\hZAI(N)} \SNI(M';\XX,-,-)
\ar[r, "\simeq","(2)"']
& \int^{\hZI(N)} \SNI(M';\XX,-,-)
	\ar[d,"\hat{π}_*","\simeq"']
\\
\SNAI(M;\XX)
	\ar[rr, "\ov{ι}_{\mathrm{t}\cI \to \cI}"]
& & \SNI(M;\XX)
\end{tikzcd}
\]
where (1) is an isomorphism by the inductive hypothesis,
(2) is an isomorphism since $\hZAI(N) \cong \hZI(N)$,
and $\hat\pi_*$ is an isomorphism by Theorem \ref{t:excision}.
It remains to show that $f$ is surjective
-- then by commutativity of the diagram, $f$ and $\ov{ι}_{\mathrm{t}\cI \to \cI}$ must be isomorphisms.

Let $L$ be the image of $N$ (and $N'$) under the gluing $M' \to M$.
Given a totally-$\cI$ graph $Γ$ in $M$,
using some isotopy, say a push-map as discussed in Definition \ref{d:push-map},
we can drag edges of $Γ$ to intersect $L$ transversally.
Such a graph is in the image of the map in question,
as it can be cut and considered a graph in $M'$
with bulk-$\cI$-admissible boundary value.
Thus the inductive step is done. $\triangle$

\medskip

For $M$ with boundary value $\XX$ on $∂M$ that is not necessarily bulk-$\cI$-admissible,
we may use excision again, as in the proof of
Proposition~\ref{p:ideal-closure},
where we induct on the number of components of $M$
on which $\XX$ is not bulk-$\cI$-admissible.

Finally, for non-finitary $d$-manifolds $M$,
an analogue of Lemma \ref{l:finitary-colimit} holds for $\SNAI$ as well,
so that $\SNAI(M;\XX)$ is a colimit of $\SNAI(M';\XX)$
over finitary submanifolds $M'$ of $M$;
since we know $\SNAI(M';\XX) \simeq \SNI(M';\XX)$
for these finitary submanifolds, the same holds for their colimits,
and we are done.
\end{proof}

%++%
\section{Applications}
\label{s:applications}

In this section we consider admissible skein modules in more specific situations, such as when $\cA$ is abelian, when $\cS$ is the projective ideal, or for low dimensions. We provide an exactness result (Proposition~\ref{p:right-exact}), recover the relation to modified traces (Proposition~\ref{p:trace-skein-D2-S2}), and give upper and lower bounds on the rank of admissible skein modules under suitable extra assumptions on $R$, $\cA$ and $\cS$ (Propositions~\ref{p:finite-dim} and \ref{p:infinite-dim}).

%++%
\subsection{Right Exactness}
\label{s:relation-MSWY}

Let $\cS ⊆ \cA$ be a full subcategory,
$M$ a $d$-manifold, and $B$
a configuration of marked points on $∂M$
with $B_+$ consisting of the outward-oriented points of $B$
and $B_-$ the inward-oriented points.
Labelling each point $b ∈ B$ with an object $X_b ∈ \cA$,
we get admissible skein modules $\SNX(M;(B,\{X_b\})) ∈ \Rmod$, which assemble into a functor
\[
\SNX(M;(B,-)) : \cA^{B_+} \times (\cA^\op)^{B_-} \to \Rmod \ .
\]
For a morphism $\{f_b\}_{b ∈ B} : \{X_b\}_{b ∈ B} \to \{Y_b\}_{b ∈ B}$,
i.e.\ $f_b : X_b \to Y_b$ for $b ∈ B_+$
and $f_b : Y_b \to X_b$ for $b ∈ B_-$,
the map $\SNX(M;(B,\{f_b\}))$ inserts a bivalent vertex
labelled by $f_b$ near each $b ∈ B$.
Equivalently, this is the functor from Section~\ref{s:skein-module-as-functor},
where we take $N ⊆ ∂M$ to be a disc-shaped neighbourhood of $B$ and add corners to $M$ if necessary so that $N$ is indeed a boundary piece containing only $b$.

\begin{proposition}
\label{p:right-exact}
Let $d \ge 2$ and assume that $\cA$ is in addition abelian. Let $\cI = \Proj(\cA)$ be the projective ideal.
Then for a $d$-manifold $M$ we have that
\[
\SNI(M;(B,\{-\})) ~:~ \cA^{B_+} \times \cA^{B_-} ~\to~ \Rmod
\]
is right exact in each argument.
\end{proposition}
\begin{proof}
Fixing labels on all points except at one fixed choice
$b ∈ B_+$ (resp.\ $B_-$),
we have a functor
$\SNI(M;\VV,-): \cA \to \Rmod$ (resp.\ $\cA^\op \to \Rmod$).
Below we will omit writing out $\VV$ in $\SNI(M;\VV,X)$ for simplicity. We consider the case $b ∈ B_+$ first.
Let $D ⊆ ∂M$ be a disc-shaped neighbourhood containing the marked point $b$ on $∂M$.
Adding corners if necessary, we can present
$M = M ∪  ([0,1] \times D)$, 
and by Theorem~\ref{t:excision} we have
\[
\SNI(M;X) \simeq \int^{P ∈ \cI} \SNI(M;P) 
⊗_R \SNI([0,1] \times D; P,X)
\simeq \int^{P ∈ \cI} \SNI(M;P) ⊗_R \cA(P,X)
\]
where in the term $\SNI([0,1] \times D; P,X)$,
$P,X$ are labels for the marked points $(0,b), (1,b)$ respectively,
where $(0,b)$ is negatively oriented and $(1,b)$ positively.
We can write the coend explicitly as a cokernel, so that we have
\begin{align*}
\begin{split}
\label{e:strnet-coker}
G(X) \xrightarrow{α_X} F(X) \to \SNI(M;X) \to 0
\end{split}
\end{align*}
where
$G(X) = \bigoplus_{P_i,P_j} \bigoplus_{f ∈ \cA(P_i,P_j)} \SNI(M;P_i) ⊗_R \cA(P_j,X)$
and $F(X) = \bigoplus_{P_i} \SNI(M;P_i) ⊗_R \cA(P_i,X)$, with
$\{P_i\}$ ia set of representatives for the isomorphism classes of projectives.
Note that $F$ and $G$ are right exact in $X$,
since $\cA(P_i,X)$ is exact in $X$ and $⊗_R$ is right exact.
Then, for an exact sequence $0 \to Z \to Y \to X \to 0$ in $\cA$,
consider the following diagram:
\[
\begin{tikzcd}
G(Z) \ar[r] \ar[d]
	& F(Z) \ar[r] \ar[d]
	& \SNI(M;Z) \ar[r] \ar[d] & 0
\\
G(Y) \ar[r] \ar[d]
	& F(Y) \ar[r] \ar[d]
	& \SNI(M;Y) \ar[r] \ar[d] & 0
\\
G(X) \ar[r] \ar[d]
	& F(X) \ar[r] \ar[d]
	& \SNI(M;X) \ar[r] \ar[d] & 0
\\
0 & 0 & 0
\end{tikzcd}
\]
A priori, all rows and columns except the rightmost column are exact.
From standard diagram chasing, it follows that the rightmost column
is exact as well, i.e.\ $\SNI(M;-)$ is right exact.
Finally, for the case $b ∈ B_-$, we have
\[
\SNI(M;X) \simeq \int^{P ∈ \cI} \SNI(M;P) ⊗_R \cA(X,P)
\]
and we note that $\cA(X,P)$ is also exact in $X$ since
rigidity of $\cA$ implies that projectives are injective,
so the rest of the argument follows through;
alternatively, we note that $\SNI(M;X) \simeq \SNI(M;(B',\{X^*,...\}))$
where $B'$ is the same as $B$ except the orientation on $b$ is flipped.
Then right exactness follows since the duality functor is exact
and sends $\cI$ to itself.
\end{proof}

\begin{remark}
\begin{enumerate}
    \item 
Going through the proof of the proposition in the case $d=1$, one finds that the statement also holds if we restrict ourselves to the outgoing points $B_+$. In that case the argument still reduces to $\cA(P,-)$ being exact.
\item
In general, $\SNI(M;-)$ is not left exact.
For example, take $d=2$ and 
let $\cA$ be a non-semisimple
unimodular finite pivotal tensor category
over a field $R = \kk$ and $Σ = \DD²$. Consider the short exact sequence
$0 \to \one \to P_\one \to P_\one / \one \to 0$,
where $P_\one$ is the projective cover of $\one$.
From \cite[Thm.\,3.1]{Costantino:2023}
(see also Proposition \ref{p:trace-skein-D2-S2} below),
the dual of $\SNI(\DD²;\one) \simeq \SNI(\DD²)$
is the space of right m-traces on $\cI$.
By \cite[Cor.\,3.2.1]{GeerKujawa:2013},
$\cI$ admits a unique right m-trace up to scalars, in particular is non-zero.
But the map $\SNI(\DD²;\one) \to \SNI(\DD²;P_\one)$ is 0:
it sends a graph $Γ ∈ \SNI(\DD²)$ to $Γ' = Γ ∪ γ$,
where $γ$ is a graph with one edge, which is attached to the boundary point
(hence labelled by $P_\one$), and one internal vertex
labelled by the morphism $\one \to P_\one$;
now $Γ$ can be evaluated in a ball that is disjoint from $γ$,
and since $Γ$ was $\cI$-admissible to begin with,
it must evaluate to some $f ∈ \cA(\one,\one)$
that factors through a projective, and such morphisms are always 0
for non-semisimple $\cA$.

\item For ideals $\cI$ besides the projective ideal,
$\SNI(M;-)$ may not even be right exact.
For example, consider $\cA$ as in part 2, but with $\cI = \cA$,
and consider the epimorphism $P_\one \to \one$,
and the corresponding map of skein modules
$\SNA(\DD²;P_\one) \to \SNA(\DD²;\one) \simeq \SNA(\DD²)$.
By the same logic above, this map is the zero map,
and yet $\SNA(\DD²) ≠ 0$, since the categorical trace on $\cA$
is a non-zero right m-trace on $\cA$,
and again by Proposition \ref{p:trace-skein-D2-S2},
$\SNA(\DD²)^* \simeq \{\text{right m-traces on } \cA\}$.
\end{enumerate}
\end{remark}

\begin{remark}
\label{r:relation-MSWY}
In \cite{Muller:2023}, for a pivotal finite tensor category $\cC$
over an algebraically closed field $\kk$,
string-net spaces $\StrNet_\cC(Σ;X₁,...,X_n)$ are defined
for surfaces $Σ$ without corners,
and it is shown that they collectively define a modular functor
that agrees with the Lyubashenko modular functor associated to
the Drinfeld centre $\cZ(\cC)$.
Here the boundary of $Σ$ is decorated with one marked point on each component,
labelled by the $X_i$'s. The spaces $\StrNet_\cC$ are defined indirectly
via $\strnet_\cC$, which only take projective inputs $X_i ∈ \cI = \Proj(\cC)$,
and are essentially the same as totally-$\cI$-skein modules defined
in Section~\ref{s:totally-S}.
In brief, for a set of outwardly oriented marked points $B$ on $∂Σ$,
$\strnet_\cC(Σ;X₁,...,X_n)$ is the colimit of a diagram in $\Vect_\kk$,
indexed over edge-labelled graphs $(\ov{Γ},χ)$
with boundary value $(B,\{X_i\})$ such that $χ_e ∈ \cI$;
the corresponding vector space $\mathbf{E}(\ov{Γ},χ)$
is the space of vertex colourings $\bigotimes_v \Col(\ov{Γ},χ,v)$
(slightly different from $\Col(\ov{Γ},χ)$, which is a product as sets),
and arrows relate graphs that have the same evaluation in a ball.
We have a map from the diagram to $\SNAI(Σ;X₁,...,X_n)$,
where from each $\mathbf{E}(\ov{Γ},χ)$,
we send a vertex colouring $φ$ to the $\cC$-coloured graph $(\ov{Γ},χ,φ)$.
It is not hard to check that this commutes with the arrows in the diagram,
and descend to a natural isomorphism
\begin{equation}
\label{e:strnet-proj}
\strnet_\cC(Σ;X₁,...,X_n)
= \Colim_{(\ov{Γ},χ)} \mathbf{E}(\ov{Γ},χ)
\simeq \SNAI(Σ;\XX = (B,\{X_i\}))
\simeq \SNI(Σ;\XX) \ ,
\end{equation}
where we used Proposition~\ref{p:ideal-closure} for the last isomorphism.
This defines a functor
$\strnet_\cC(Σ;-,...,-) : \cI \times ⋅⋅⋅ \times \cI \to \Vect_\kk$
which extends uniquely to a functor
$\StrNet_\cC(Σ;-,...,-) : \cC \times ⋅⋅⋅ \times \cC \to \Vect_\kk$
that preserves colimits.
For example, in the case of a single boundary point and 
given a projective resolution
$Q \to P \to X \to 0$,
we have
\begin{equation}
\StrNet_\cC(Σ;X) \simeq \coker \big(
	\strnet_\cC(Σ;Q) \to \strnet_\cC(Σ;P)
\big) \ .
\end{equation}
By Proposition~\ref{p:right-exact}, $\SNI(Σ;-)$ is right-exact,
so we also have $\StrNet_\cC(Σ;-) \simeq \SNI(Σ;-)$.
\end{remark}

%++%
\subsection{Relation to Traces}
\label{s:trace}

We recall modified traces, m-traces for short, as in \cite{Costantino:2023}
(see also \cite{GeerKujawa:2013}, \cite{GeerPatureauMirand:2013}).
Let $\cS ⊆ \cA$ be a full subcategory.
A \emph{trace on $\cS$} is an assignment
$\ttt_{-} = \{ \ttt_V ∈ \End_\cA(V) \to R \}_{V ∈ \cS}$
of an $R$-linear functional on the endomorphism space of each $V ∈ \cS$,
which is  \emph{cyclic} in the sense that for all $V,W ∈ \cS$,
and $f : V \to W, g : W \to V$, we have
$\ttt_V(gf) = \ttt_W(fg)$.

When $\cA$ is pivotal, for $f ∈ \cA(V ⊗ W, V ⊗ W')$, where $V,W,W' ∈ \cA$, we define its \emph{left partial trace (over $V$)},
denoted $\ptr_l^V(f)$, to be the morphism
\footnote{Here our convention is
$\ev_V: V^* ⊗ V \to \one$,
$\coev_V: \one \to V ⊗ V^*$,
$\wdtld{\ev}_V: V ⊗ V^* \to \one$,
$\wdtld{\coev}_V: \one \to V^* ⊗ V$.}
\[
\ptr_l^V(f) = (\ev_V ⊗ \id_W) ∘ (\id_{V^*} ⊗ f) ∘ (\wdtld{\coev}_V ⊗ \id_{W'})
∈ \cA(W,W') \ .
\]
Similarly, for $g ∈ \cA(W ⊗ V, W' ⊗ V)$ we define its \emph{right partial trace (over $W$)},
denoted $\ptr_r^V(g)$, to be the morphism
\[
\ptr_r^V(g) = (\id_{W'} ⊗ \wdtld{\ev}_V) ∘ (g ⊗ \id_{V^*}) ∘ (\id_W ⊗ \coev_V)
∈ \cA(W,W') \ .
\]
For a tensor-dual closed subcategory $\cI$ (recall Definition~\ref{d:tensor-dual-closure}),
we say a trace $\ttt$ on $\cI ⊆ \cA$ has the
\emph{right partial trace property} if $\ttt_{V ⊗ W}(f) = \ttt_V(\ptr_r^W(f))$,
for $V ∈ \cI$, $W ∈ \cA$, and $f ∈ \End_\cA(V ⊗ W)$,
i.e.\ $\ttt$ is invariant under right partial traces;
the \emph{left partial trace property} is similarly defined.
We say $\ttt$ is a \emph{left (resp. right) m-trace}
if it has the left (resp. right) partial trace property.
We say $\ttt$ is a \emph{two-sided m-trace}, or simply \emph{m-trace},
if it is simultaneously a left and right m-trace.

There is a close relation between admissible skein modules and (modified) traces, see \cite[Thm.\,3.1, 3.3]{Costantino:2023} and the announcement \cite{Reutter:2020}. Here we explain how to recover these results using excision.
One advantage of our approach is that it makes the relation between
the traces and topology conceptually clearer:
the partial trace properties of left/right m-traces
are ``due to'' the closing of a boundary component of the annulus.
Our proofs are also slightly more general, in that we do not require closure
of $\cS$ under direct sums or retracts,
and in the annulus case, not even closure under tensor products.

The following is essentially \cite[Thm.\,3.3]{Costantino:2023},
generalised to any dimension:

\begin{proposition}
\label{p:trace-skein-Ann}
Let $\cS ⊆ \cA$ be a full subcategory.
For $d = 1$, let $\cI = \cS$, and for $d ≥ 2$,
let $\cI = \eval{\cS}$ be the tensor-dual closure of $\cS$
(recall Definition \ref{d:tensor-dual-closure}).
Let $M = S¹ \times \DD^{d-1}$.
Then
\[
\SNX(M) \simeq \bigoplus_{X ∈ \cI} \End(X) / \Span \{fg - gf\} \ ,
\]
where the $\Span$ is over all pairs of morphisms $f : X \to Y, g : Y \to X$
for $X,Y ∈ \cI$. Hence
\[
\SNX(M)^* \simeq \{\text{traces on }\cI\} \ .
\]
\end{proposition}

Note that the quotient on the right hand side is precisely the
the 0-th Hochschild-Mitchell homology group of $\cI$, $\HH₀(\cI)$
(see \cite{Beliakova:2017,Mitchell:1972}). 

\begin{proof}
We think of $M$ as being obtained from
$M' = [0,1] \times \DD^{d-1} \simeq \DD^d$ by gluing
$N = \{1\} \times \DD^{d-1}$ to $N' = \{0\} \times \DD^{d-1}$
(with $N$ outgoing, $N'$ incoming).

Choose some point $p ∈ \Int N$.
Define $\hZI^*(N)$ to be the full subcategory of $\hZI(N)$
consisting of boundary values with exactly one marked point at $p$.
Then by closure under tensor products,
any $\cI$-admissible boundary value is isomorphic to such a boundary value,
so $\hZI^*(N) \simeq \hZI(N)$.
We can identify the objects of $\hZI^*(N)$ with $\Obj \cI$, and using Proposition~\ref{p:ideal-closure} and Theorem~\ref{t:excision} we get
\[
\SNX(M) \simeq \SNI(M)
\simeq \int^{\hZI(N)} \SNI(\DD^d;-,-)
\simeq \int^{\hZI^*(N)} \SNI(\DD^d;-,-)
\simeq \int^{\cI} \cA(-,-) \ ,
\]
where in the last isomorphism, we used Lemma \ref{l:bulk-cS}
(more specifically the observation about $\DD^d$ afterwards).
Unpacking the definition of the last coend,
we see that it is indeed equal to 
$\bigoplus_{X ∈ \cI} \End(X) / \Span \{fg - gf\}$ as in the statement of the proposition. 
Taking the dual $R$-modules,
\[
\SNX(M)^*
\simeq \Big(\int^{\cI} \cA(-,-) \Big)^*
\simeq \int_{\cI} \cA(-,-)^*
\simeq \{\text{traces on }\cI\} \ ,
\]
where the middle isomorphism follows from the left-exactness of 
$(-)^* = \Hom_R(-,R)$,
and the last isomorphism is
the definition of cyclicity:
as subspaces of $\bigoplus_{V ∈ \cI} \End_\cA(V)^*$,
the cyclic property of a trace precisely defines the end
$\int_{\cI} \End_\cA(V)^*$.
\end{proof}

\begin{proposition}[{\cite[Thm.\,3.1]{Costantino:2023}}]
\label{p:trace-skein-D2-S2}
Take $d=2$, let
$\cS ⊆ \cA$ be a full subcategory, and let
$\cI = \eval{\cS}$ be the tensor-dual closure of $\cS$.
Then there are isomorphisms
\begin{align*}
\SNX(\DD²)^* &\simeq \{\text{right m-traces on }\cI\}
\simeq \{\text{left m-traces on }\cI\} \ ,
\\
\SNX(S²)^* &\simeq \{\text{m-traces on }\cI\} \ .
\end{align*}
\end{proposition}

Note that we do not assume closure under direct sums/retracts of $\cS$ or $\cI$.
It is not hard to prove that the natural map
$\{\text{right m-traces on }\rcls{\cI}\}
\to \{\text{right m-traces on }\cI\}$
which just takes a trace $\ttt$ on $\rcls{\cI}$ to its restriction to $\cI$
is in fact a bijection; alternatively, this is follows from
Proposition \ref{p:ideal-closure},
i.e.\ from the fact that $\SNI(\DD²) \simeq \SNJ{\rcls{\cI}}(\DD²)$.

\begin{proof}
By  Remark~\ref{r:ideal-closure}
we can just consider $\cS = \cI$.
For the equalities involving $\DD²$, we prove the left m-traces case,
with the right m-traces being essentially the same.
Let $M = \DD²$ be obtained from $M' = \Ann \sqcup \DD²$ by gluing along $S^1$;
more precisely, we glue $N = S¹ \times \{-1\} ⊆ S¹ \times [-1,1] = \Ann$
to $N' = ∂\DD²$ (with $N$ outgoing, $N'$ incoming).
This gives us an inclusion $ι: \Ann ↪ M$, which induces a map
$ι_* : \SNI(\Ann) \to \SNI(M)$ that is clearly surjective.
Our goal is to understand the kernel of $ι_*$.

By Proposition \ref{p:different-skein},
\[
\SNI(M)
\simeq \int^{\XX ∈ \hZA(S¹)} \SNA(\DD²;\XX) ⊗_R  \SNI(\Ann;\XX) \ ,
\]
and the map $ι_* : \SNI(\Ann) \to \SNI(M)$ induced by inclusion
is, under the isomorphism above, given by
\[
\SNI(\Ann) \to \SNA(\DD²) ⊗_R \SNI(\Ann)
\overset{τ_{\one,\one}}{\to}
\int^{\XX ∈ \hZA(S¹)} \SNA(\DD²;\XX) ⊗_R \SNI(\Ann;\XX) \ ,
\]
where the first map is the identity on $\SNI(\Ann)$, tensored with the map $R \to \SNA(\DD²)$
given by sending $1 ∈ R$ to the empty graph $∅ ∈ \SNA(\DD²)$.
In the second map, $τ_{-,-}$ is the canonical dinatural map
to the coend,
and we implicitly identify the empty boundary value with $\one$,
a boundary value with one marked point $p ∈ S¹$,
which is labelled by $\one$. Similarly to before,
we replace $\hZJ{\cA}(S¹)$ with the equivalent full subcategory $\hZJ{\cA}^*(S¹)$ with exactly one marked point $p$, and we
identify the objects of $\hZJ{\cA}^*(S¹)$ with $\Obj \cA$.
From Example \ref{x:S1} we have $\hZJ{\cA}^*(S¹) \simeq \htr_\cA(\cA)$,
so the map above can be rewritten
\[
\ov{ι}_* : \SNI(\Ann) \to
\int^{X ∈ \htr_{\cA}(\cA)} \SNA(\DD²;X) ⊗_R \SNI(\Ann;X)
\]
We have $\SNA(\DD²;X) \simeq \cA(X,\one)$,
and from a slight generalisation of the arguments in the proof of
Proposition~\ref{p:trace-skein-Ann}, we have
$\SNI(\Ann;X) \simeq \int^{V ∈ \cI} \cA(V,XV)$.
Thus we may rewrite the above map as
\begin{align*}
\ov{ι}_*: \int^{V ∈ \cI} \cA(V,V)
&~\to~
\int^{X ∈ \htr_{\cA}(\cA)}
\Big( \cA(X,\one) ⊗_R \int^{V ∈ \cI} \cA(V,XV) \, \Big)
\\
[f] &~\mapsto~
τ_{\one,\one} ( \id_\one ⊗_R [f] ) \ .
\end{align*}
In writing $[f]$ on the right hand side, we implicitly used the unit constraint map
$V \simeq \one ⊗ V$; we will subsequently omit $τ_{\one,\one}$.
By unpacking the topological actions on skein modules,
the coend on the right is described purely algebraically as follows.
A morphism $[φ] ∈ \htr_{\cA}(\cA)(X,Y)$, represented by $φ ∈ \cA(AX,YA)$,
induces the maps
\begin{align*}
\SNI(\Ann;X) \simeq \int^{V ∈ \cI} \cA(V,XV)
&~\to~ \int^{W ∈ \cI} \cA(W,YW) \simeq \SNI(\Ann;Y)
\\
[f: V \to XV] &~\mapsto~ [φ] \lact [f] := [(φ ⊗ \id_V) ∘ (\id_A ⊗ f) : AV \to YAV]
\end{align*}
and
\begin{align*}
\SNI(\DD²;Y) \simeq \cA(Y,\one)
&\to \cA(X,\one) \simeq \SNI(\DD²;X)
\\
(g: Y \to \one) &\mapsto g \ract [φ] := \ptr_l^A(g ∘ φ) \ .
\end{align*}
Hence the coend in question can be described as the quotient
\[
\bigoplus_{X ∈ \htr_{\cA}(\cA)}
\cA(X,\one) ⊗_R
\Big(\int^{V ∈ \cI} \cA(V,XV) \Big)
\Big/
\big\{ g ⊗_R ([φ] \lact [f]) - (g \ract [φ]) ⊗_R [f] \big\} \ ,
\]
where in the denominator, $φ,f,g$ run over all ($\cA$-)morphisms
$φ : AX \to YA, f : V \to XV, g : Y \to \one$.
Since $ι_*: \SNI(\Ann) \to \SNI(M)$ is surjective,
then so is $\ov{ι}_*$, but we can also see it explicitly from the algebraic
description.
Indeed, any $h : X \to \one$ represents a morphism
$[φ_h: \one ⊗ X \to \one ⊗ \one] ∈ \htr_{\cA}(\cA)(X,\one)$
which is just $h$ composed with the
left and right unit constraints on $X$ and $\one$ respectively,
and $h = \id_\one \ract [φ_h]$,
so setting $Y = \one, g = \id_\one$, we have the relation
\[
h ⊗_R [f] = (\id_\one \ract [φ_h]) ⊗_R [f]
~\sim~ \id_\one ⊗_R [φ_h] \lact [f]
= \id_\one ⊗_R [(φ_h ⊗ \id_V) ∘ (\id_\one ⊗ f)] \ ,
\]
and the rightmost expression is in the image of $\ov{ι}_*$.

Consider the relations $N_l$ in $\SNI(\Ann)$ generated by left partial traces,
that is,
\[
N_l := \Span \{ [f] - [\ptr_l^X(f)] \;|\; f ∈ \End_\cA(XV), X ∈ \cA, V ∈ \cI \}
⊆ \SNI(\Ann)
\]
It is easy to see from the topological picture that $N_l$
should lie in the kernel of $\ov{ι}_*$.
We can show it algebraically as follows:
for $φ: AX \to \one A$ representing a morphism
$[φ] ∈ \htr_{\cA}(\cA)(X,\one)$,
we have $\id_\one \ract φ = \ptr_l^A(φ)$,
and we can always factor any $f ∈ \End_\cA(XV)$ as
$f = \wdtld{f} \lact ψ_V$,
where $\wdtld{f} := (\id_X ⊗ \wdtld{\ev}_V) ∘ f : X (VV^*) \to X \simeq \one X$
(representing a morphism in $\htr_{\cA}(\cA)(VV^*,\one)$)
and $ψ_V := \id_V ⊗ \wdtld{\coev}_V: V \to V(V^*V) \simeq (VV^*)V$,
so we have
\begin{align*}
\id_\one ⊗_R [f]
&= \id_\one ⊗_R [\wdtld{f} \lact ψ_V]
\\
&\sim (\id_\one \ract \wdtld{f}) ⊗_R [ψ_V]
\\
&= \ptr_l^X(\wdtld{f}) ⊗_R [ψ_V]
\\
&= (\id_\one \ract \ptr_l^X(\wdtld{f})) ⊗_R [ψ_V]
\\
&\sim \id_\one ⊗_R [\ptr_l^X(\wdtld{f}) \lact ψ_V]
\\
&= \id_\one ⊗_R [\ptr_l^X(f)] \ ,
\end{align*}
i.e.\ $\ov{ι}_*([f]) = \ov{ι}_*([\ptr_l^X(f)])$.
So $\ov{ι}_*$ descends to a map
\begin{align*}
\wdtld{ι}_* : \int^{V ∈ \cI} \cA(V,V) \Big/ N_l
&\to \int^{X ∈ \htr_{\cA}(\cA)}
\Big( \cA(X,\one) ⊗_R \int^{V ∈ \cI} \cA(V,XV) \Big)
\\
[f] &\mapsto \id_\one ⊗_R [f]
\end{align*}

We claim that the assignment
\[
σ: g ⊗_R [f] \mapsto [(g ⊗ \id_V) ∘ f]
\]
for $f: V \to XV, g: X \to \one$,
is a left-inverse to $\wdtld{ι}_*$.
In fact, that it is a left inverse is clear from the above explicit expression for $\wdtld{ι}_*$, and we only need to check that $\sigma$ is well-defined.
For $φ : AX \to YA, f : V \to XV, g : Y \to \one$,
we have
\begin{align*}
[((g \ract φ) ⊗ \id_V) ∘ f]
&= [(\ptr_l^A((g ⊗ \id_A) ∘ φ) ⊗ \id_V) ∘ f]
\\
&= \big[\ptr_l^A\big((((g ⊗ \id_A) ∘ φ) ⊗ \id_V) ∘ (\id_A ⊗ f)\big)\big]
\end{align*}
and
\begin{align*}
[(g ⊗ \id_{AV}) ∘ (φ \lact f)]
&= [(g ⊗ \id_{AV}) ∘ ((φ ⊗ \id_V) ∘ (\id_A ⊗ f))]
\end{align*}
and these are equal in $\SNX(\Ann)/N_l$.
Thus $\sigma$ is a left-inverse, and hence $\wdtld{ι}_*$ is injective. But $\wdtld{ι}_*$ is surjective by construction, so we arrive at
\[
\wdtld{ι}_* : \SNX(\Ann)/N_l \simeq \SNX(\DD²) \ .
\]

Taking duals (of $R$-modules),
\[
\SNX(\DD²)^* \simeq
\big( \SNI(\Ann)/N_l \big)^*
= \{ \ttt ∈ \SNI(\Ann)^* \;|\; \ttt|_{N_l} = 0 \} \ .
\]
By definition of $N_l$, traces $\ttt$ that vanish on $N_l$ are precisely left m-traces on $\cI$.

Finally, by gluing two discs to the annulus, one along each boundary component,
$\SNX(S²) \simeq \SNI(S²) \simeq \SNI(\Ann)/(N_r + N_l)$,
where $N_r$ is defined similarly to $N_l$ except using right partial traces,
and
\[
\SNX(S²)^* \simeq
\big( \SNI(\Ann)/(N_l+N_r) \big)^*
= \{ \ttt ∈ \SNI(\Ann)^* \;|\; \ttt|_{N_l} = \ttt|_{N_r} = 0 \}
\]
and traces $\ttt$ that vanish on both $N_l,N_r$
are precisely two-sided m-traces on $\cI$.
\end{proof}

%++%
\subsection{(In)finite Dimensionality}
\label{s:dimension}

In this section, we fix a full subcategory $\cI ⊆ \cA$
that is sum-retract closed, and for $d ≥ 2$ also tensor-dual closed,
\[
\cI = \begin{cases} \rcls{\cI} &; d=1 \ ,
\\
\rcls{\eval{\cI}} &; d \ge 2 \ .
\end{cases}
\]
We will investigate when skein modules are finitely generated, and when not. We start with a finiteness result.

\begin{proposition}
\label{p:finite-dim}
Suppose $\cA$ has hom-spaces that are finitely generated as $R$-modules.
If $\cI$ is generated by finite direct sums,
i.e.\ if there exist finitely many objects $X₁,...,X_k ∈ \cI$
such that all objects of $\cI$ are direct sums of $X_i$'s,
then $\SNI(M;\XX)$ is finitely generated as an $R$-module for any $\XX$
and any finitary $M$.
\end{proposition}

If $\cA$ is in addition additive then we could instead demand the existence of a generator $G \in \cI$ in the sense that $\{G\}^\oplus = \cI$. We do not assume that $\cA$ is additive, and so it could happen that neither $X_1 \oplus \cdots \oplus X_n$ is in $\cI$, nor is any other direct sum that might serve as a generator.

\begin{proof}
By our conditions on $\cI$ in this subsection, Proposition \ref{p:totally-I} states that $\SNI(M;\XX) \simeq \SNAI(M;\XX)$, and so it suffices prove that the latter is finitely generated.

Choose an uncoloured graph	$G$ in $M$
with boundary $\XX$ and such that the inclusion $G ⊆ M$
induces a surjective map on fundamental groups
(e.g.\ for closed $M$, we may take $G$ to be a 1-skeleton
of a triangulation of $M$).
Then any ($\cA$-coloured) graph with boundary value $\XX$
can be isotoped into a neighbourhood of $G$.
Note that since $M$ is finitary, $G$ can be chosen to be finite.
Let $Γ$ be a totally-$\cI$ graph.
We isotope $Γ$ into a neighbourhood of $G$,
and by further isotopy, we can arrange $Γ$ so that,
away from the internal vertices of $G$, the edges of $Γ$ run parallel to $G$.
Using $\cI$-null relations, we can merge these parallel strands into one,
so that the underlying graph of $Γ$ now agrees with $G$
except in a neighbourhood of internal vertices of $G$.
Finally, applying $\cI$-null relations in these neighbourhoods,
we can convert $Γ$ into a graph that has underlying graph exactly matching $G$.

Thus, $\SNAI(M;\XX)$ is spanned by graphs whose underlying graph is $G$.
Now let $Γ$ be such a graph.
For an internal edge $e$ of $G$, suppose $Γ$ is labelled by $Y$.
It has a direct sum decomposition into the $X_i$'s,
with projection-inclusions $p_a : Y ⇌  X_{i_a} : ι_a$.
With an $\cI$-null relation applied in the middle of $e$,
this allows us to write $Γ = \sum Γ_a$,
where $Γ_a$ has two extra vertices in the middle of $e$,
labelled by $p_a$ and $ι_a$.
We then merge these new vertices with respective vertices of $G$.
Repeating this procedure for every edge of $G$,
we arrive at a sum of graphs, such that each of them has internal edges
labelled by some $X_i$.

Thus, $\SNAI(M;\XX)$ is spanned by graphs whose underlying graph is $G$
and each internal edge is labelled by some $X_i$.
Since the label on the boundary edge is already determined by $\XX$,
this means that there are only finitely many edge labels.
Furthermore, for each of these edge labels,
the set of vertex labels is the tensor product of a finite number of
hom spaces in $\cA$, hence is finitely generated as an $R$-module.
\end{proof}

\begin{example}\label{e:fin-gen-skein}
Let $\cA$ be  a pivotal/ribbon finite tensor category over $\mathbb{C}$.
The projective ideal $\Proj(\cA)$ is 
finitely direct-sum-generated by the projective covers 
of the simple objects of $\cA$,
of which there are finitely many.
Thus by the above proposition, $\SNJ{\Proj(\cA)}(M;\XX)$ is finite dimensional for finitary $M$.
\end{example}

Next we consider admissible skein modules for $S¹ \times \DD^{d-1}$. For $X \in \cS$ denote by $Γ_{X}$ a single loop $S¹ \times\{ 0\} ⊆ S¹ \times \DD^{d-1}$ and coloured by $X$. We will show a linear independence result for these elements which allows one to estimate the dimension of the skein module from below.

\begin{proposition}
\label{p:infinite-dim}
Let $R = \kk$ be an algebraically closed field, and let
$\cA$ be in addition abelian and Hom-finite.
Let $J$ be a minimal set of indecomposable objects generating $\cI$ under direct sums.
Then the set $\{Γ_{X} \;|\; X ∈ J \}$ is linearly independent in $\SNI(S¹ \times \DD^{d-1})$, and so
$\dim_\kk \SNI(S¹ \times \DD^{d-1}) ≥ |J|$.
\end{proposition}

The proof relies on the following lemma, which is actually a slightly more general version of that statement, replacing algebraic closedness by a restriction on the endomorphism spaces.

\begin{lemma}
\label{l:infinite-dim}
Let $R = \kk$, $\cA$ and $J$ be as in Proposition~\ref{p:infinite-dim}, except that $\kk$ need not be algebraically closed. Let $J' ⊆ J$ be a subset consisting of $X ∈ J$
such that all commutators $[a,b]$ of endomorphisms $a,b ∈ \End(X)$
are non-invertible.
Then the set $\{Γ_{X} \;|\; X ∈ J'\}$ is linearly independent
in $ \SNI(S¹ \times \DD^{d-1})$.
\end{lemma}

\begin{proof}
Let $\cS$ and $\cS'$ be the full subcategories of $\cA$
with set of objects $J$ and $J'$, respectively.
In the proof of Proposition \ref{p:trace-skein-Ann}
we saw that $\SNI(M) \simeq \int^\cI \cA(-,-)$.
By Lemma \ref{l:coend-closure}
and since by assumption $\cS^\oplus \cong \cI$,
a coend over $\cI$
is the same as a coend over $\cS$,
so we have the following exact sequence
\[
\bigoplus_{i,j, f:X_j \to X_i} \cA(X_i,X_j)
\xrightarrow{ζ}
\bigoplus_{i} \cA(X_i,X_i)
\xrightarrow{q}
\int^\cS \cA(-,-) \simeq \int^\cI \cA(-,-)
\to 0
\]
where the first map $ζ$ sends $g ∈ \cA(X_i,X_j)$ to
$gf - fg ∈ \cA(X_i,X_i) ⊕ \cA(X_j,X_j)$.

We claim that $\{q(\id_{X_i}) \;|\; X_i ∈ J' \}$
is linearly independent in the coend;
clearly this set corresponds to the set $\{Γ_{X_i} \;|\; X_i ∈ J' \}$.
Let us modify the first map $ζ$ above, which would result in a larger image:
We define
\[
ζ' :
\bigoplus_{i,j, f:X_j \to X_i} \cA(X_i,X_j)
\to
\bigoplus_{i} \cA(X_i,X_i)
\]
where, writing $ζ_f'$ for the restriction of $ζ'$ to the
copy of $\cA(X_i,X_j)$ indexed by $i,j,f:X_j \to X_i$,
\[
ζ_f' : g \mapsto
\begin{cases}
fg & \text{ if } i ≠ j
\\
fg - gf & \text{ if } i = j
\end{cases}
∈ \cA(X_i,X_i)
\]
Clearly, for $i ≠ j$ we have $ζ_f(g) = ζ_f'(g) - ζ_g'(f)$,
so the image of $ζ'$ contains the image of $ζ$.

Grouping the direct sum $\bigoplus_{i,j, f:X_j \to X_i} \cA(X_i,X_j)$ by $X_i$,
observe that the image of $ζ_i'$,
the restriction of $ζ'$ to $\bigoplus_{j, f:X_j \to X_i} \cA(X_i,X_j)$,
is contained in $\cA(X_i,X_i)$.
Thus the cokernel can be written as:
\[
\bigoplus_i \bigoplus_{j, f:X_j \to X_i} \cA(X_i,X_j)
\xrightarrow{ζ'}
\bigoplus_{i} \cA(X_i,X_i)
\xrightarrow{q'}
\bigoplus_i \big( \cA(X_i,X_i) / \Ima(ζ_i') \big)
\to 0
\]

Since $\Ima(ζ) ⊆ \Ima(ζ')$, the map $q'$ factors through
a surjective map from the coend $\int^{\cS} \cA(-,-)$.
Hence, it suffices to show that $\id_{X_i} \nin \Ima(ζ_i')$
for all $X_i ∈ \cS'$, which we do by way of contradiction.

Suppose $\id_{X_i} ∈ \Ima(ζ_i')$.
Then there are finitely many morphisms
$f_l: X_{j_l} ⇌  X_i : g_l$, with $j_l ≠ i$,
and finitely many pairs of morphisms $a_m,b_m : X_i \to X_i$,
such that
\[
\id_{X_i} = \sum_l f_l g_l + \sum_m [a_m,b_m]
\]

Since $\cA$ is abelian, there exists an object
$Y_m = \Ima([a_m,b_m])$
and we may write the endomorphism $[a_m,b_m]$
as a composition $[a_m,b_m] = ι_m h_m$ factoring through $Y_m$,
with $h_m : X_i \to Y_m$ epic and
$ι_m : Y_m \to X_i$ monic;
note that $Y_m$ is not necessarily in $\cS$ or even in $\cI$.
Since $X_i ∈ \cS'$, $[a_m,b_m]$ is non-invertible by definition,
so $Y_m$ is not isomorphic to $X_i$.

Consider
\begin{align*}
g = \sum_l g_l + \sum_m h_m
&~:~ X_i ~\to~ \bigoplus_l X_{j_l} ⊕ \bigoplus_m Y_m \ ,
\\
f = \sum_l f_l + \sum_m ι_m
&~:~ \bigoplus_l X_{j_l} ⊕ \bigoplus_m Y_m ~\to~ X_i \ .
\end{align*}
By definition, $fg = \id_{X_i}$, hence $(g,f)$ is an
inclusion-projection pair that realises $X_i$ is a direct summand of
$\bigoplus_l X_{j_l} ⊕ \bigoplus_m Y_m$.
By the Krull-Schmidt theorem
(for Hom-finite abelian categories, see e.g.\ Theorems 4.2 and 5.5 and Lemma 5.1 in \cite{Krause:2014}),
the decomposition into indecomposable objects is unique up to permutation.
Hence $X_i$ must be isomorphic to some $X_{j_l}$ or $Y_m$,
but this is impossible.
\end{proof}

\begin{proof}[Proof of Proposition~\ref{p:infinite-dim}]
We check that the additional condition in Lemma~\ref{l:infinite-dim} holds automatically if $\kk$ is algebraically closed. Since $X \in J$ is indecomposable, $\End(X)$ is a local ring
(see e.g.\ \cite[Prop.\,5.4]{Krause:2014}). Denote by $\mathfrak{m} \subseteq \End(X)$ its maximal ideal, which consists precisely of the non-units in $\End(X)$. 
The quotient $\End(X) / \mathfrak{m}$ is a finite-dimensional division algebra over $\kk$, hence equal to $\kk$. For $a,b \in \End(X)$ write $a = \alpha 1 +  m$ and $b = \beta 1 + n$ with $\alpha,\beta \in \kk$ and $m,n \in \mathfrak{m}$. Then $[a,b] = mn-nm \in \mathfrak{m}$.
\end{proof}

\begin{remark}
    \begin{enumerate}
\item The set $\{Γ_{X} \;|\; X ∈ J\}$ in Proposition~\ref{p:infinite-dim} does not in general span $\SNI(S¹ \times \DD^{d-1})$.
An example in the case $d=3$ can be obtained from the non-semisimple TFT $Z$ defined in \cite{DeRenzi:2022}:

\textbf{Claim.} Let $\cA$ be a modular tensor category over $\mathbb{C}$. The set $\{Γ_{P} \;|\; P ∈ \Proj(\cA)\}$ spans $\SNJ{\Proj(\cA)}(S¹ \times \DD^{2})$ if and only if $\cA$ is semisimple.

\begin{proof}
Fix a set $\mathrm{Irr}(\cA)$ of (representatives of isomorphism classes of) simple objects in $\cA$. Suppose $\cA$ is semisimple. Then $\Proj(\cA) = \cA$, and the set $\{Γ_{U} \;|\; U ∈ \mathrm{Irr}(\cA) \}$ is actually a basis of $\SNJ{\cA}(S¹ \times \DD^{2})$.

The more interesting statement is the converse direction. Thus let $\cA$ be non-semisimple. 
The indecomposable projectives in $\cA$ are given by projective covers $P_U$ of simple objects $U \in \cA$, and so we get a minimal set of indecomposables $\{ P_U | U \in \mathrm{Irr}(\cA) \}$ generating $\Proj(\cA)$.
The subspace spanned by $\{Γ_{P} \;|\; P ∈ \Proj(\cA)\}$ in $\SNJ{\Proj(\cA)}(S¹ \times \DD^{2})$ thus has exactly dimension $|\mathrm{Irr}(\cA)|$. We will now show that $\SNJ{\Proj(\cA)}(S¹ \times \DD^{2})$ surjects onto a vector space of dimension strictly larger than $|\mathrm{Irr}(\cA)|$.

Namely, the state space $Z(T^2)$ of the TFT constructed from $\cA$ on the 2-torus $T^2$ is spanned by $\Proj(\cA)$-admissible graphs in the solid torus $S¹ \times \DD^{2}$. The gluing properties of the TFT imply that $Z$ respects $\Proj(\cA)$-null-relations, and we get a surjective $\mathbb{C}$-linear map
\[
    \SNJ{\Proj(\cA)}(S¹ \times \DD^{2}) \to Z(T^2) \ .
\]
There is a (non-canonical) isomorphism $Z(T^2) \cong \cA(\int^{X \in \cA} X^* \otimes X, \one)$, cf.\ \cite[Sec.\,4]{DeRenzi:2022}. 
Since $\cA$ is modular, we have $\int^{X \in \cA} X^* \otimes X \cong \int_{X \in \cA} X^* \otimes X$ as objects in $\cA$.
But by \cite[Thm.\,5.9]{Shimizu:2017}), for non-semisimple $\cA$ the dimension of $\cA(\int_{X \in \cA} X^* \otimes X, \one)$ is strictly larger than $|\mathrm{Irr}(\cA)|$.
\end{proof}

To obtain a spanning set, define $Γ_{X}(f)$ to be the $X$-coloured loop $Γ_{X}$ together with a vertex labelled by $f \in \End(X)$ placed on it. 
By the argument in the proof of Proposition~\ref{p:finite-dim}, the set $\{Γ_{X}(f) \;|\; X ∈ J, f \in \End(X) \}$ spans $\SNI(S¹ \times \DD^{d-1})$. 
\item
For the choice $\cI=\cA$, Proposition~\ref{p:infinite-dim} states that skein modules for $S¹ \times \DD^{d-1}$ are infinite dimensional if $\cA$ has infinitely many pairwise non-isomorphic indecomposable objects. This happens often in representation theoretic examples, even if $\cA$ is in addition a finite tensor category. Example are
$\cA = \mathrm{Rep}(\ov{U}_q(\mathfrak{sl}_2))$, the category of
finite dimensional representations of the restricted quantum group, see \cite{KondoSaito:2010},
or $\cA = \bigwedge V\text{-}\smod$, 
where $V$ is an odd vector space of dimension at least 2 -- see Example~\ref{e:also-disc-infinite} below for more details. 
\end{enumerate}
\end{remark}

Next we give a lower bound on the dimensions of admissible skein modules of surfaces by providing a surjective map to the admissible skein module of an annulus. The proof will make use of admissible skein modules in one dimension higher, hence the stronger requirement on $\cA$ in the statement:

\begin{proposition}
\label{p:surface-inf-dim}
Let $\cA$ be suitable for $d=3$ and let $R,J$ be as in Proposition~\ref{p:infinite-dim}.
Then for any connected surface $Σ ≠ \DD²,S²$, we have $\dim \SNI(Σ) ≥ |J|$.
\end{proposition}

\begin{proof}
Let us first assume that $Σ$ is a finitary connected surface.
Then if we ignore corners, $Σ$ is diffeomorphic to
some standard surface $S_{g,n}$ of genus $g$ with $n$ open balls cut out,
and since $Σ ≠ \DD², S²$, we have $g ≥ 1$ or $n ≥ 2$ or both.
By Lemma \ref{l:ignore-corners},
$\SNI(Σ) \simeq \SNI(S_{g,n})$, so we will assume $Σ = S_{g,n}$.
Proposition~\ref{p:infinite-dim} already applies to the annulus,
i.e.\ $g = 0, n = 2$. Our strategy is to compare
$\SNI(Σ)$ to $\SNI(\Ann)$.

Suppose $g ≥ 1$.
Write $Σ = S_{1,1} ∪_{S¹} S_{g-1,n+1}$.
We may embed $Σ$ in $\RR^3$ so that there is a ball $B$
such that $S_{g-1,n+1} = Σ ∩ B, S_{1,1} = Σ \backslash B$.
Let $D$ be a 2-ball spanning the $S¹$ between $S_{1,1}$ and $S_{g-1,n+1}$
(e.g. we can take $D ⊆ ∂B$ to be one of the components
of $∂B \backslash Σ$).
After some smoothing, $S_{1,1} ∪_{S¹} D$ is an embedded torus.
Let $M₁$ be the solid torus obtained from filling it,
and let $γ$ be a longitude on this torus that avoids $D$,
($γ$ should be isotopic to a core of the solid torus $M₁$).

Let $M = Σ \times [-1,1]$,
which we also identify with a neighbourhood of the embedded $Σ$.
By Lemma \ref{l:thickened-mfld}, we have $\SNI(Σ) \simeq \SNI(M)$.
Let $M' = M₁ ∪ M ∪ B$. Clearly $M'$ is a solid torus,
and it is isotopic to a solid torus neighbourhood $M_{γ}$ of $γ$;
this isotopy can be chosen to be fixed on $γ$.
Thus, the inclusion of 3-manifolds induces maps
\begin{equation}
\label{e:M-gamma}
\SNI(M_{γ}) \to \SNI(M) \to \SNI(M') \simeq \SNI(M_{γ})
\end{equation}
whose composition is the identity,
so the first map is injective, and in particular,
\begin{equation}
\label{e:M-gamma-dim}
\dim \SNI(Σ) = \dim \SNI(M) ≥ \dim \SNI(M_{γ}) = \dim \SNI(\Ann) ≥ |J|
\end{equation}
where $\SNI(M_{γ}) \simeq \SNI(\Ann)$ by Lemma \ref{l:thickened-mfld},
and $\dim \SNI(\Ann) ≥ |J|$ 
by Proposition \ref{p:infinite-dim}.

Now suppose $g = 0$, so that we have $n ≥ 2$.
For $n = 2$, $S_{0,2} = \Ann$, so we may assume $n ≥ 3$.
The argument is similar to above.
We write $Σ = S_{0,2} ∪_{[0,1]} S_{0,n-1}$,
and embed $Σ$ in $\RR^3$ in an unknotted fashion,
e.g.\ as a ball with $n-1$ holes cut out,
so that there is a ball $B$ such that
$S_{0,n-1} = Σ ∩ B, S_{0,2} = Σ \backslash B$.
Let $γ$ be the core circle of $S_{0,2} = \Ann$.
Let $M = Σ \times [-1,1]$ be a thickening,
and $M' = M ∪ B$.
Then (\ref{e:M-gamma}), (\ref{e:M-gamma-dim}) work the same.

For a non-finitary surface, we present it as a colimit of finitary surfaces,
$Σ = \Colim \{ Σ₁ ⊆  Σ₂ ⊆ ... \}$,
with $Σ₁ = S_{g,n}$ such that either $g ≥ 1$ or $n ≥ 2$ or both.
Note that the genus of the $Σ_k$'s is non-decreasing with $k$.
If $Σ_k$ has genus at least one for some $k$,
then we start from $Σ_k$ and reindex it to $Σ₁$,
and we choose $γ$ for $Σ₁$ as above.
Treating $γ$ as loops in subsequent $Σ_k$'s,
the above arguments are compatible under the inclusion maps,
that is, for each $k$ the inclusions $j_k : \SNX(M_{γ}) ↪ \SNX(Σ_k)$
produced by the arguments above commute with
the inclusion-induced maps $i_k: \SNX(Σ_k) \to \SNX(Σ_{k+1})$,
i.e.\ $j_{k+1} = i_k ∘ j_k$.
Similarly for the case when $Σ_k$ has genus 0 for all $k$.
Hence, in the colimit, we still have the inclusion
$\SNI(M_{γ}) ⊆ \Colim \SNI(Σ_j) \simeq \SNI(Σ)$,
where the isomorphism is given by Proposition \ref{l:finitary-colimit},
and thus the dimension lower bound follows.
\end{proof}

Proposition \ref{p:surface-inf-dim} excludes the surfaces $\DD², S²$, 
which are nonetheless interesting because of their relation to modified traces (Proposition~\ref{p:trace-skein-D2-S2}).
Below, we give an example of a finite pivotal tensor category $\cA$ (in fact a ribbon category), together with an ideal $\cI$ such that even
$\SNI(\DD²)$ and $\SNI(S²)$ are infinite-dimensional.
Note that the two canonical choices $\cI=\cA$ and $\cI = \Proj(\cA)$ result in $\SNI(\DD²)$ and $\SNI(S²)$ being finite-dimensional: for $\cI=\cA$ one gets the usual skein modules which for $\DD²$ and $S²$ are spanned by the empty skein, and for $\cI = \Proj(\cA)$ one is in the situation of Example~\ref{e:fin-gen-skein}. Thus we need to find an $\cA$ which has a suitable intermediate ideal $\Proj(\cA) \subsetneq \cI \subsetneq \cA$.

\begin{example}\label{e:also-disc-infinite}
This example is adapted from \cite[Sec.\,4.4]{Berger:2023}.
Take $R = \kk$ to be an algebraically closed field.
Let $W$ be a vector space with basis $\{a,b,c\}$,
and $V ⊆ W$ be the 2-dimensional subspace spanned by $a,b$, and
let $A = \bigwedge W$ be the free super algebra generated by $W$
placed in odd degree, and let $\ov{A} = \bigwedge V$ be the subalgebra generated by $V$.
Then $A$, $\ov{A}$ are in fact Hopf algebras in super-vector spaces, and they can furthermore be equipped with an $R$-matrix and a ribbon element (we only need these structures for $A$).
Let $\cA = A\text{-}\smod, \ov{\cA} = \ov{A}\text{-}\smod$
be the category of super modules over $A,\ov{A}$ respectively.
Since $A \in \mathrm{sVec}$ is a ribbon Hopf algebra, $\cA$ is a finite ribbon category over $\kk$.
The inclusion $\ov{A} ↪ A$ induces
the restriction functor $\Res^* : \cA \to \ov{\cA}$ which is monoidal.
Then we take $\cI = (\Res^*)^\inv(\Proj(\ov{\cA}))$,
the pullback ideal of the projective ideal of $\ov{\cA}$ along $\Res^*$,
i.e.\ $\cI ⊆ \cA$ is the full subcategory consisting of objects $X$
such that $\Res^*(X) ∈ \Proj(\ov{\cA})$.
Since the only projective modules over $\ov{A}\text{-}\smod$
are direct sums of $\ov{A}$ and $Π\ov{A}$, where $Π(-)$ shifts the parity,
modules in $\cI$ can be thought of as such a direct sum
with an action of $c$ that (super) commutes with the actions of $a$ and $b$.

Write $L_{X,f,g} \in \SNI(\DD²)$ for an $X$-labelled line segment embedded in $\DD^2$ with a 1-valent vertex at either end, labelled by $f \in \cA(\one,X)$ and $g \in \cA(X,\one)$, respectively.

\medskip

\noindent
\textbf{Claim:} There is a family $\{ X(\lambda,\mu) \}_{\lambda,\mu \in \kk}$ of objects in $\cI$, as well as $f_{\lambda,\mu},g_{\lambda,\mu}$ as above, such that the vectors $\{ L_{X(\lambda,\mu),f_{\lambda,\mu},g_{\lambda,\mu}}\}_{\lambda,\mu \in \kk}$ are linearly independent in $\SNI(\DD²)$.

\medskip

Since $\kk$ is algebraically closed, it is necessarily infinite, and so the claim implies that $\SNI(\DD²)$ is infinite-dimensional.
Note that $\SNI(S²) \simeq \SNI(\DD²)$ because $\cA$ is ribbon 
(use that left/right modified traces and m-traces agree for ribbon categories together with Proposition~\ref{p:trace-skein-D2-S2}), so the same follows for $S²$.

\begin{proof}[Proof of claim.]
We present the ball $\DD²$ as two balls glued along an interval;
then by excision and similar methods to the proof of
Proposition~\ref{p:trace-skein-D2-S2},
$\SNI(\DD²) \simeq \int^{X ∈ \cI} \cA(X,\one) ⊗ \cA(\one,X)$.
Write $F(X,Y) = \cA(X,\one) ⊗ \cA(\one,Y)$ for simplicity.

The key idea is to produce some $\kk$-linear endomorphisms
out of the action of $c$ that are natural in $X ∈ \cI$,
and relate them to elements of $F(X,X)$.
The eigenvalues these endomorphisms will force them to be
linearly independent in the coend, and will result in the lower bound for the dimension.

\smallskip

\noindent
\emph{Step 1: Remove projective direct summands.} 
Let $\cI' ⊆ \cI$ be the full subcategory consisting of objects
without projective direct summands.
Then the coend of $F$ over $\cI$ is the same when restricted to $\cI'$,
\[
\SNI(\DD²) \simeq \int^{X ∈ \cI'} F(X,X) \ .
\]
To see this, let $\cI_{\mathrm{indec}} ⊆ \cI$ be the full subcategory of indecomposable objects. By Lemma~\ref{l:coend-closure}, the coend over $\cI$ is the same as the coend over $\cI_{\mathrm{indec}}$. Let $P$ be a projective indecomposable with respect to $A$. Then $P \cong A$ or $P \cong \Pi A$. Since $W$ is odd-dimensional, in each case either $\cA(\one,P) = 0$ or $\cA(P,\one) = 0$.
Omitting indecomposable projectives also does not remove any relations. E.g.\ for $a \otimes b \in F(X,P) = \cA(X,\one) ⊗ \cA(\one,P)$ and $h : P \to X$, the relation in 
$F(X,X) \oplus F(P,P)$ is $a \otimes (h \circ b) - (a \circ h) \otimes b$. But $F(P,P)=0$, and $h \circ b : \one \to P \to X$ with $X$ non-projective is always zero in our example.

\smallskip

\noindent
\emph{Step 2: Simplify $F(X,Y)$.} 
Let $B₀(-) = \cA(\one,-) : \cI' \to \Vect_\kk$. Then
$B₀(X)$ is the subspace of $X$ which is even and annihilated by $a,b,c$. Since as an $\ov{A}$-module, $X$ is a sum of $\ov{A}$'s and $\Pi\ov{A}$'s, $B₀(X)$ is also the even part of the image of the action of $ab$ on $X$. Indeed, since $X$ is not projective as an $A$-module, $abc$ automatically acts as zero. We get a natural isomorphism $F(X,Y) \simeq \cA(X,B₀(Y))$.
Acting with $ab$ and projecting to the even subspace defines a surjective $A$-intertwiner $π₀: X \to B₀(X)$. Pick a $\kk$-linear inverse $\sigma_0 : B₀(X) \to X$ and define $\tau : \cA(X,B₀(Y)) \to \Vect_\kk(B₀(X),B₀(Y))$ via $\varphi \mapsto \varphi \circ \sigma_0$. One checks that $\tau$ is independent of the choice of $\sigma_0$ and is in fact an isomorphism. Altogether $F(X,Y) \cong \Vect_\kk(B₀(X),B₀(Y))$, naturally in $X,Y$.

\smallskip

\noindent
\emph{Step 3: Reduce the coend.} 
The action of $ac$ on $X$ defines a morphism $ψ_{ac}' ∈ \cA(X,B₀(X))$
that is natural in $X$, i.e.\ a natural transformation $\id_{\cA} \to B₀$.
Then we have a natural endomorphism $ψ_{ac} := τ(ψ_{ac}') : B₀ \to B₀$.
Similarly, from the action of $bc$ we get $ψ_{bc}$.
The coend 
$\int^{X ∈ \cI'} \Vect_\kk(B₀(X),B₀(X))$ surjects onto the quotient
\begin{align*}
    C := \bigoplus_{X ∈ \cI'} \Vect_\kk(B₀(X),B₀(X)) \,\Big/\,
\Big\{
	g ∘ Ψ - Ψ ∘ g \;\Big|\;
& 
	Ψ : B₀(X)
	\overset{\kk-\text{linear}}{\xrightleftharpoons{\hspace{1cm}}}
	B₀(Y) : g,
	\text{ and}
 \\[-.9em]
&	\quad g ∘ ψ_{ac} = ψ_{ac} ∘ g
~,~
	g ∘ ψ_{bc} = ψ_{bc} ∘ g
~\Big\}\ .
\end{align*}
Indeed, for the coend the quotient would only contain terms where $g$ is the form $h \circ (-) : \one \to X$, where $h : Y \to X$ is an $A$-intertwiner.

\smallskip

\noindent
\emph{Step 4: Trace on eigenspaces.} 
We observe that $ψ_{ac} , ψ_{bc} : B_0(X) \to B_0(X)$ commute, and so we can decompose each $B_0(X)$ into joint generalised eigenspaces $B_0(X)_{\lambda,\mu}$, where $\lambda \in \kk$ is the generalised eigenvalue of $ψ_{ac}$ and $\mu$ that of $ψ_{bc}$:
$B_0(X) = \bigoplus_{\lambda,\mu \in \kk} B_0(X)_{\lambda,\mu}$. We define a linear map $\mathrm{tr}_{\alpha,\beta} : C \to \kk$ by setting, for $[f] \in C$, $f : B_0(X)_{\lambda,\mu} \to B_0(X)_{\lambda',\mu'}$,
\[
\mathrm{tr}_{\alpha,\beta}\big( [f] \big) 
= \begin{cases}
\mathrm{tr}_{B_0(X)_{\alpha,\beta}}(f) &; \lambda=\lambda'=\alpha \,,\, \mu = \mu' = \beta \ ,
\\
0 &; \text{else} \ .    
\end{cases}
\]
To see that $\mathrm{tr}_{\alpha,\beta}$ is well-defined, first note that
a relation $g ∘ Ψ - Ψ ∘ g$ arising from $Ψ : B₀(X)_{\lambda,\mu}
\xrightleftharpoons{~~}	B₀(Y)_{\lambda',\mu'} : g$ can only be nontrivial for $\lambda=\lambda'$ and $\mu=\mu'$ as $g$ preserves generalised eigenspaces, and in this case well-definedness follows from cyclicity of the trace.

\smallskip

\noindent
\emph{Step 5: Linearly independent elements in $C$.} 
For $\lambda,\mu \in \kk$ consider the $A$-module $X(\lambda,\mu)$
defined as $X(\lambda,\mu) := \ov{A}$ as an $\ov{A}$-module, with $c$ acting as $-μa + λb$. Then $B_0(X(\lambda,\mu))$ is one-dimensional, and $\psi_{ac} = \lambda \id$, $\psi_{bc} = \mu \id$ on $B_0(X(\lambda,\mu))$. It follows that
\[
\mathrm{tr}_{\alpha,\beta}\big( [\id_{X(\lambda,\mu)}] \big) 
=
\delta_{\alpha,\lambda} \delta_{\beta,\mu} \ .
\]
In particular the elements $[\id_{X(\lambda,\mu)}]$ are linearly independent in $C$, and hence also 
in $\SNI(\DD²)$.
\end{proof}
\end{example}

%++%
\appendix
\section{Appendix}

\subsection{List of Notation}
\label{s:notation}

Summary of notation/terminology:
\begin{itemize}

\item $M$: $d$-manifold, $N$: $(d-1)$-manifold.
\hfill \pageref{pg:manifold}

\item Boundary piece of $M$: closure of connected top dimensional stratum of $∂M$
\hfill \pageref{pg:boundary-piece}

\item Isotopy $φ^{\bullet}$ is $C^{∞}$-constant on $V ⊆ M$: $φ^t|_V = \id$
up to all $k$-jets, for all $t$
\hfill \pageref{pg:Cinfty-isotopy}

\item $\ov{Γ}$, (uncoloured-)graph in $M$:
finite smoothly embedded oriented graph in $M$,
with framing on normal bundle
\hfill \pageref{pg:uncoloured-graph}

\item Configuration of marked points on $N$:
finite set of framed points $B ⊆ N \backslash ∂N$
\hfill \pageref{pg:config-marked-points}

\item $R$: commutative ring
\hfill \pageref{p:ring-R}

\item $\cA$: $R$-linear essentially small category with extra structure as Table~\ref{t:A-structure-vs-d}, depending on the dimension $d$ of the manifolds that skein modules are considered for
\hfill \pageref{p:cat-A}

\item $\cS,\cT$: typically denotes a full subcategory of $\cA$

\item $\cI$ ideal: full subcategory of $\cA$, closed under $⊗$, duals, direct sum, and retracts
\hfill \pageref{d:tensor-ideal}

\item $χ$ an $\cA$-edge-colouring on $\ov{Γ}$: assign object to edges of $\ov{Γ}$
\hfill \pageref{pg:edge-colouring}

\item $χ$ is $\cS$-admissible:
	at least one $\cS$ label in each component of $M$
\hfill \pageref{pg:edge-colouring-admissible}

\item $\Col(\ov{Γ},χ, v) ∈ \Rmod$: space of colourings for vertex $v$,
given by $\cA(χ_{e₁}^{[*]} ⊗ \cdots ⊗ χ_{e_k}^{[*]},
χ_{e₁'}^{[*]} ⊗ \cdots ⊗ χ_{e_l'}^{[*]})$
\hfill \pageref{pg:vertex-colouring}

\item vertex colouring of $(\ov{Γ}, χ)$:
	assignment $φ$ of a colouring $φ_v ∈ \Col(\ov{Γ},χ,v)$
	for each vertex $v$ of $\ov{Γ}$
\hfill \pageref{pg:vertex-colouring-all}

\item $\cA$-coloured graph: $Γ = (\ov{Γ}, χ, φ)$
\hfill \pageref{pg:coloured-graph}

\item $\cS$-admissible graph: $χ$ is $\cS$-admissible
\hfill \pageref{pg:admissible-graph}

\item Boundary value: $\VV = (B, \{χ_b\})$,
	configuration of marked points $B$
	with object assignments $χ_b ∈ \cA$ for each point $b ∈ B$;
	write $∂Γ$ for boundary value $Γ ∩ ∂M$.
\hfill \pageref{pg:boundary-value}

\item $\cS$-admissible boundary value $\VV$: at least one $\cS$ label in each component of $N$
\hfill \pageref{pg:bv-admissible}

\item Bulk-$\cS$-admissible:
for each component $M_j$ of $M$, $\VV ∩ ∂M_j$ has at least one $\cS$-labelled point
\hfill \pageref{pg:bv-admissible-bulk}

\item $\VGraph(M;\VV)$: free $R$-module spanned by graphs with b.v. $\VV$
\hfill \pageref{e:VGraph}

\item $\VGraphX(M;\VV)$: free $R$-module spanned by $\cS$-admissible graphs with b.v. $\VV$
\hfill \pageref{e:VGraphX}

\item $\eval{Γ}_D$: (Reshetikhin-Turaev) evaluation of a graph in a ball $D ⊆ M$
\hfill \pageref{e:RT-evaluation}

\item $\PrimNull(M,D;\VV)$: primitive null graphs w.r.t.\ $D$, i.e.\ linear combinations $\sum a_j Γ_j$ such that $Γ_j$ agree outside $D$,
	and $\sum a_j \eval{Γ_j}_D = 0$
\hfill \pageref{pg:primitive-null}

\item $\Null(M,U;\VV)$, null graphs in $U$: sums of primitive null graphs
	w.r.t.\ $D$'s such that $\Int D ⊆ U$.
\hfill \pageref{pg:null}

\item Skein module $\SNJ{}(M;\VV) := \VGraph(M;\VV) / \Null(M;\VV)$
\hfill \pageref{pg:skein}

\item $\PrimNullX(M,D;\VV)$: primitive $\cS$-null graphs w.r.t.\ $D$, i.e.\ primitive null graphs $\sum a_j Γ_j$ w.r.t.\ $D$ such that
	the graphs $Γ_j \backslash D$ are also $\cS$-admissible
\hfill \pageref{pg:S-primitive-null}

\item $\NullX(M,U;\VV)$, $\cS$-null graphs in $U$: sums of primitive $\cS$-null graphs
	w.r.t.\ several discs such that $\Int D ⊆ U$.
\hfill \pageref{pg:S-null}

\item $\cS$-admissible skein module $\SNX(M;\VV) := \VGraphX(M;\VV) / \NullX(M;\VV)$
\hfill \pageref{d:S-skein}

\item $ι_{\cS} : \SNX \to \SNJ{}$: induced by forgetting $\cS$-admissibility
\hfill \pageref{pg:iota-S}

\item $ι_{\cS ⊆ \cT} : \SNX \to \SNXp$: induced by inclusion of
	subcategories $\cS ⊆ \cT$
\hfill \pageref{pg:iota-S-T}
\end{itemize}

\subsection{Manifolds with Corners}
\label{s:mfld-corner}

The following exposition on manifolds with corners
is largely based on \cite{Hajek:2014} (which in turn follows \cite{Lee-book,Joyce:2012}).
Manifolds with corners have coordinate charts
that are locally modelled on $\RRO{d}$, with $\RRO{} = [0,∞)$,
where smooth maps from a subset $A ⊆ \RRO{d}$ to $\RRO{m}$
are functions that extend to a smooth map from an open neighbourhood
$U_A ⊆ \RR^d$ of $A$ as a subset in $\RR^d$.
More precisely, from \cite[Sec.\,2.1]{Hajek:2014},
a $d$-dimensional manifold with corner
is a Hausdorff, second-countable topological space $M$
with a maximal collection of charts $(U ⊆ M, φ: U \to \RRO{d})$
such that transition maps between charts are diffeomorphisms
(smooth map with smooth inverse).
A smooth map $f: M \to P$ between manifolds with corners
is a continuous map that is smooth in coordinate charts.
Our manifolds with corners will be assumed to have
finitely many connected components, but they need not be compact.

The \emph{depth} of $p ∈ M$ is the number of 0 coordinates of $φ(p)$
in a coordinate chart \cite[Def.\,6]{Hajek:2014};
note that smooth maps do not necessarily preserve depth,
but diffeomorphisms do, hence depth is independent of coordinate chart.
For $0 ≤ k ≤ d$, a $k$-stratum is a connected component
of the subspace of $M$ consisting of depth-$k$ points,
and a $k$-face is the closure of a $k$-stratum.
There is a different notion of $k$-face,
the (iterated) boundary $∂^kM = ∂(\cdots(∂M)\cdots)$
as described in \cite[Def.\,2.6, Prop.\,27]{Joyce:2012}:
for $k = 1$, $∂M$ consists of pairs $(p,β)$,
where $p$ has depth at least 1,
and $β$ is a local choice of 1-stratum
(see \cite[Def.\,2.5]{Joyce:2012});
in particular, if $p$ has depth $l$, then there are $l$ choices of $β$.
For example, the boundary of $\RRO{d}$ consist of $d$ components,
each obtained by setting one coordinate to 0.
We call a connected component of $∂M$ a \emph{boundary piece of $M$},
and refer to its interior as an \emph{open boundary piece}.
The sets of boundary pieces and 1-faces of $M$ are naturally in bijection,
and there is a natural map from each boundary piece
to its corresponding 1-face which is a diffeomorphism
(onto the corresponding 1-stratum) when restricted to its interior.
We also refer to 1-strata as open boundary pieces,
as there is no real ambiguity by the previous observation.

An instructive example showing the difference between these two notions
is the teardrop \cite[Fig.\,2.1]{Joyce:2012},
defined as $T = \{x ≥ 0, y² ≤ x² - x⁴\} ⊆ \RR²$,
where the boundary $∂T$ according to \cite{Joyce:2012} is a closed interval,
while the 1-face according to \cite{Hajek:2014} is topologically a circle;
this distinction will be irrelevant for the types of boundary components
that we are particularly interested in, namely embedded boundary pieces
(i.e.\ boundary pieces whose natural map to its corresponding 1-face
is a homeomorphism).

The \emph{interior $\Int(M)$} is the submanifold of depth 0 points.
An \emph{orientation} on $M$ is an orientation of its interior.

A \emph{neat embedding $ι: M^d ↪ P^n$} is a smooth map
which is a homeomorphism onto its image
and locally of the form $\RRO{d} \ni x \mapsto (x,a) ∈
\RRO{d} \times \RRO{n-d} = \RRO{n}$,
where $a$ is some fixed vector
(see \cite[Def.\,15]{Hajek:2014}, which gives a coordinate-independent definition
of neat embedding, and \cite[Thm.\,1, Def.\,15]{Hajek:2014},
which proves the existence of the above local form).
It follows that neat embeddings preserve the depth of points.

A \emph{collar neighbourhood} of an embedded boundary piece $N ⊆ ∂M$,
or more generally
of a union $N$ of pairwise disjoint embedded boundary pieces,
is a neat and proper embedding $ξ: [0,1] \times N ↪ M$ that is a diffeomorphism
onto its image, is the identity when restricted to $\{0\} \times N$
(slightly modified from \cite[Def.\,35]{Hajek:2014},
in particular a half-open interval $[0,ε)$ is used instead of $[0,1]$);
we also use collar neighbourhoods of the form $ξ: [-1,0] \times N ↪ M$.
Given a neatly embedded $L ⊆ M$ whose embedding is proper
(i.e.\ preimages of compact sets are compact),
a \emph{tubular neighbourhood of $L$}
is a neat and proper embedding $\wdtld{ι}: [-1,1] \times L ↪ M$
that restricts to the identity on $\{0\} \times L$.

The properness condition on the extension of $\wdtld{ι}$
means that $[-1,1] \times L$ only meets the rest of $M$ at
$\{-1,1\} \times L$, or more precisely,
$\wdtld{ι}([-1,1] \times L)$ is open and closed in
$M \backslash \wdtld{ι}(\{-1,1\} \times L)$.
In particular, it ensures that the vector field $X = ι_*( (f ⊗ \id_L) ⋅ ∂_t)$
used in Section~\ref{s:excision-proof}
to construct the translations $θ_m^{\bullet}$ is continuous and indeed smooth,
and it is also used the construction of $M'$
described in the following paragraph.
Likewise for collar neighbourhoods, the properness condition
implies that we can construct a ``squishing map'' $M \to M$
that is the identity away from the collar neighbourhood
and is of the form $(a,x) \mapsto (g(a),x)$ in the collar neighbourhood,
where $g: [0,1] \to [0,1/2] ⊆ [0,1]$,
which we may use in Section~\ref{s:skein-module-as-functor}
to construct the diffeomorphism
$[0,1] \times N ∪_{\{0\} \times N \sim N} M \simeq M$
as well as in Section~\ref{s:excision-cylinder}
to construct the bimodule structure.

For such $L ⊆ M$ with tubular neighbourhood $\wdtld{ι}$,
we construct the manifold with corners $M'$ that is
\emph{obtained by cutting $M$ along $L$},
as $M' = (M \backslash \wdtld{ι}((0,1/2] \times L)
∪_{\wdtld{ι}((1/2,1) \times L) \sim (1/2,1) \times L} [0,1) \times L$;
denote $L_- = \wdtld{ι}(\{0\} \times L) ⊆ M \backslash \wdtld{ι}((0,1/2] \times L)$
and $L_+ = \{0\} \times L ⊆ [0,1] \times L$.
Basically, we remove $L$ from $M$ and attached two copies of $L$,
one copy $L_-$ to the $\wdtld{ι}([-1,0) \times L)$ end,
the other copy $L_+$ to $\wdtld{ι}((0,1] \times L)$.
The natural quotient map $π: M' \to M$ identifying $L_-$ and $L_+$
is a diffeomorphism when restricted to $M' \backslash (L_- ∪ L_+)$,
and is a diffeomorphism when restricted to $L_-$ and $L_+$.
There are unique collar neighbourhoods
$ξ_-: [-1,0] \times L_- \to M', ξ_+: [0,1] \times L_+ \to M'$
such that, when composed with $π$,
become the restrictions of $\wdtld{ι}$ to
$[-1,0] \times L$ and $[0,1] \times L$,
or more precisely,
$π ∘ ξ_- = \wdtld{ι} ∘ (\id_{[-1,0]} \times π|_{L_-})$
and similarly for $ξ_+$.

\subsection{Proof of Lemma \ref{l:isotopy-null}}
\label{s:appendix-isotopy-null}

Let us recall the following isotopy extension theorem:

\begin{proposition}[{\cite[Thm.\,1.3]{Hirsch:2012}}]
\label{p:isotopy-extension-submfld}
Let $V ⊆ M$ be a compact submanifold and $F^{\bullet}: V \to M$
an isotopy of $V$ such that $F^t(V) ∩ ∂M = ∅$ for all $t$.
Then $F$ extends to an ambient isotopy of $M$ with compact support.
\end{proposition}

The \emph{track} of an isotopy $F^{\bullet} : U \to M$
of a subset $U ⊆ M$
is the map $\hat{F}: U \times [0,1] \to M \times [0,1]$,
where $\hat{F}(x,t) = (F(x),t)$.

\begin{proposition}[{\cite[Thm.\,1.4]{Hirsch:2012}}]
\label{p:isotopy-extension-2}
Let $U ⊆ M$ be an open set and $A ⊆ U$ a compact set.
Let $F^{\bullet}: U \to M$ be an isotopy of $U$
such that the image of its track $\hat{F}(U \times [0,1]) ⊆ M$ is open.
Then there is an ambient isotopy of $M$ having compact support,
which agrees with $F$ on a neighbourhood of $A \times [0,1]$.
\end{proposition}

\begin{definition}
Given a manifold $M$ with an open cover $\{U_i\}$
and diffeomorphisms $φ,ψ: M \to M$,
a \emph{sequence of moves $(h_j^{\bullet})_{j = 1,...,k}$ connecting $φ$ to $ψ$}
is a sequence of ambient isotopies $h_j^{\bullet}$ of $M$
such that each $h_j^{\bullet}$ begins at $\id_M$
and is supported in some $U_{i_j}$,
and $ψ = h_k¹ ∘ \cdots ∘ h₁¹ ∘ φ$.
We say that $(h_j^{\bullet})$ is \emph{constant (or fixed) on a subset $A ⊆ M$}
if each $h_j^{\bullet}$ is fixed on $A$.
\end{definition}

\begin{lemma}
\label{l:isotopy-moves}
Let $M$ be a compact manifold with smooth boundary
(no corners, possibly empty boundary),
and let $A ⊆ M$ be a subset.
Let $\{U_i\}$ be a finite open cover of $M$.
Suppose $ψ: M \to M$ is a diffeomorphism of $M$
that is isotopic relative $A$ to $\id_M$, i.e.\ there is an ambient isotopy
$h^{\bullet}: M \to M$ with $h⁰ = \id_M$ and $h¹ = ψ$,
and $h^{\bullet}$ is constant on $A$.
Then there exists a sequence of moves supported on $\{U_i\}$
that connects $\id_M$ to $ψ$, and that is constant on $A$.
\end{lemma}

In particular, if the open cover consists of balls,
then, since the evaluation $\eval{⋅}$ is isotopy invariant,
it follows that each move yields a null relation.

\begin{proof}
We induct on the size $n = \{U_i\}$ of the open cover.
The base case is trivial, so we proceed with the inductive step.

Let $Z,W,V ⊆ M$ be open sets and
$f$ a smooth function on $M$ such that
\[
\ov{M \backslash (U₁ ∪ \cdots ∪ U_{n-1})}
⊆ Z ⊆ \ov{Z} ⊆ W ⊆ \ov{W} ⊆ V ⊆ \ov{V} ⊆ U_n
\;\; , \;\;
f|_V \cong 1 \;\; , \;\; f|_{M \backslash U_n} \cong 0 \ .
\]
Let $h₁^{\bullet}$ be the ambient isotopy of $M$
generated by the vector field $f ⋅ X^t$,
where $X^t = \frac{dh^t}{dt}$ is the vector field generating $h^t$;
clearly $h₁^{\bullet}$ is supported on $U_n$.
By compactness of $\ov{W}$, there exists $δ > 0$ such that
$\ov{W}$ remains in $V$ up to time $δ$,
i.e.\ $h^t(\ov{W}) ⊆ V$ for $t ∈ [0,δ]$.
In particular, $h^t ∘ (h₁^t)^\inv$ is the identity on $h^t(W)$.
By compactness of $\ov{Z}$, there exists $δ' > 0$ such that
$\ov{Z} ⊆ h^t(W)$ for $t ∈ [0,δ']$, so that
$h^t ∘ (h₁^t)^\inv$ is constant on $\ov{Z}$ for $t ∈ [0,δ']$.
Applying the inductive hypothesis to $h^t ∘ (h₁^t)^\inv$
(restricted to $t ∈ [0,δ']$)
which is constant on $A ∪ \ov{Z}$,
with open cover $\{U₁ ∪ Z, U₂, ..., U_{n-1}\}$,
we obtain moves $h₂^{\bullet},...,h_k^{\bullet}$ supported in
the above open cover connecting $\id_M$ to $h¹ ∘ (h₁¹)^\inv$.
Since these isotopies are fixed on $A ∪ \ov{Z}$,
they may also be considered as supported by the given open cover $\{U_i\}$.
Concatenating with $h₁^{\bullet}$ (restricted to $t ∈ [0,δ']$),
we have moves supported by $\{U_i\}$ from $\id_M$ to $h^{δ'}$
that are constant on $A$.

In short, we have shown that we can always
``make nonzero forward progress by moves''.
This applies to any starting time $t₀ < 1$,
i.e. we apply the above argument to $h^{t₀ + t} ∘ (h^{t₀})^\inv$,
so that we get moves connecting $h^{t₀}$ to $h^{t₀ + δ}$.
Similarly, applying to the reverse isotopy $(h^{1-t} ∘ (h¹)^\inv)$,
we see that we can also always
``make nonzero backward progress by moves''.
In other words, for each $t ∈ [0,1]$,
there is an interval $[t - δ(t), t + δ'(t)]$, open as a subset of $[0,1]$,
such that there are moves supported by $\{U_i\}$
from $h^{t₁}$ to $h^{t₂}$ for any $t₁,t₂ ∈ [t - δ(t), t + δ'(t)]$.
Then by compactness of $[0,1]$, we are done.
\end{proof}

We want to apply Lemma~\ref{l:isotopy-moves} to an isotopy supported on a
compact submanifold $W ⊆ M$, but we cannot directly apply it to $W$,
since the resulting moves $h_j^{\bullet}$ are not guaranteed to be
$C^{∞}$-constant at $∂W$, and hence may not extend to the ambient manifold $M$.
However, we really only need the moves to restrict to
a graph isotopy $h_j^{\bullet}(Γ)$ when applied to a graph $Γ ⊆ M$;
in particular, we only need the moves to be $C^{∞}$-constant at $Γ ∩ ∂W$,
which we generally want to be a transversal intersection,
and hence a finite number of points.
To this end, we give the following refinement to the proposition above:

\begin{lemma}
\label{l:cover-C-infty}
Let $M$ be a compact manifold with smooth boundary,
and let $B ⊆ ∂M$ be a finite set on the boundary.
Let $h^{\bullet}: M \to M$ be an isotopy of $M$
that is $C^{∞}$-constant on $B$,
and let $\{U_i\}$ be a finite open cover of $M$.
Then there exist moves supported in $\{U_i\}$
from $\id_M$ to $h¹$ that are also $C^{∞}$-constant on $B$.
\end{lemma}
\begin{proof}
For simplicity, we assume $B = \{b\}$, with the general case
following the same arguments.
Label the open cover as $\{ U_i \}_{i=1,\dots,n}$.
Consider a small open neighbourhood $W_b$ of $b$,
such that $W_b$ is contained in each open set $U_i$ that contains $b$.
We refine the open cover by adding a small open neighbourhood of $b$
and carving out a neighbourhood of $b$ from each $U_i$
so that exactly one open set contains $b$.
More precisely, let $U_i' = U_i \backslash \ov{V_b}$, $i=1,\dots,n$,
where $V_b$ is an open neighbourhood of $b$ such that $\ov{V_b} ⊆ W_b$.
Then we consider the open cover $\{U₀' = W_b, U₁',\dots,U_n' \}$,
which is a refinement of the given open cover.
We apply Lemma~\ref{l:isotopy-moves} 
to get the isotopies $h_j^{\bullet}$,
each supported in $U_{i_j}'$ and fixed (but not necessarily $C^{∞}$-constant) at $b$.

Choose a small compact neighbourhood $C'$ of $b$
such that $C'$ remains in $V_b$ under all possible compositions of the $h_j^t$'s,
i.e.\ $h_k^{t_k} ∘ \cdots ∘ h₁^{t₁}(C') ⊆ V_b$ for all $t_i ∈ [0,1]$,
and choose an even smaller compact neighbourhood $C''$ of $b$
that remains in $C'$ under all compositions of the $h_j^t$'s in the same sense.
For each $j$ for which $U_{i_j}' = U₀'$,
consider the isotopy $ψ_j^{\bullet}$ of the compact set
$C = C'' ∪ \ov{M \backslash C'}$
which restricts to $h_j^{\bullet}$ on $C''$ and $\id$ on $\ov{M \backslash C'}$.
We can extend $ψ_j^{\bullet}$ to a neighbourhood of $C$
using $h_j^{\bullet}$ on a neighbourhood of $\ov{M \backslash C'}$
and $\id$ on a neighbourhood of $C''$,
and it is clear that the image of the track of this extension is open,
so by Proposition~\ref{p:isotopy-extension-2},
$ψ_j^{\bullet}$ extends to $M$.
Set $g_j^t := (ψ_j^t)^\inv ∘ h_j^t$;
for other $j$, i.e.\ $U_{i_j}' ≠ U₀'$, we simply set $g_j^t := h_j^t$.

By construction, $g_j^{\bullet}$ is $C^{∞}$-constant at $b$ for each $j$,
and by choice of $C$, the effects of $ψ_j^{\bullet}$ do not propagate
into other $U_i'$, in particular,
$g_k^t ∘ \cdots ∘ g₁^t$ and $h^t$ agree on each $U_i' ≠ U₀'$.
Thus, we define $g_{k+1}^t = h^t ∘ (g_k^t ∘ \cdots ∘ g₁^t)^\inv$,
which is supported on $U₀'$ and is $C^{∞}$-constant at $b$,
so that $g₁^{\bullet},...,g_{k+1}^{\bullet}$ are the desired moves.
\end{proof}

The following is a useful type of isotopy,
often used in the study of mapping class groups
(see \cite[Sec.\,4.2.1]{FarbMargalit:2011},
originally ``spin maps'' in \cite{Birman:1969}),
although there push maps are only defined for loops.

\begin{definition}
\label{d:push-map}
Let $ρ:[0,1] \to M$ be a regularly parametrised immersed curve in
$M \backslash ∂M$.
A \emph{push-map along $ρ$} is an isotopy $Ψ_{ρ}^t$
such that $Ψ_{ρ}^t(ρ(0)) = ρ(t)$,
and is supported in a small neighbourhood of $ρ$.

\begin{figure}
\center
\includegraphics[width=12cm]{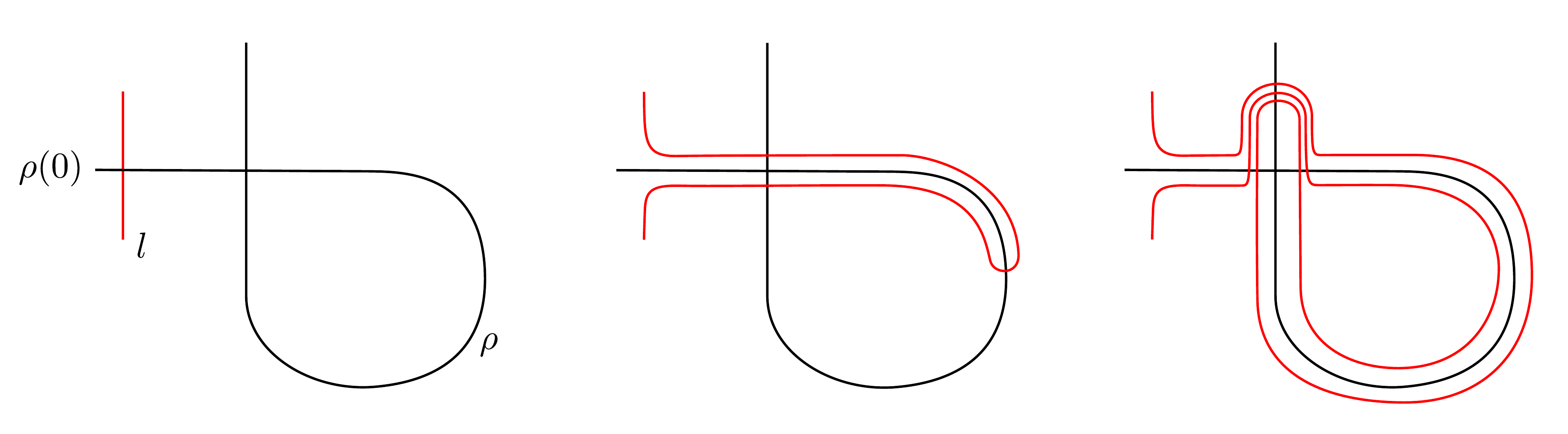}
\caption{Push map $ψ_{ρ}^t$ for immersed curve $ρ$; the segment $l$ is stretched
by $ψ_{ρ}^t$.}
\label{f:push-map}
\end{figure}
\end{definition}

Push-maps can be constructed as follows.
If $ρ$ is even an embedding, simply take a vector field $X$
that extends the tangent vector field $ρ'$ to a neighbourhood of $ρ$,
and let $Ψ_{ρ}^t$ be the flow of $X$;
by construction we have $Ψ_{ρ}^t(ρ(0)) = ρ(t)$.
For an immersed curve $ρ$, just break it into finitely many pieces $ρ_i$,
each an embedded curve, and concatenate the push-maps for each $ρ_i$.
(See Figure \ref{f:push-map}.)
Note that the neighbourhood on which the vector field is supported
can be made arbitrarily small.

Recall that $Ψ_ρ^t$ is defined for all $t$,
and so in particular for small $δt > 0$,
$Ψ_ρ^{1 + δt}$ sends $ρ(0)$ to a point just past $ρ(1)$
in the direction of $ρ'(1)$;
we write $Ψ_ρ^{1+}$ to express that we have implicitly chosen
some small $δt > 0$, which may need to satisfy some additional conditions
based on the context.

The proof of Lemma \ref{l:isotopy-null}
is based on the idea that a graph isotopy can be broken into
an ambient isotopy plus tiny isotopies fixing the edges at vertices:

\begin{lemma}
\label{l:ambient-graph}
Let $Γ,Γ'$ be related by a graph isotopy $Γ^{\bullet}$ in $M$
that is $C^{∞}$-constant on $∂Γ$, so that $Γ = Γ⁰, Γ' = Γ¹$.
For each internal vertex $v ∈ Γ⁰$, let $B_v$ be a closed ball
containing $v$ in its interior and does not meet $∂M$.
Then there exists a graph $\hat{Γ}$ and an ambient isotopy $φ^{\bullet}$ on $M$
such that $Γ$ and $\hat{Γ}$ are related by a graph isotopy that is
constant outside the union $\bigcup B_v$ over all internal vertices,
and $φ^{\bullet}$ takes $\hat{Γ}$ to $Γ'$.
\end{lemma}

\begin{proof}
Let us reduce to the case where $Γ^{\bullet}$ is fixed on internal vertices.
For each internal vertex $v$ of $Γ = Γ⁰$,
let $v^{\bullet}$ denote the restriction of the graph isotopy to $v$.
By Proposition~\ref{p:isotopy-extension-submfld},
there exists an ambient isotopy $H^{\bullet}$ of $M$ that restricts to
$v^{\bullet}$ for each internal vertex $v$.
Thus $\wdtld{Γ}^{\bullet} := ((H^t)^\inv (Γ^t))_{t ∈ [0,1]}$
is a graph isotopy from $Γ$ to $(H¹)^\inv(Γ')$ that is fixed at internal vertices,
and $\wdtld{H}^t = (H^{1-t})^\inv ∘ H¹$
defines an ambient isotopy $\wdtld{H}^{\bullet}$
that takes $(H¹)^\inv(Γ¹)$ to $Γ¹ = Γ'$.
Thus it suffices to work with $\wdtld{Γ}^{\bullet}$.

Next we construct an ambient isotopy $G^{\bullet}$ such that
$(G^t)^\inv(\wdtld{Γ}^t)$ is constant outside $\bigcup B_v$.
Let $A = Γ \backslash \bigcup D_v$ consist of the edges of $Γ$
truncated at the internal vertices, where $\ov{D_v} ⊂ \Int(B_v)$.
We can extend the isotopy $A^{\bullet}$ (the restriction of $Γ^{\bullet}$ to $A$)
to a neighbourhood $U$ of $A$, as follows.
Fix a Riemannian metric on $M \backslash \{\text{corners}\}$,
and consider a smoothly varying orthonormal frame of $TM$ along $A$,
i.e. orthonormal frames $Ξ^t(a^t)$ at each point $a^t ∈ A^t$
that varies smoothly with $a$ and $t$.
Then using geodesics, we get an embedding $ι^t: U \to M$,
where $U$ is a small enough neighbourhood of the 0-section in $TM|_A$,
that depends smoothly on $t$, i.e.\ $ι^{\bullet}$ is an isotopy of $ι(U)$.
Then by Proposition~\ref{p:isotopy-extension-2},
$ι^{\bullet}$ can be extended to an ambient isotopy $G^{\bullet}$.
Thus, $\hat{Γ}^{\bullet} := ((G^t)^\inv(\wdtld{Γ}^t))$
is a graph isotopy from $Γ$ to $\hat{Γ} := (G¹)^\inv(\wdtld{Γ}¹)$
that is fixed outside of $\bigcup B_v$,
and $G^{\bullet}$ is an ambient isotopy that takes $\hat{Γ}$ to $\wdtld{Γ}¹$.
\end{proof}

\begin{lemma}
\label{l:isotopy-null-ambient}
Let $\{U_i\}$ be a finite open cover of $M$.
Let $Γ ∈ \Graph(M;\XX)$ be an $\cA$-coloured graph.
If $h^{\bullet}$ is an ambient isotopy that is $C^{∞}$-constant at $∂Γ$,
then we can write $h¹(Γ) - Γ = \sum_{j=1}^k Γ^{(j)} - Γ^{(j-1)}$,
where $Γ^{(0)} = Γ, Γ^{(k)} = h¹(Γ)$,
and for each $j$, there exists a smoothly embedded closed ball $D_j$
with $D_j ⊆ U_{i_j}$ for some $U_{i_j}$ from the open cover,
such that $Γ^{(j-1)}$ and $Γ^{(j)}$ are transverse to $∂D_j$,
and related by a graph isotopy that is constant outside $D_j$ and
$C^{∞}$-constant at $∂Γ$.
In particular, $h¹(Γ) - Γ ∈ \Null(M;\XX)$.

Moreover, if $Γ ∈ \Graph(M;\XX)$ is an $\cS$-admissible graph,
then we can choose the $Γ^{(j)}$'s and $D_j$'s
so that $Γ^{(j)} \backslash D_j$ is $\cS$-admissible.
In particular, $h¹(Γ) - Γ ∈ \NullX(M;\XX)$.
\end{lemma}

\begin{proof}
We first assume that $M$ is compact and has smooth boundary (no corners),
and also put aside $\cS$-admissibility.
Consider the collection $\{D_{α}\}$ of
all closed smoothly embedded balls $D_{α}$
such that balls that touch the boundary are ``squashed against the boundary''
appearing in a standard form in some chart (see Figure \ref{f:squashed-ball}).
Importantly, the intersection of such $D_{α}$ with $∂\bH^d$,
the boundary of the half-space,
is also a closed ball $V_{α}$, but of dimension $d-1$;
in particular, $E_{α} := \Int(D_{α}) ∪ \Int(V_{α})$ is an open set,
and $\{E_{α}\}$ is an open cover of $M$.

\begin{figure}
\center
\includegraphics[width=6cm]{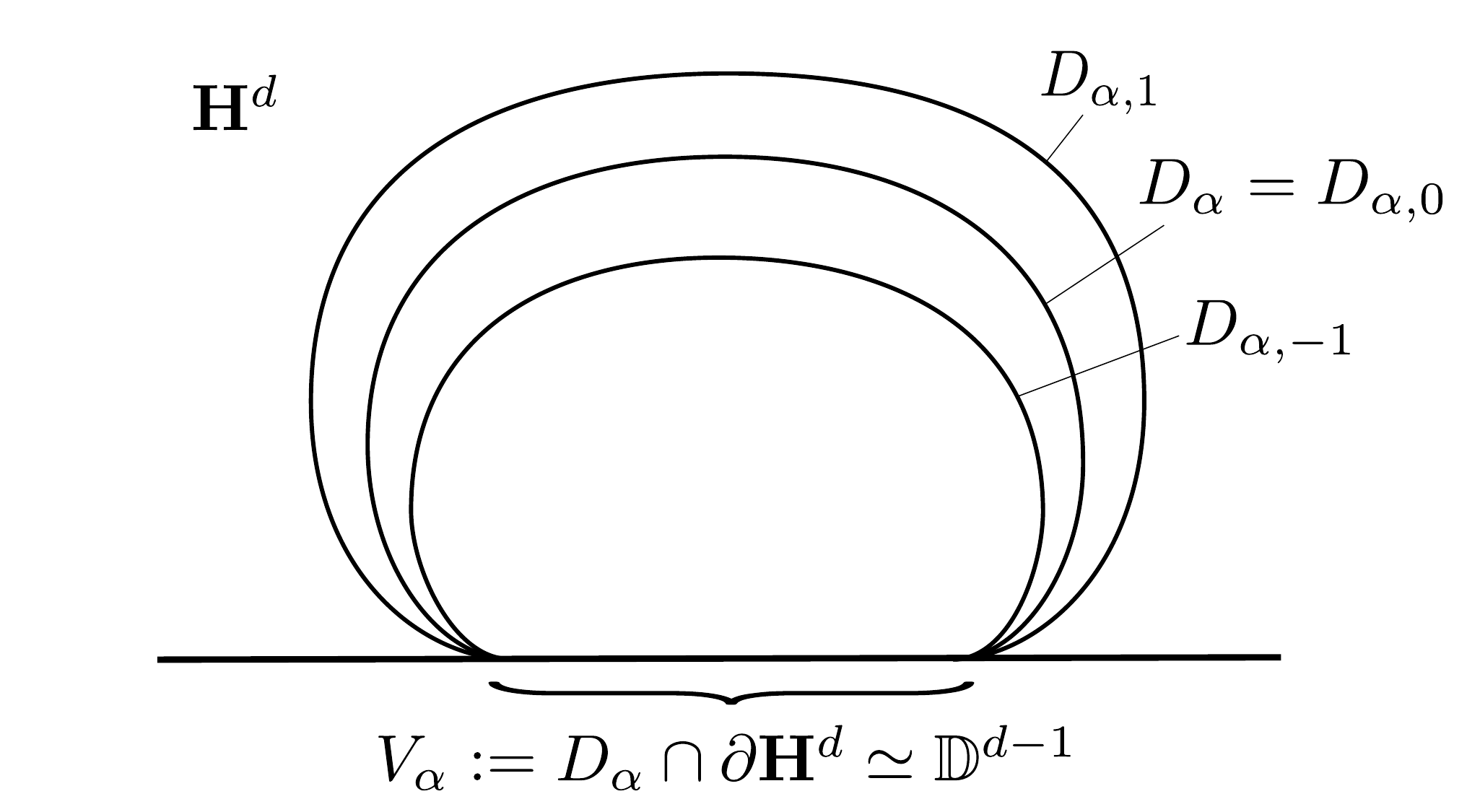}
\caption{Squashed ball $D_{α}$ and family $D_{α,τ}$,
all meeting the boundary in the same $V_{α}$;
we can construct $D_{α,τ}$ by flowing $D_{α}$ under some vector field
that is 0 on $∂M$ and always pointing outwards along $∂D_{α} ∩ \Int(M)$.
}
\label{f:squashed-ball}
\end{figure}

In the following, we have to consider intersections of graphs
with these balls, and we would like for these intersections to be
transversal at the boundary of the balls;
in order to guarantee transversality, we also consider a family of balls
around each ball, $\{D_{α,τ}\}$, $τ ∈ \RR$,
such that $D_{α,0} = D_{α}$ and
each $D_{α,τ}$ has the same (possibly empty) intersection
$D_{α,τ} ∩ ∂\bH^n = V_{α}$ with the boundary of half-space,
with the following properties: (1) any graph is transverse to $∂D_{α,τ}$
for all but countably many $τ$,
(2) $D_{α,τ'}$ contains an open neighbourhood of $D_{α,τ}$ for $τ' > τ$,
and (3) any open neighbourhood $U$ that contains $D_{α,τ}$
also contains $D_{α,τ'}$ for some $τ' > τ$
(see Figure~\ref{f:squashed-ball} for more details).

We can consider a subcollection of the balls $\{D_{α}\}$,
consisting of those $D_{α}$ contained in some $U_i$
in the given finite open cover.
It is easy to see that the subcollection of interiors $E_{α}$
of such balls still covers $M$,
hence there exists a finite subcover $\{E_{i'}\}$ of $M$.
Applying Lemma \ref{l:cover-C-infty}, with cover $\{E_{i'}\}$
and finite set of points $B = ∂Γ$,
we get a sequence of moves $(h_j^{\bullet})_{j=1,...,k}$
supported by $\{E_{i'}\}$ from $\id_M$ to $h¹$
that is $C^{∞}$-constant on $B$.

Let $Γ^{(j)} = h_j¹ ∘ \cdots ∘ h₁¹(Γ)$, so $Γ^{(0)} = Γ$ and $Γ^{(k)} = h¹(Γ)$,
and let $h_j^{\bullet}$ be supported in $D_{i_j'}$,
which is contained in some $U_{i_j}$.
If $Γ^{(j-1)}$ is transverse to $∂D_{i_j}$, take $D_j = D_{i_j}$;
otherwise, if $Γ^{(j-1)}$ is not transverse to $∂D_{i_j}$,
then we choose $τ > 0$ such that
$Γ^{(j-1)}$ is transverse to $∂D_{i_j,τ}$ and $D_{i_j',τ} ⊆ U_{i_j}$,
and we take $D_j = D_{i_j',τ}$.
Then the graph isotopy $h_j^{\bullet}(Γ^{(j-1)})$
which is constant outside $D_j$ and $C^{∞}$-constant at $∂Γ$.

Next we consider $\cS$-admissibility.
Let $Γ ∈ \GraphX(M;\XX)$ be an $\cS$-admissible graph.
Consider the $j$-th move, $h_j^{\bullet}$, and the graph isotopy
$h_j^{\bullet}(Γ^{(j-1)})$ from $Γ^{(j-1)}$ to $Γ^{(j)}$.
If $Γ^{(j-1)}$ has an $\cS$-edge outside $D_j$,
then we do not have to do anything. 
If not, then there must be some $\cS$-edge inside $D_j$.
Drag a bit of the $\cS$-edge out of the ball $D_j$,
that is, apply a push-map $Ψ_{γ}^{\bullet}$ along a simple path $γ$
from the $\cS$-edge to outside $D_j$, always staying inside $U_{i_j}$.
This push-map is taken to have a small enough neighbourhood $D_{γ}$ of $γ$,
that is a ball that admits an $\cS$-edge outside it, and $D_{γ} ⊆ U_{i_j}$.
Now apply $h_j^{\bullet}$ to the new graph $Ψ_{γ}¹(Γ^{(j-1)})$,
giving a graph isotopy in $D_j$, with an $\cS$-edge outside $D_j$.
Finally, we undo the effect of the push-map by applying
$(h_j¹ ∘ Ψ_{γ}^{1-t} ∘ (h_j¹)^\inv)_t$,
which is a push-map for $h_j¹(γ)$ (in reverse),
and has support $h_j¹(D_{γ}) ⊆ U_{i_j}$.

Finally, we drop the assumptions of compactness and smooth boundary on $M$.
Let $W$ be a compact submanifold of $M$
containing an open neighbourhood of $Γ$.
Let $W_t = h^t(W)$.
For each $t₀$, by compactness of $\ov{Γ}$,
there exists $δ > 0$ such that $h^t(Γ)$ remains in $W_{t₀}$ for $|t - t₀| < δ$.
By compactness of $[0,1]$, we only need finitely many such intervals
to cover $[0,1]$.
Applying the above arguments with $M = W_{t₀}$ for each interval, we are done.
\end{proof}

\begin{proof}[\textbf{Proof of Lemma \ref{l:isotopy-null}}]
For each internal vertex of $Γ$, we can choose a small enough $B_v$
such that $B_v$ is contained in an open set $U_i$ from the open cover
and $Γ$ is transverse to $∂B_v$, and in the $\cS$-admissible case,
that $Γ \backslash B_v$ is $\cS$-admissible.
Then by Lemma \ref{l:ambient-graph},
there exists a graph $\hat{Γ}$ and an ambient isotopy $φ^{\bullet}$ on $M$
such that $Γ$ and $\hat{Γ}$ are related by a graph isotopy that is
constant outside the union $\bigcup B_v$ over all internal vertices,
and $φ^{\bullet}$ takes $\hat{Γ}$ to $Γ'$.

The graph isotopy from $Γ$ to $\hat{Γ}$ can be separated into a
sequence of isotopies, each constant outside some $B_v$.
Lemma \ref{l:isotopy-null-ambient} takes care of the $φ^{\bullet}$ part.
\end{proof}

We will also need the following relative version of
Lemma~\ref{l:isotopy-null}:

\begin{lemma}
\label{l:isotopy-null-submfld}
Let $W ⊆ M$ be a smoothly embedded submanifold of $M$,
where $W$ has smooth boundary (i.e.\ no corners).\footnote{One can certainly generalise this to $W$ with corners, but we do not want to get into technicalities with embeddings of manifolds with corners.}
Let $\{U_i\}$ be a finite open cover of $W$,
and let $Γ,Γ'$ be $\cA$-coloured graphs in $M$ with boundary value $\XX$,
and suppose they are related by a graph isotopy $Γ^{\bullet}$,
with $Γ = Γ⁰, Γ' = Γ¹$, that is $C^{∞}$-constant on the boundary $∂Γ$
and constant outside $W$, and $Γ$ is transverse to $∂W$.
Then there exists a sequence of graph isotopies
$Γ_j^{\bullet}$, $j = 1,...,k$,
with $Γ = Γ₁⁰, Γ_j¹ = Γ_{j+1}⁰, Γ_k¹ = Γ'$,
that are $C^{∞}$-constant on the boundary,
each $Γ_j^{\bullet}$ is supported in a ball $D_j$
contained in some $U_{i_j}$,
and $Γ_j⁰$ is transverse to $∂D_j$.

Moreover, suppose that $Γ$ is $\cS$-admissible.
Then each $Γ_j^{\bullet}$ and $D_j$ can be chosen so that
$Γ_j⁰ \backslash D_j$ is $\cS$-admissible.
In particular, $Γ' - Γ ∈ \NullX(M,W;\XX)$.
\end{lemma}

\begin{proof}
We simply apply Lemma~\ref{l:isotopy-null} to $M = W$,
and observe that since the resulting isotopies are $C^{∞}$-constant
at the graph boundaries, they extend to outside $W$ by the identity.
One small thing to note is that in the $\cS$-admissible case,
if $Γ ∩ W$ is $\cS$-admissible, then the lemma can be applied directly,
and if $Γ ∩ W$ is not $\cS$-admissible,
then $Γ \backslash W$ must be $\cS$-admissible,
so we would apply the first part of Lemma~\ref{l:isotopy-null}
(non $\cS$-admissible case).
\end{proof}

\subsection{Proof of Lemma \ref{l:cT-s}}
\label{s:appendix-cT-s}

Let us first set up and prove Lemma~\ref{l:cT},
which is a version of Lemma~\ref{l:cT-s} without the $\cT$-admissibility
condition on the boundary value on $L$.
We set up Lemma~\ref{l:cT} in a similar manner to 
Section~\ref{s:excision-proof}. We collect the gluing operations
$π_*: \Graph(M';\XX,\VV,\VV) \to \Graph(M;\XX)$
induced by the gluing map $π: M' \to M$ into a single map
\[
π_*: \bigsqcup_{\VV ∈ \hZ(N)} \GraphX(M';\XX,\VV,\VV) \to \GraphX(M;\XX) \ .
\]
This map is an extension of the map $\pi_*$ given in Section~\ref{s:excision-proof} from $\VV ∈ \hZXp(N)$ to $\VV ∈ \hZ(N)$.
The map $\pi_*$ is almost surjective -
it is easy to see that it misses precisely those graphs in $M$ which are
not transverse to $L$, the image of $N$ (and $N'$) under $π$.
This motivates the following definitions:
\begin{align*}
\GraphXt(M;\XX)
	&:= \{ Γ ∈ \GraphX(M;\XX) \,|\, Γ \text{ is transverse to } L \}
\\
\VGraphX^{\trv}(M;\XX) &:= \Span ( \GraphXt(M;\XX) ) ⊂ \VGraphX(M;\XX)
\\
\PrimNullX^{\trv}(M,U;\XX) &:= \PrimNullX(M,U;\XX) ∩ \VGraphXt(M;\XX)
\\
\NullX^{\trv}(M,U;\XX) &:= \Span ( \PrimNullXt(M,U;\XX) )
	⊆  \NullX(M;\XX) ∩ \VGraphXt(M,U;\XX)
\end{align*}
These differ from the corresponding spaces $\GraphX^{\trvs}(M;\XX)$, etc.,  
in Section~\ref{s:excision-proof} in that here no $\cT$-admissibility is required for the intersection $Γ ∩ L$.
Note that for $Γ ∈ \GraphXt(M;\XX)$,
its intersection with $L$ defines a boundary value on $L$;
clearly, if $Γ = π_*(Γ')$ for some $Γ' ∈ \GraphX(M';\XX,\VV,\VV)$,
then this boundary value is $\VV$,
and vice versa.

We write
\begin{align*}
\GraphX^{\trv}(M;\XX, L = \VV) &= \{ Γ ∈ \GraphX^{\trv}(M;\XX) \,|\,
	Γ ∩ L = \VV \}
\\
\VGraphX^{\trv}(M;\XX, L = \VV) &= \Span(\GraphX^{\trv}(M;\XX,L = \VV))
\end{align*}

Recall from Section~\ref{s:excision-proof} that we defined
neighbourhoods $U_{1m}$ of $L_m$ and ambient isotopies $θ_m^{\bullet}$
supported in $U_{1m}$ that performed ``translations across $L_m$''.
We define $Θ_m^\trv$ to be the subspace of $\VGraphXt(M;\XX)$
spanned by elements of the form $Γ- θ_m^{α}(Γ)$,
where both $Γ,θ_m^{α}(Γ)$ are transverse to $L$,
and let $Θ^\trv = \sum_{L_m ⊆ L} Θ_m^\trv$;
these differ from their counterparts $Θ_m^\trvs, Θ^\trvs$
in Section~\ref{s:excision-proof} in that no $\cT$-admissibility condition
is placed on the intersection of $Γ$ and $θ_m^{α}(Γ)$ with $L_m$.

\begin{lemma}
\label{l:cT}
$\NullXt(M;\XX) = \NullXt(M,U₀;\XX) + Θ^\trv$.
\end{lemma}

\begin{proof}
Let $\sum a_j Γ_j ∈ \PrimNullXt$ be primitive $\cS$-null
with respect to the closed ball $D$,
with each summand $Γ_j$ transverse to $L$.
Using isotopies supported in balls touching $\XX$,
we can bring each $Γ_j$ to agree near their boundary
(equal on the nose, not just equal $k$-jets for all $k$).
Since $\XX$ does not meet $L$,
the balls can be chosen to avoid $L$,
so by Lemma \ref{l:isotopy-null-submfld},
the resulting graphs differ from their original by
an element of $\NullXt(M,U₀;\XX)$.
Thus, we can safely assume that the graphs $Γ_j$ all
agree in a neighbourhood of $∂M$,
and hence we may choose $D$ so that it is not touching $∂M$.

\begin{figure}
\center
\includegraphics[width=10cm]{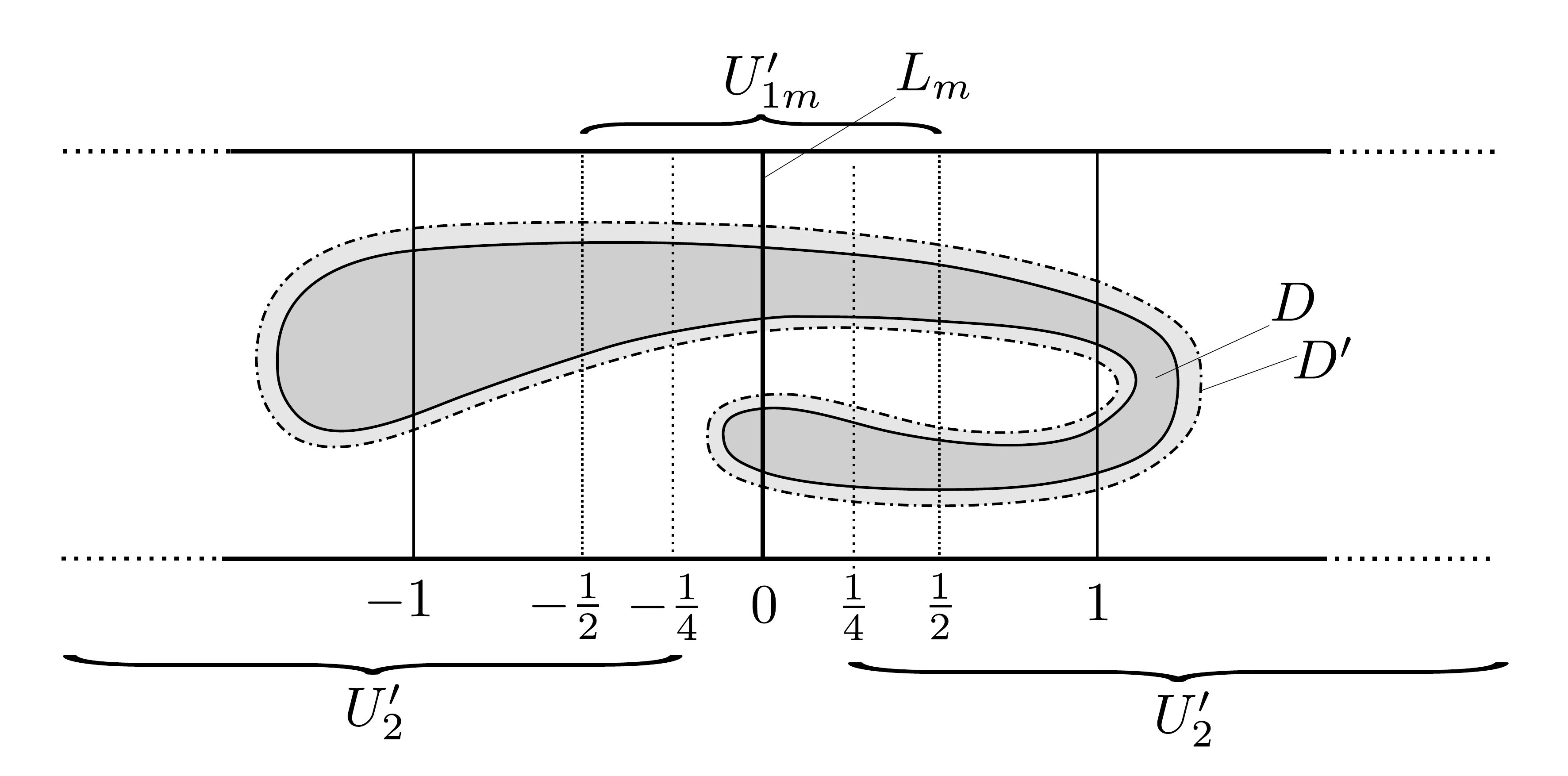}
\caption{$D$ and $D'$ can intersect $L_m$ multiple times}
\label{f:Lm-U1m-Dprime}
\end{figure}

Let $φ^{\bullet}$ be an isotopy shrinking $D$ into itself and into $U₀$,
which is supported in a small neighbourhood $D'$ of $D$,
where $D'$ is also a closed ball disjoint from $∂M$
(see Figure~\ref{f:Lm-U1m-Dprime};
e.g.\ take an isotopy of the standard ball $\DD^d$
that is $C^{∞}$-constant on $∂\DD^d$ and
shrinks $1/2 \DD^d$ radially to almost a point,
and apply that to $D'$ with a embedding sending $\DD^d$ onto $D'$
and $1/2 \DD^d$ onto $D$.)
We may choose $φ^{\bullet}$ so that each $φ¹(Γ_j)$ is also transverse to $L$.
Then $\sum a_j φ¹(Γ_j)$ is primitive $\cS$-null with respect to $φ¹(D) ⊆ U₀$.
It remains to show that $φ¹(Γ_j) - Γ_j ∈ \NullXt(M,U₀;\XX) + Θ^\trv$
for each $j$.
For the rest of the proof, we focus on one summand,
simply rewriting $Γ = Γ_j$.

Consider the open cover $\{U₀', U_{11}',...,U_{1l}'\}$,
where $U_{1m}' = ι((-1/2,1/2) \times L_m)$ are smaller
tubular neighbourhoods of each component $L_m$ of $L$
and $U₀' = M \backslash ι([-1/4,1/4] \times L)$.
Let $\wdtld{Γ} = Γ ∩ D'$ be its restriction to $D'$,
and consider $φ^{\bullet}(\wdtld{Γ})$ as a graph isotopy in $D'$.
Let $V₀ = D' ∩ U₀', V₁ = D' ∩ U_{11}', ..., V_l = D' ∩ U_{1l}'$.
Applying Lemma \ref{l:isotopy-null-submfld} to $D'$ with open cover $\{V₀,...,V_l\}$,
we obtain graph isotopies $\wdtld{Γ}_a^{\bullet}$, $a = 1,...,k$,
with each $\wdtld{Γ}_a^{\bullet}$ supported in a ball $D_a$
which is contained in some $V_{m_a}$,
$\wdtld{Γ}_a^{\bullet}$ is always transverse to $∂D_a$,
and each $\wdtld{Γ}_a^{\bullet}$ is $C^{∞}$-constant at $∂\wdtld{Γ}$,
so we may extend them by the constant isotopy outside $D'$,
so we have graph isotopies $Γ_a^{\bullet}$
that connect $Γ$ to $φ¹(Γ)$.
(See Figure \ref{f:isotopy-cover}.)

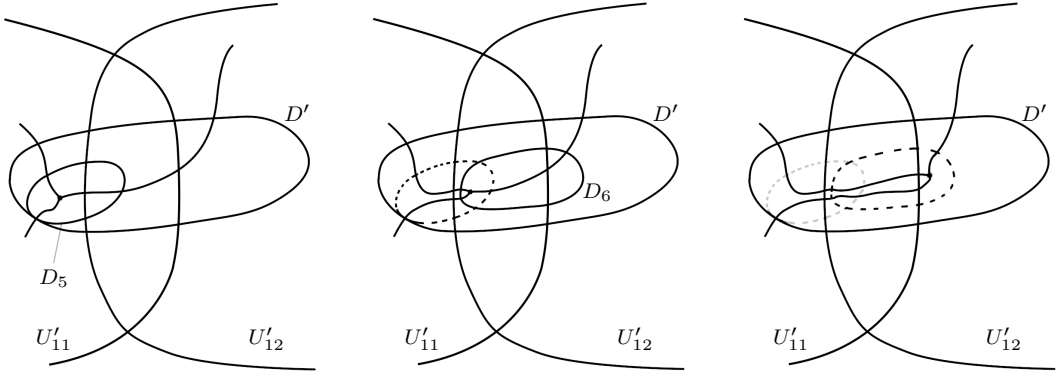
\begin{figure}
\center
\begin{tikzpicture}
\node at (0,0) {\includegraphics[width=15cm]{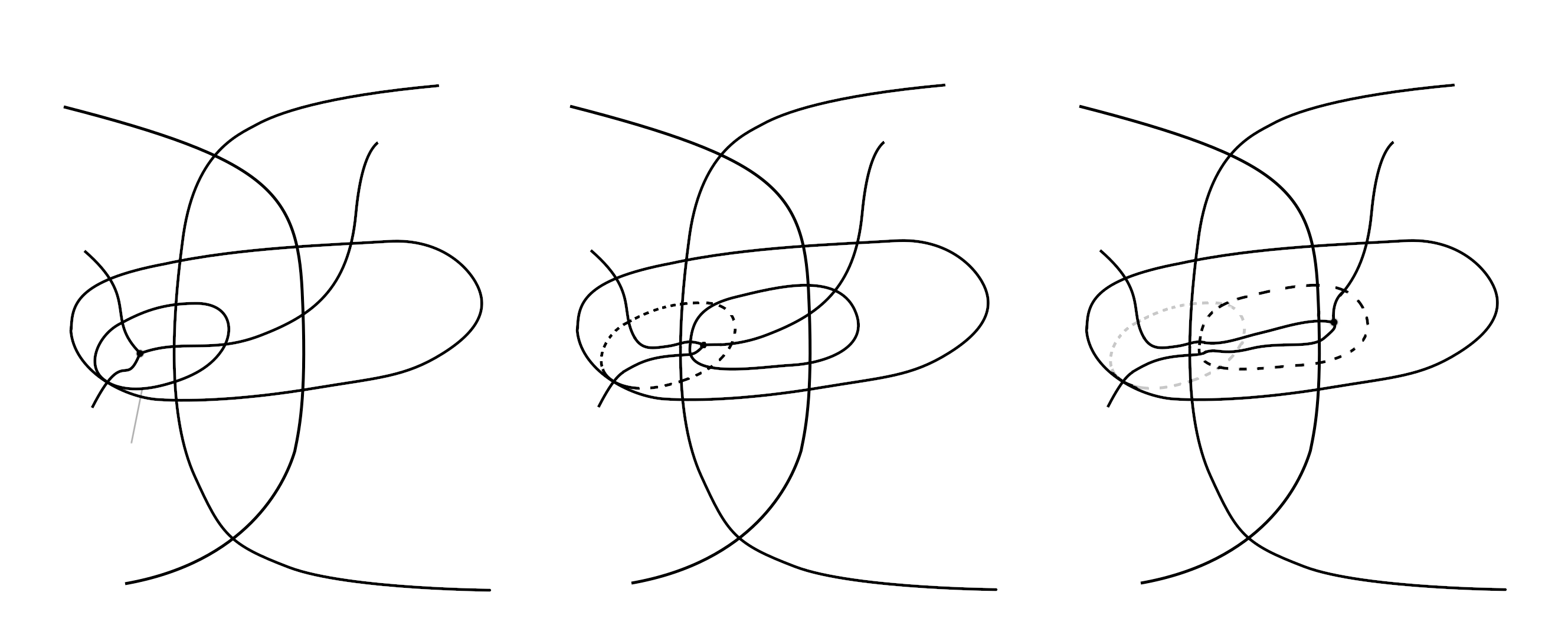}};
\begin{scope}[shift={(0,0)}]
\node at (-1.4,-2.2) {\scriptsize $U_{11}'$};
\node at (1.4,-2.2) {\scriptsize $U_{12}'$};
\node at (1.8,0.8) {\scriptsize $D'$};
\node at (0.9,-0.25) {\scriptsize $D_6$};
\end{scope}
\begin{scope}[shift={(-4.85,0)}]
\node at (-1.4,-2.2) {\scriptsize $U_{11}'$};
\node at (1.4,-2.2) {\scriptsize $U_{12}'$};
\node at (1.8,0.8) {\scriptsize $D'$};
\node at (-1.4,-1.4) {\scriptsize $D_5$};
\end{scope}
\begin{scope}[shift={(4.85,0)}]
\node at (-1.4,-2.2) {\scriptsize $U_{11}'$};
\node at (1.4,-2.2) {\scriptsize $U_{12}'$};
\node at (1.8,0.8) {\scriptsize $D'$};
\end{scope}
\end{tikzpicture}
\caption{A concatenation of graph isotopies supported in balls
(here $D_5$ and $D_6$),
each ball contained in some $U_{1m}'$.
}
\label{f:isotopy-cover}
\end{figure}

We want $Γ_a¹$, for each $a = 1,...,k$, to be transverse to $L$;
then for $a$ such that $D_a ⊆ V₀ ⊆ U₀'$,
we have $Γ_a¹ - Γ_a⁰ ∈ \NullXt(M,U₀;\XX)$
by Lemma~\ref{l:isotopy-null-submfld},
while for other $a$, we will later show that
$Γ_a¹ - Γ_a⁰ ∈ \NullXt(M,U₀;\XX) + Θ^\trv$,
so that it immediately follows that
$φ¹(Γ) - Γ = \sum_{a = 1}^k Γ_a¹ - Γ_a⁰ ∈
\NullXt(M,U₀;\XX) + Θ^\trv$.
We perform the following corrections
to the graph isotopies $Γ_a^{\bullet}$,
applied to each $m$; in the following, fix a component $L_m$.
Let $a$ be the first integer $a$ for which $Γ_a⁰$ is not transverse to $L_m$.
(Note $Γ = Γ₁⁰$ is already assumed to be transverse to $L$).
The support $D_a$ of $Γ_a^{\bullet}$ must be in $V_m$;
otherwise $D_a ⊆ V_n$ for some $n ≠ m$,
then since $V_n ∩ L_m = ∅$, we would have that
$Γ_{a-1}¹ = Γ_a⁰$, and hence $Γ_{a-1}⁰$,
would already be not transverse to $L_m$,
contradicting the minimality of $a$.
The genericity of the transversality means that it is easy to find a
tiny perturbation of $Γ^a$, supported near $L_m$ and in $D_a$
(in particular supported outside the other open sets $V_{≠ m}$),
such that the resulting graph is indeed transverse to $L_m$;
we append this perturbation to the end of $Γ_a^{\bullet}$.
Let $b$ be the next integer such that $Γ_b^{\bullet}$
is supported in $V_m$.
The perturbation above commutes with the isotopies between the
$a$-th and $b$-th isotopy (since they have disjoint supports).
Then we modify the $Γ_b^{\bullet}$ by prepending the inverse of the perturbation.
Repeat this process until all $Γ^c$, $c = 0,1,...,k$ are transverse to $L_m$.

It remains to show that $Γ_a¹ - Γ_a⁰ ∈ \NullXt(M,U₀;\XX) + Θ^\trv$
for $a$ such that $D_a ⊆ V_{m_a} ⊆ U_{1m_a}'$ with $m_a ≠ 0$.
There exists $α$ such that $θ_m^{α}$ displaces $D_a$ away from
$[-1/4,1/4] \times L_m$, so that $θ_m^{α}(D_a) ⊆ U₀'$,
and that $θ_m^{α}(Γ_a⁰)$ (and hence also $θ_m^{α}(Γ_a^t)$ for $t ∈ [0,1]$)
is transverse to $L_m$.
(See Figure \ref{f:theta-displace}.)
So we replace the isotopy by the concatenation of the following three isotopies:
(1) the graph isotopy induced by the ambient isotopy $θ_m^{tα}$,
(2) then the graph isotopy $θ_m^{α}(Γ_a^{\bullet})$,
and (3) the reverse of the first isotopy, $θ_m^{-tα}$.
Then we can write
\[
Γ_a¹ - Γ_a⁰ =
(Γ_a¹ - θ_m^{α}(Γ_a¹)) +
(θ_m^{α}(Γ_a¹)) - θ_m^{α}(Γ_a⁰)) -
(θ_m^{α}(Γ_a⁰) - Γ_a⁰)
\]
so that the first and third term are in $Θ^\trv$,
and the second term is in $\NullXt(M,U₀;\XX)$.

\begin{figure}
\center
\includegraphics[width=15cm]{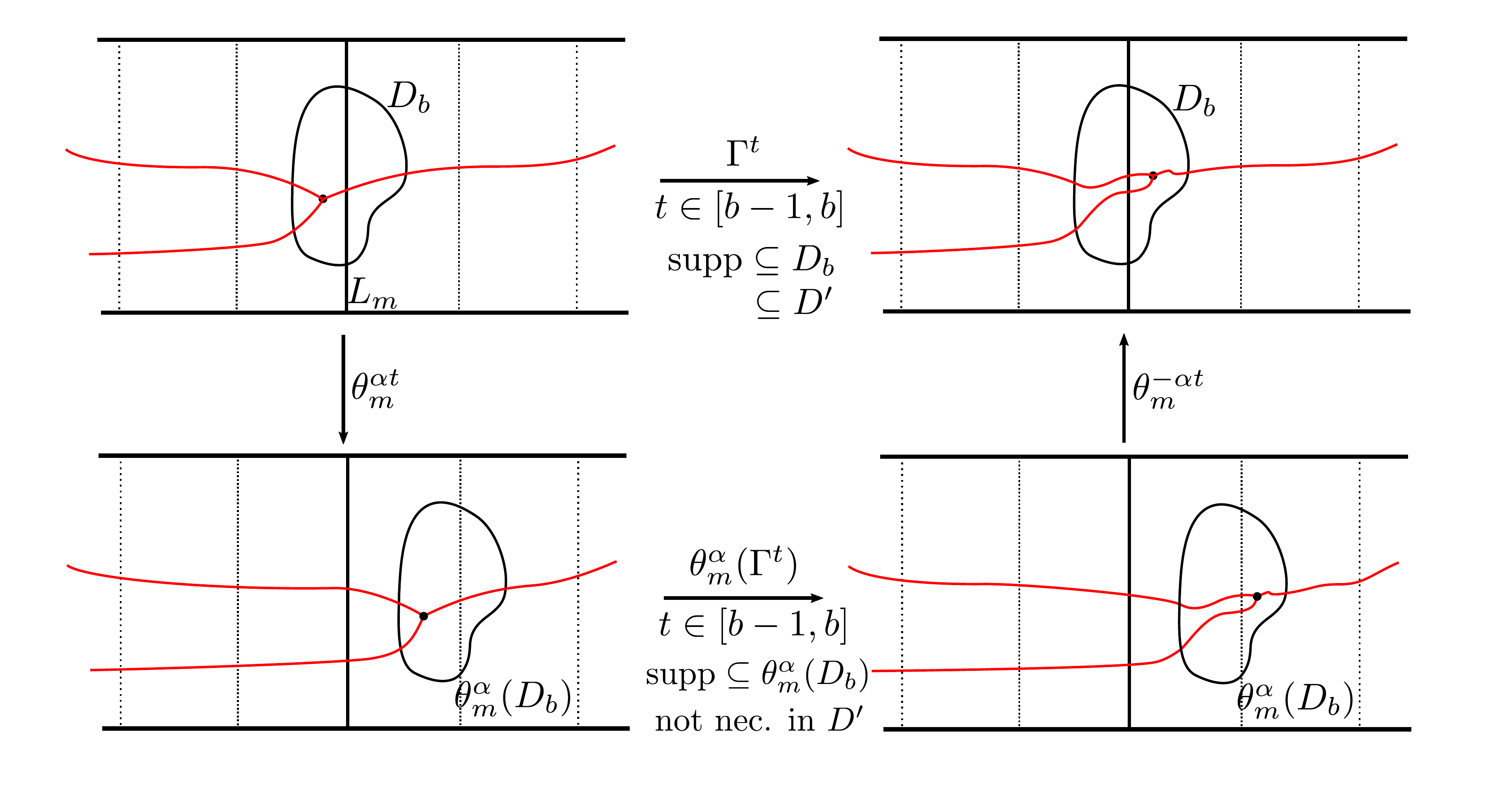}
\caption{Using $θ_m^{α}$ to displace $D_b$}
\label{f:theta-displace}
\end{figure}

This concludes the proof of the lemma, but before we end the proof,
in anticipation of reapplying most of the arguments above
for the proof of Lemma \ref{l:cT-s},
we make the following observation:
suppose there is a region $K ⊆ L$
such that $\wdtld{K} := [-1,1] \times K$ is disjoint from $D'$.
Then clearly none of the balls $D_a$ that supported the graph isotopies
ever intersect $\wdtld{K}$,
and since $\wdtld{K}$ is invariant under translation,
i.e. $θ_m^{α}(\wdtld{K}) = \wdtld{K}$,
their displaced versions $θ_m^{α}(D_b)$ also do not intersect $\wdtld{K}$.
In other words, throughout the isotopies $Γ_a^{\bullet}$,
the only changes that can happen to the graph within $\wdtld{K}$
is translation by some $θ_m^t$.
\end{proof}

Similarly, recall from Section~\ref{s:excision-proof} that
$Θ_m^\trvs$ is the subspace of $\VGraphXt(M;\XX)$
spanned by elements of the form $Γ- θ_m^{α}(Γ)$,
where both $Γ,θ_m^{α}(Γ)$ are transverse to $L$
and their intersection boundary values on $L$ are $\cT$-admissible,
and $Θ^\trvs := \sum_{L_m ⊆ L} Θ_m^\trvs$.

\begin{proof}[Proof of Lemma \ref{l:cT-s}]
Let $\sum a_j Γ_j ∈ \PrimNullXs$
be primitive $\cS$-null with respect to $D$,
with each $Γ_j$ transverse to $L$ and having $\cS$-admissible
boundary value on $L$.
As before, we may assume $D$ does not meet the boundary $∂M$,
and choose a slightly bigger ball $D'$ containing $D$ in its interior,
which is also disjoint from $∂M$,
and consider an isotopy $φ^{\bullet}$ supported in $D'$
that shrinks $D$ into $D ∩ U₀$,
and such that for each $j$, $φ(Γ_j)$ is transverse to and has
$\cT$-admissible boundary value on $L$.
Then $\sum a_j φ¹(Γ_j) ∈ \NullXs(M,U₀;\XX)$,
and it suffices to show that
$φ¹(Γ_j) - Γ_j ∈ \NullXs(M,U₀;\XX) + Θ^\trvs$ for each $j$.
In the following, we describe some modification to the $Γ_j$'s
that is performed by some ambient isotopies,
whose construction is based on choices depending only on
the part of the graph outside $D$,
thus we may ignore the subscript $j$ of $Γ_j$ and simply write $Γ$ as before.

From the proof of Lemma~\ref{l:cT},
there are graph isotopies $Γ_a^{\bullet}$ of two kinds,
(1) supported in $U₀$ and (2) induced by ambient isotopy $θ_m^{\bullet}$,
that connect $Γ$ to $φ¹(Γ)$,
and such that the endpoints of these isotopies, $Γ_a⁰$ and $Γ_a¹$,
are transverse to $L$.
We would like to modify these isotopies so that they maintain
their $\cT$-admissible intersection with $L$ at their endpoints.
We make use of the observation made at the end of the proof
of Lemma \ref{l:cT} as follows.
Let $K ⊆ L$ be a closed region of $L$,
such that $\wdtld{K} := [-1,1] \times K$ is disjoint from $D'$,
and $K_m := K ∩ L_m ≠ ∅ $ for each component $L_m$ of $L$.
Suppose each $K_m$ contains a segment of an $\cS$-edge of $Γ$
such that the projection of the segment to the $[-1,1]$ factor
has no critical values.
Then the argument in the proof of Lemma \ref{l:cT} works verbatim,
as this $\cS$-edge (hence also a $\cT$-edge) will always be transverse to $L$,
so that whenever the graph isotopy $Γ_a^{\bullet}$
is transverse to $L$ (in particular at the endpoints $Γ_a⁰$ and $Γ_a¹$),
its intersection boundary value on $L$ is also $\cT$-admissible.

\begin{figure}
\includegraphics[width=16cm]{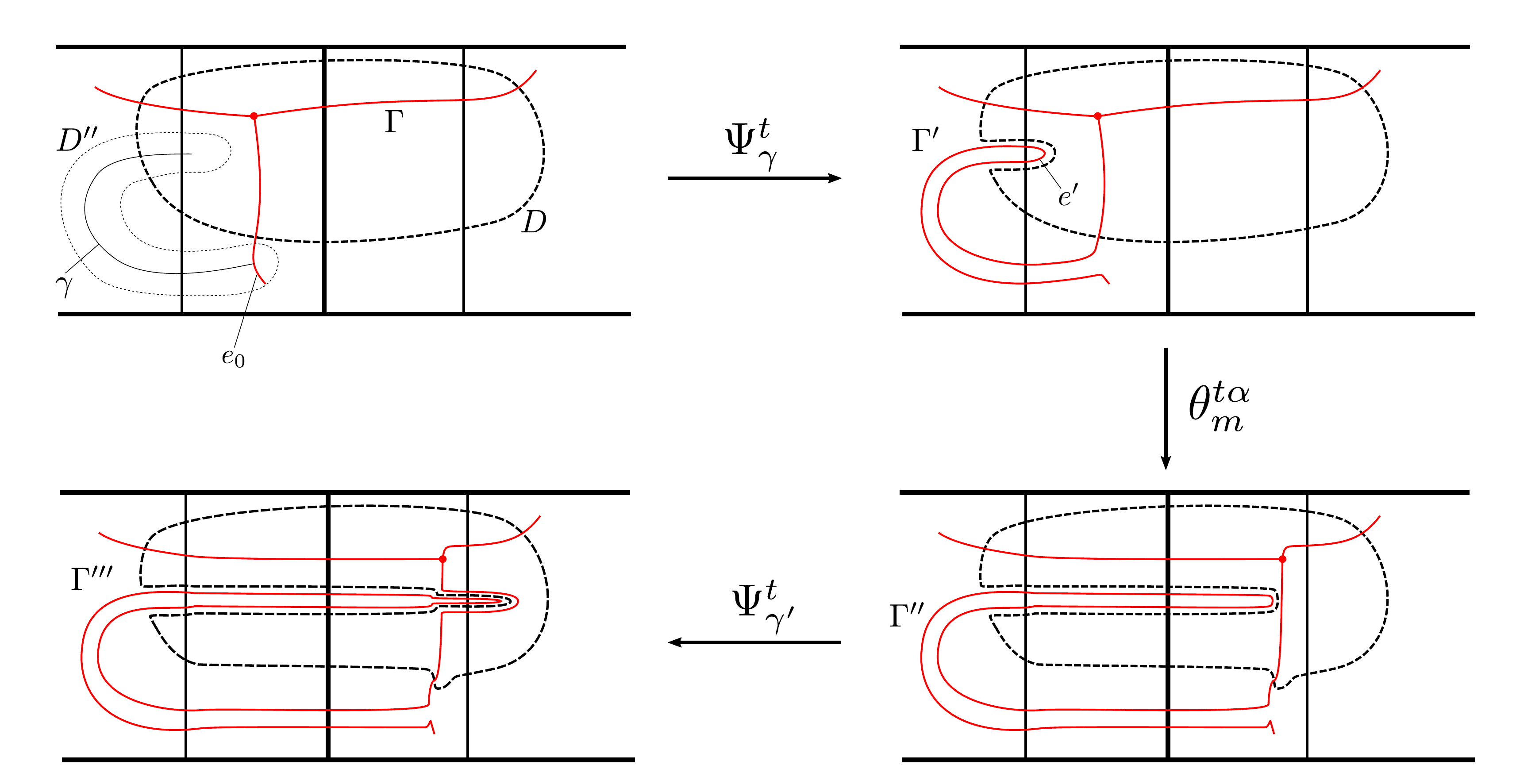}
\caption{Modifying $Γ$ to have a ``straight segment'' in some $\wdtld{K}$.
The idea is to drag an $\cS$-edge outside $D$ through $[-1,1] \times L_m$,
which accomplished in three steps: first with a push-map in $U₀$,
then with a translation to carry it through $L_m$,
then finally another push-map in $U₀$ to drag the edge out the other side.
Using a translation $θ_m^{α}$ with a large $α$ stretches the edge
till it is almost straight, so we can take $\wdtld{K}$ to be a cylinder
neighbourhood around it.
}
\label{f:translation-T-adm}
\end{figure}

\begin{figure}
\includegraphics[width=16cm]{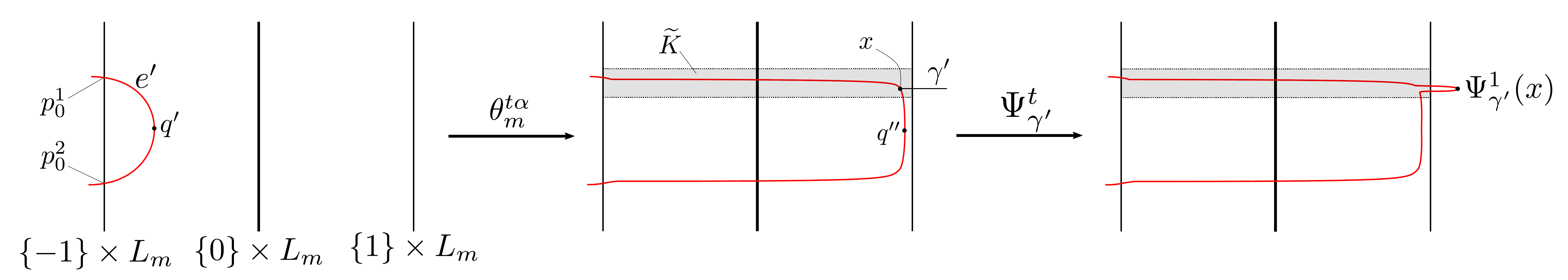}
\caption{Focus around $e'$}
\label{f:translation-T-adm-focus}
\end{figure}

It remains to show that we can perform such a modification
with relations from $\NullXs(U₀) + \cT^\trvs$.
This is shown in Figure \ref{f:translation-T-adm},
with a quick summary in the caption;
what follows is simply excruciating details.

Fix a component $L_m$ of $L$.
Let $p₀ ∈  \Int L_m ⊆ M$ and $p = (-1,p₀) ∈ ∂([-1,1] \times L_m) ⊆ M$.
Since $Γ \backslash D$ is $\cS$-admissible,
there is a point $q$ on an $\cS$-edge $e$ of $Γ$ that lies outside $D$.
Let $γ$ be an embedded regular path from $q$ to $p$
(note $γ$ does not need to be completely disjoint from $D$,
in particular, $p$ may be in $D$),
and such that $γ$ is not tangent to $Γ$ at $q$.

Let us suppose for now that the path $γ$ can be chosen such that
it does not intersect $L$ (every component, not just $L_m$),
and at the end, $γ$ is transverse to $\{-1\} \times L_m$ at $p$
and is heading into $[-1,1] \times L_m$;
extend $γ$ slightly at the end so that it ends at a point inside $[-1,1] \times L_m$.
Choose a ball $D''$ around $γ$ that is small enough
so that it does not intersect $L$,
and narrow enough so that the segment $e₀$ (i.e. component) of $D'' ∩ e$
that contains $p$ also lies outside $D$
(this is where non-tangency of $γ$ with $e$ is used).
We can also choose $D''$ so that $∂D''$ is transverse to $Γ$.

Apply a push-map $Ψ_{γ}^{\bullet}$ along $γ$ with support in $D''$.
Let $Γ' = Ψ_{γ}¹(Γ)$, and $q' = Ψ_{γ}¹(q)$
be the results of applying the push-map to $Γ$ and $q$.
Let $e' \ni q'$ be the tiny segment of the $\cS$-edge of $Γ'$
that is in $[-1,1] \times L_m$ and contains $q'$
(i.e. $e'$ is the connected component of $Γ' ∩ [-1,1] \times L_m$
containing $q'$), in particular $e' ⊆ ψ_{γ}¹(e₀)$;
choosing the push-map and the end of $γ$ more carefully,
we can ensure that $e'$ is transverse to $\{-1\} \times L_m$.
Note that $e'$ is disjoint from $Ψ_{γ}¹(D)$,
since $e' ⊆ Ψ_{γ}¹(e₀)$ and we had $e₀$ disjoint from $D$.
Let $p¹ = (-1,p₀¹),p² = (-1,p₀²) ∈ \{-1\} \times L_m$ be the endpoints of $e'$.

Since the push-map $Ψ_{γ}$ is supported in a ball $D'' ⊆ U₀$,
by Lemma \ref{l:isotopy-null-submfld} applied to $W = D''$,
we have $Γ' - Γ ∈ \NullXt(M,U₀;\XX)$,
and furthermore, since $Γ$ has $\cT$-admissible intersection with $L$,
we also have $Γ' - Γ ∈ \NullXs(M,U₀;\XX)$.

Next consider the effect of $θ_m^{α}$ on the edge $e'$:
as we increase $α$, the edge gets stretched
as its endpoints $p¹,p²$ on $\{-1\} \times L_m$ remain fixed.
For very large $α$, the resulting edge segment $e'' := θ_m^{α}(e')$
appears like two almost straight edges approximating $[-1,1] \times p₀¹$
and $[-1,1] \times p₀²$, the straight edges containing $p¹,p²$ respectively,
plus some tiny arc that goes between them
(see middle diagram in Figure \ref{f:translation-T-adm-focus}).
We choose such large $α$ so that the resulting graphs,
$Γ'' := θ_m^{α}(Γ')$, are transverse to $L$.
Since $e''$ is $\cS$-labelled (hence $\cT$-labelled) and intersects $L_m$,
then $Γ''$ has $\cT$-admissible intersection with $L$,
so $Γ'' - Γ' ∈ Θ_m^\trvs ⊆ Θ^\trvs$.
We take $K$ to be a neighbourhood of $p₀¹$ and $\wdtld{K} = [-1,1] \times K$.

Now we apply another push-map along a path $γ'$
starting from $x$ and leaving $\wdtld{K}$ through $\{1\} \times L_m$
(see rightmost diagram in Figure \ref{f:translation-T-adm-focus});
we choose $γ'$ such that it is straight inside $\wdtld{K}$
(i.e. follows $[-1,1] \times {*}$ for some point $* ∈ K$).
Similar to the first push-map, we have
$Γ''' - Γ'' ∈ \NullXs(M,U₀;\XX)$.

Finally, we need to perform this for every component of $L$,
so we need to consider the situation where $γ$ goes through
finitely many components of $L$,
say $L_{m₁},L_{m₂},...,L_{m_k}$, before reaching $L_m$,
Then we first perform the operation above to $L_{m₁}$,
so that the result graph $Γ'''$ has an $\cS$-edge $e'''$
on the other side of $L_{m₁}$ as well.
Now we repeat the argument for $Γ'''$,
with path $γ'$ obtained by cutting off the beginning of $γ$
and instead starting from this $\cS$-edge $e'''$, and so on,
eventually reaching $L_m$.

To summarise, the argument above gives ambient isotopies
$φ_b^{\bullet}$, $b = 1,...,k$,
that are either (1) supported in a ball contained in $U₀$
or (2) of the form $θ_m^{\bullet}$,
such that when restricted to each $Γ_j$,
connects $Γ_j$ to new graphs $\wdtld{Γ}_j$
with the desired ``no critical value $\cS$-edges'' in $\wdtld{K}$.
Then we apply the argument in the proof of Lemma~\ref{l:cT}
to conclude that:
\[
\sum a_j Γ_j = \sum a_j(Γ_j - \wdtld{Γ}_j) + \sum a_j \wdtld{Γ}_j
∈ \NullXs(M,U₀;\XX) + Θ^\trvs \ .
\]
\end{proof}

\subsection{Proof of Lemma \ref{l:coend-closure}}
\label{s:appendix-closure}

\begin{proof}[Proof of Lemma \ref{l:coend-closure}]
Let us write the action of $F(\id, g : X \to Y)$ on $Φ ∈ F(Z,X)$ by
$g ∘ Φ$, and the action of $F(g, \id)$ on $Φ ∈ F(Y,Z)$ by $Φ ∘ g$.
We denote by $[v]$ the image of an element $v ∈ \bigoplus F(X,X)$
in $\int^{\cC} F$, and similarly for the coend over $G$.

We first assume that $\cC$ is additive.
We prove the lemma for two special cases:
(1) $\cC = \cB^{Σ} := \{\bigoplus_{i=1}^k X_i \;|\; X_i ∈ \cB\}$,
closure by finite direct sums only, and
(2) $\cC = \cB^{rt} := \{Y \;|\; Y \text{ is retract of some } X ∈ \cB\}$,
closure by retracts only.
By composing the two inclusions $\cB ↪ \cB^{Σ} ↪ (\cB^{rt})^{Σ}$
and noting that $\rcls{\cB} = (\cB^{Σ})^{rt}$ (because $\cC$ is additive),
we see that these imply the lemma for $\cC = \rcls{\cB}$.
In both cases, we construct a map $\wdtld{ψ} : \int^{\cC} F \to \int^{\cB} G$
that is inverse to $\xi$.

For case (1), for an object $X' ∈ \cC$ that is a direct sum
of finitely many $X_j ∈ \cB$,
with inclusion/projections given by $ι_j : X_j ⇌  X' : π_j$,
we send $[v] ∈ \int^{\cC} F$, represented by $v ∈ F(X',X')$,
to $\wdtld{ψ}([v]) := \sum_j [π_j ∘ v ∘ ι_j] ∈ \int^{\cB} G$,
where $π_j ∘ v ∘ ι_j ∈ G(X_j,X_j) = F(X_j,X_j)$.
For case (2), for an object $Y ∈ \cC$ that is a retract of $X ∈ \cB$,
with inclusion/projection given by $ι : Y ⇌  X : π$,
we send $[v] ∈ \int^{\cC} F$, represented by $v ∈ F(Y,Y)$,
to $\wdtld{ψ}([v]) := [ι ∘ v ∘ π] ∈ \int^{\cB} G$,
where $ι ∘ v ∘ π ∈ G(X,X) = F(X,X)$.
It is not hard to see that, if an inverse to $\xi$ were to exist,
then it must be of this form.

It remains to prove that $\wdtld{ψ}$ is well-defined;
we show it for case (2), as case (1) is very similar.
There are two potential sources of non-well-definedness of $\wdtld{ψ}$:
the choice of $X$, $ι$, and $π$,
and the choice of representative of the element of $\int^{\cC} F$.
For the former, consider $ι' : Y ⇌ X' : π'$,
so that we have $ι' ∘ v ∘ π' ∈ G(X',X')$.
This element represents the same element as $ι ∘ v ∘ π$ in $\int^{\cB} G$,
as we have the coend relation
\[
ι' ∘ v ∘ π' - ι ∘ v ∘ π
= (ι' ∘ v ∘ π) ∘ ιπ' - ιπ' ∘ (ι' ∘ v ∘ π)
\]
from the element $ι' ∘ v ∘ π ∈ G(X,X')$ and morphism $ιπ' : X' \to X$.
For the latter, suppose we have the coend relation for $\int^{\cC} F$
defined by $Φ ∈ F(Y,Y')$ and $g : Y' \to Y$, namely $Φ ∘g - g ∘ Φ$.
Let $ι : Y ⇌ X : π$ and $ι' : Y' ⇌ X' : π'$ with $X,X' ∈ \cB$.
We have $\wdtld{ψ}([Φ ∘g]) = [ι' ∘ Φ ∘ gπ']$ and
$\wdtld{ψ}([g ∘ Φ]) = [ιg ∘ Φ ∘ π]$,
and these are the same, as we have the coend relation
\[
ι' ∘ Φ ∘ gπ' - ιg ∘ Φ ∘ π
= (ι' ∘ Φ ∘ π) ∘ ιgπ' - ιgπ' ∘ (ι' ∘ Φ ∘ π)
\]
from the element $ι' ∘ Φ ∘ π ∈ G(X,X')$ and morphism $ιgπ' : X' \to X$.

Finally, we address the situation when $\cC$ is not additive.
Consider the additive closure $\ov{\cC}$ of $\cC$,
say realised as the closure of $\cC$ in its Yoneda embedding
under direct sums. 
Clearly, the sum-retract closure of $\cC$ in $\ov{\cC}$ (in the sense of Definition~\ref{d:sum-retract-closure}) is the whole of $\ov{\cC}$.
It is easy to see that $\rcls{\cB} = \cC$
if and only if $\rcls{\cB} = \ov{\cC}$,
where the sum-retract closures are taken in $\cC$ and $\ov{\cC}$ respectively.
The functor $F : \cC^\op \times \cC \to \Rmod$
extends to a functor $\ov{F} : \ov{\cC}^\op \times \ov{\cC} \to \Rmod$.
Thus $\xi : \int^{\cB} G \simeq \int^{\ov{\cC}} \ov{F}$
and $\xi : \int^{\cC} F \simeq \int^{\ov{\cC}} \ov{F}$,
and by the naturalness of $\xi$, we are done.
\end{proof}

\newcommand\arXiv[2]      {\href{https://arXiv.org/abs/#1}{#2}}
\newcommand\doi[2]        {\href{https://dx.doi.org/#1}{#2}}

\end{document}